\documentclass[11pt]{aart}

\usepackage[letterpaper, hmargin=1in, top=1in, bottom=1.2in, footskip=0.6in]{geometry}
\usepackage{titlesec}
\titleformat{\section}[block]{\filcenter\normalfont\bfseries\large}{\thesection.}{.5em}{}\titlespacing*{\section}{0pt}{2\baselineskip}{1\baselineskip}
\titleformat{\subsection}[runin]{\normalfont\bfseries}{\thesubsection.}{.4em}{}[.]\titlespacing{\subsection}{0pt}{2ex plus .1ex minus .2ex}{.8em}
\titleformat{\subsubsection}[runin]{\normalfont\itshape}{\thesubsubsection.}{.3em}{}[.]\titlespacing{\subsubsection}{0pt}{1ex plus .1ex minus .2ex}{.5em}
\titleformat{\paragraph}[runin]{\normalfont\itshape}{\theparagraph.}{.3em}{}[.]\titlespacing{\paragraph}{0pt}{1ex plus .1ex minus .2ex}{.5em}

\usepackage[labelfont=bf,font=small,labelsep=period]{caption}
\usepackage[T1]{fontenc}
\usepackage[utf8]{inputenc}
\usepackage{microtype}

\usepackage{mlmodern}

\usepackage{amsmath}
\usepackage{amssymb}
\usepackage{amsfonts}
\usepackage{latexsym}
\usepackage{amsthm}
\usepackage{amsxtra}
\usepackage{amscd}
\usepackage{bbm}
\usepackage{mathrsfs}
\usepackage{bm}
\usepackage{mathtools}
\usepackage[dvipsnames]{xcolor}
\usepackage[textsize=footnotesize]{todonotes}
\presetkeys%
		{todonotes}%
		{color=Dandelion}{}%
\usepackage{hyperref}
\hypersetup{
    colorlinks,
    linkcolor={red!50!black},
    citecolor={blue!50!black},
    urlcolor={blue!80!black}
}
\usepackage[capitalise, nameinlink]{cleveref}

\definecolor{darkred}{rgb}{0.9,0,0.3}
\definecolor{darkblue}{rgb}{0,0.3,0.9}

\theoremstyle{plain} 
\newtheorem{theorem}{Theorem}[section]
\newtheorem*{theorem*}{Theorem}
\newtheorem{lemma}[theorem]{Lemma}
\newtheorem*{lemma*}{Lemma}
\newtheorem{corollary}[theorem]{Corollary}
\newtheorem*{corollary*}{Corollary}
\newtheorem{proposition}[theorem]{Proposition}
\newtheorem*{proposition*}{Proposition}

\newtheorem*{conjecture*}{Conjecture}

\theoremstyle{definition} 
\newtheorem{definition}[theorem]{Definition}
\newtheorem*{definition*}{Definition}
\newtheorem{example}[theorem]{Example}
\newtheorem*{example*}{Example}
\newtheorem{remark}[theorem]{Remark}
\newtheorem*{remark*}{Remark}
\newtheorem{assumption}[theorem]{Assumption}
\newtheorem*{assumption*}{Assumption}

\AddToHook{env/lemma/begin}{\crefalias{theorem}{lemma}}
\AddToHook{env/proposition/begin}{\crefalias{theorem}{proposition}}
\AddToHook{env/assumption/begin}{\crefalias{theorem}{assumption}}
\AddToHook{env/example/begin}{\crefalias{theorem}{example}}
\AddToHook{env/remark/begin}{\crefalias{theorem}{remark}}
\AddToHook{env/definition/begin}{\crefalias{theorem}{definition}}
\AddToHook{env/corollary/begin}{\crefalias{theorem}{corollary}}

\usetikzlibrary{calc}
\usetikzlibrary{decorations.markings}

\newcommand{\ind}{\mathbbmss{1}}
\newcommand{\Real}{\mathbb{R}}

\newcommand{\Sym}{\mathfrak{S}}

\newcommand{\hcV}{\mathcal{V}}
\newcommand{\cycV}{S_{1}^{\rm cyc}}
\newcommand{\du}{\rm du}
\newcommand{\Sdu}{S_{1}^{\du}}
\newcommand{\pp}{{\rm p}}
\newcommand{\vertices}{[N]}

\newcommand{\Ec}{\E}

\newcommand{\nc}{{\rm nc}}
\newcommand{\Gnc}{G^{\nc}}

\newcommand{\Anc}{A^{\nc}}
\newcommand{\simnc}{\overset{\nc}{\sim}}

\newcommand{\W}{\mathcal{W}}
\newcommand{\Tree}{\mathcal{T}}
\newcommand{\simT}[1]{\overset{\Tree_{#1}}{\sim}}
\newcommand{\simTc}[1]{\overset{\check{\Tree}_{#1}}{\sim}}

\DeclareMathOperator{\Id}{I}

\DeclareMathOperator{\Spec}{Spec}
\DeclareMathOperator{\Sib}{Sib}
\DeclareMathOperator{\sign}{sign}
\DeclareMathOperator{\sgn}{sgn}

\renewcommand{\leq}{\leqslant}
\renewcommand{\geq}{\geqslant}
\renewcommand{\epsilon}{\varepsilon}

\numberwithin{equation}{section}
\numberwithin{figure}{section}

\usepackage{enumitem}

\let\originalleft\left
\let\originalright\right
\renewcommand{\left}{\mathopen{}\mathclose\bgroup\originalleft}
\renewcommand{\right}{\aftergroup\egroup\originalright}

\renewcommand{\b}[1]{\boldsymbol{\mathrm{#1}}} 
\renewcommand{\r}{\mathrm} 

\renewcommand{\P}{\mathbb{P}}
\newcommand{\E}{\mathbb{E}}
\newcommand{\R}{\mathbb{R}}

\newcommand{\N}{\mathbb{N}}

\newcommand{\col}{\vcentcolon}
\newcommand*{\deq}{\mathrel{\vcenter{\baselineskip0.65ex \lineskiplimit0pt \hbox{.}\hbox{.}}}=}

\newcommand{\p}[1]{(#1)}
\newcommand{\pb}[1]{\bigl(#1\bigr)}
\newcommand{\pB}[1]{\Bigl(#1\Bigr)}
\newcommand{\pbb}[1]{\biggl(#1\biggr)}

\newcommand{\abs}[1]{\lvert #1 \rvert}

\newcommand{\norm}[1]{\lVert #1 \rVert}

\newcommand{\normBB}[1]{\Biggl\lVert #1 \Biggr\rVert}

\newcommand{\scalar}[2]{\langle#1 \mspace{2mu}, #2\rangle}

\newcommand{\order}{O}

\newcommand{\ee}{\r e}

\newcommand{\dd}{\r d}

\author{Thomas Buc--d'Alché \and Antti Knowles}
\title{Semilocalization for inhomogeneous random graphs}

\begin{document}

\maketitle

\begin{abstract}
We analyse the eigenvectors of the adjacency matrix of a random inhomogeneous graph constructed from a specified degree sequence. We assume that the empirical degree sequence has bounded mean and variance. We show that near the edges of the spectrum, the eigenvectors are semilocalized in the sense that their mass concentrates around a small set of resonant vertices. For the extremal eigenvalues, we establish localization around a single vertex. In order to obtain effective estimates in the presence of highly inhomogeneous degrees, we introduce a new economical pruning procedure that carefully extracts a forest from the original graph, whose adjacency matrix is compared to that of the original graph using a suitably constructed local coupling to random trees with independent edges.
\end{abstract}

\section{Introduction}
\label{sec:introduction}

\subsection{Overview}
Let $A$ be the adjacency matrix of a random graph on the vertex set $[N] = \{1, \dots, N\}$. We are interested in the geometric structure of the eigenvectors of $A$, in particular their \emph{spatial localization}. An $\ell^2$-normalized eigenvector $\b q = (q_x)_{x \in [N]} \in \R^N$ gives rise to a probability measure $x \mapsto q_x^2$ on the set of vertices $[N]$. Informally, $\b q$ is \emph{delocalized} if its mass is approximately uniformly distributed throughout $[N]$, and \emph{localized} if its mass is essentially concentrated on a small number of vertices.

In this paper we study the spatial localization of eigenvectors for a general inhomogeneous random graph model from \cite{hofstad_random_2016, bollobas2007phase, chung2002connected}, where the typical degree of a vertex is of order one, but these degrees can be of very different sizes. They are constructed from a specified degree sequence, which we assume to have finite first and second empirical moments; see \eqref{P_GRG} below. Our result applied when the tail of the empirical degree sequence is no lighter than exponential. In particular, we consider graphs whose empirical degree distribution is exponential or has a power law behaviour (so-called scale-free graphs).

Such inhomogeneous graphs differ substantially from their homogeneous counterpart, the Erd\H{o}s-Rényi graph. Heuristically, homogeneous random graphs are expected to exhibit random matrix behaviour such as delocalized eigenvectors. For the Erd\H{o}s-Rényi graph, the delocalized region is well understood \cite{ADK20,ADK_delocalized, HeKnowlesMarcozzi2018, EKYY1}; it corresponds essentially to the regime where the expected degree is at least of logarithmic size. In contrast, if the graph becomes too inhomogeneous, this random matrix behaviour breaks down and one expects eigenvectors to localize. For the Erd\H{o}s-Rényi graph, such localized behaviour has been investigated in \cite{ADK20, alt2024localized, ADK21, hiesmayr2025spectral}. In particular, it was shown in \cite{ADK20} that eigenvectors associated with eigenvalues near the spectral edges are \emph{semilocalized}, which means that their mass is concentrated on a vanishing fraction of the total number of vertices. More precisely, a semilocalized eigenvector has its mass concentrated around a small number of \emph{resonant vertices}, whose local energy is close to the associated eigenvalue, and around which the distribution of mass is radial and exponentially decaying.

An important motivation for our work stems from the localization-delocalization transition for disordered quantum systems, whereby the adjacency matrix $A$ is interpreted as the Hamiltonian of a free quantum particle hopping in the random geometry defined by the graph. This transition is an example of an \emph{Anderson transition}, where a disordered quantum system exhibits localized or delocalized states depending on the disorder strength and the location in the spectrum, corresponding to an insulator or conductor, respectively. Originally proposed in the 1950s \cite{anderson1958absence} to model conduction in semiconductors with random impurities, this phenomenon is now recognized as a general feature of wave transport in disordered media, and is one of the most influential ideas in modern condensed matter physics \cite{lee1985disordered, evers2008anderson, lagendijk2009fifty, abrahams201050}. It is expected to occur in great generality whenever linear waves, such as quantum particles, propagate through a disordered medium. One expects localization to occur for strong enough disorder, or inhomogeneity, of the system and for eigenvalues close enough to the spectral edge.

The main result of this paper is a proof of semilocalization for general inhomogeneous random graphs and eigenvectors associated with eigenvalues near the spectral edge. In addition, we prove a stronger localization result around a single vertex for eigenvectors associated with the extremal eigenvalues. This shows that the phenomenon is not tied to the homogeneous, mean-field law of the Erd\H{o}s-Rényi graph, for which it was established in the works \cite{ADK20, alt2024localized, ADK21, hiesmayr2025spectral} cited above.

A fundamental difference between this paper and the homogeneous result of \cite{ADK20} on the  Erd\H{o}s-Rényi graph is the source of the inhomogeneity. In  \cite{ADK20}, the inhomogeneity arises from a lack of concentration of the degrees in the regime where the typical degree is of order $\log N$. In contrast, in this paper, the inhomogeneity arises from the inhomogeneity of the degree sequence in the regime where the typical degree is of order 1. Our results are effective as soon as the largest degree is much larger than $\frac{\log N}{\log \log N}$, which is the order of the largest degree in the Erd\H{o}s-Rényi graph at fixed expected degree. In that sense, our results exploiting inhomogeneities of the degree sequence are sharp up to a constant.

We conclude this overview with a brief sketch of the new ideas of our proof. Obtaining sharp control on the optimal scale $\frac{\log N}{\log \log N}$ in the presence of highly inhomogeneous degrees requires several fundamental changes in the argument. This is also manifested in structure of the semilocalization profile vectors $\b u_\pm(x)$ defined in \cref{prop:ps-ev} below, which are no longer spherical, but incorporate a hierarchical structure among the the neighbours and their neighbours, arising from the presence of vertices whose degrees may of a different order of magnitude. In order to deal with such a large range of vertex degrees, we develop a new pruning procedure of the graph, based on so-called down-up paths (see \cref{def:down-up}). This results in a more economical pruning procedure than the ones employed in \cite{ADK20, alt2024localized, ADK21}, which would remove too many edges when applied to highly inhomogeneous graphs. This pruning procedure results in a carefully extracted global forest, unlike the locally tree-like pruned graphs from \cite{ADK20, alt2024localized, ADK21}. The key tool to analyse the spectrum of this forest is a coupling to a random tree, whose edges, unlike those of the forest obtained from the pruning, are independent, and which contains balls of small enough radius of the original forest. This coupling allows us to obtain sufficiently strong estimates on the operator norm of the error resulting from the pruning. In this latter step, we bypass the arguments of \cite{ADK20, alt2024localized, ADK21} relying on spectral bounds on the nonbacktracking matrix and Ihara-Bass-type formulas for norm estimates, which are not effective in the inhomogeneous setting.

\paragraph{Conventions and notations}
Every quantity that is not explicitly \emph{constant} depends on $N$. We omit this dependence in our notation. We use $C,c$ to denote generic positive constants, which may change from step to step. We write $X = O(Y)$ to mean
$\abs{X} \leq C Y$. 
We write $\N = \{0,1,2,\dots\}$.
We set $[n] \deq \{1, \ldots, n\}$ for any $n \in \N^*$. 
We write $\# X$ for the cardinality of a finite set $X$.

We use the following notations for vectors. Vectors in $\R^N$ are denoted by boldface lowercase Latin letters; we use the notation $\b v = (v_x)_{x \in [N]} \in \R^N$ for the entries of a vector. We denote by $\scalar{\b v}{\b w} = \sum_{x \in [N]} v_x w_x$ the Euclidean scalar product on $\R^N$ and by $\norm{\b v} = \sqrt{\scalar{\b v}{\b v}}$ the induced Euclidean norm. We denote by $\norm{A}$ the induced operator norm on $N \times N$ matrices $A$.
For any $x \in [N]$, we define the standard basis vector $\b 1_x \deq (\delta_{xy})_{y \in [N]} \in \R^N$.
To any subset $S \subset [N]$ we assign the vector $\b 1_S\in \R^N$ given by $\b 1_S \deq \sum_{x \in S} \b 1_x$. For $S \subset [N]$ we denote by $G \vert_S$ the subgraph of $G$ induced by $S$.

\subsection{Model and assumptions}

Let $G$ be a graph on the vertex set $[N]$, 
whereby we identify $G$ with its set of edges. It is characterized by its adjacency matrix $A = (A_{xy})_{x,y \in [N]}$, where $A_{xy} \deq \ind_{\{x,y\} \in G}$. We only consider simple graphs, so that $A_{xx} = 0$ for all $x \in [N]$. We write $x \sim y$ whenever $A_{xy} = 1$. We endow $G$ with the usual graph distance: the distance between $x$ and $y$ is the number of edges in the shortest path in the graph joining $x$ and $y$. For $x \in [N]$ and $r \in \N$, we denote by $B_r(x)$ the ball of radius $r$ around $x$, i.e.\ the set of vertices at distance at most $r$ from $x$, as well as by $S_r(x)$ the sphere of radius $r$ around $x$, i.e.\ the set of vertices at distance $r$ from $x$. We denote by $D_x = \sum_{y \in [N]} A_{xy}$ the degree of $x$.

In this paper we consider random graphs $G$ whose law is given by the Generalized Random Graph (GRG) model (see e.g.\ \cite[Chapter 6]{hofstad_random_2016}): the family $(A_{xy} \col x < y)$ is independent with
\begin{equation} \label{P_GRG}
\P(A_{xy} = 1) = p_{xy} \deq \frac{w_{x}w_{y}}{\sum_{z}w_{z} + w_{x}w_{y}}\,,
\end{equation}
where $(w_x)_{x \in [N]} \in (0,\infty)^N$ is a family of strictly positive weights.

For $k \in \N^*$, we define the empirical moment
\begin{equation}\label{eq:def-empirical-moment}
  m_k \deq \frac{1}{N}\sum_{x = 1}^{N}w_{x}^k\,.
\end{equation}
We fix two constant exponents $0 < \epsilon < 1/2$ and $0 < \delta < 1/3$ and make the following two assumptions on the weights $(w_x)$.

\begin{assumption}\label{hyp:upperbd-weights}
  For all $x \in \vertices$,
  \begin{equation*}
    w_{x} \leq N^{1/2 - \epsilon}\,.
  \end{equation*}
\end{assumption}

\begin{assumption}\label{hyp:behavior-m21}
  The first and second empirical moments satisfy
  \begin{equation*}
    m_{1} \geq N^{-\epsilon} \,, \qquad \frac{m_{2}}{m_{1}} = \order\Bigl((\log N)^{\delta}\Bigr)\,.
  \end{equation*}
\end{assumption}

\begin{remark}\label{rem:asympt-expr-p}
Another commonly used model of an inhomogeneous random graph is the Chung-Lu model \cite{chung2002connected}, where, instead of \eqref{P_GRG}, we set
\begin{equation*}
\P(A_{xy} = 1) = \tilde p_{xy} \deq \frac{w_x w_y}{\sum_z w_z} \wedge 1\,.
\end{equation*}
It follows immediately from Assumptions \ref{hyp:upperbd-weights} and \ref{hyp:behavior-m21} that $p_{xy} = \tilde p_{xy} (1 + \order(N^{-\epsilon}))$,
so that, in the regime that we consider, the GRG and Chung-Lu models are asymptotically equivalent. In particular, all of our results easily carry over to the Chung-Lu model.
\end{remark}

\begin{remark} \label{rem:mu_weights}
A natural way to construct the weights $(w_x)$ is to choose a fixed probability measure $\mu$ on $(0,\infty)$ and to choose the weights as either (i) independent random variables with law $\mu$ or (ii) the $(N+1)$-quantiles of $\mu$. We give two standard examples of $\mu$, to which our main result is applicable.
\end{remark}

\begin{example}[Power law]\label{ex:heavy-tail}
Suppose that $\mu$ is a power law with exponent $\alpha > 2$. That is,
\begin{equation} \label{mu_power}
\mu([t,\infty)) = L(t) \, t^{-\alpha}\,,
\end{equation}
where $L$ is a slowly varying function, i.e.\  there exists $u_{0} > 0$ such that
\begin{equation*}
\lim_{t \to \infty} \frac{L(tu)}{L(u)} = 1 \quad \text{for all } u \geq u_{0}\,.
\end{equation*}
It is easy to check that Assumptions \ref{hyp:upperbd-weights} and \ref{hyp:behavior-m21} hold (in case (i) of \cref{rem:mu_weights}, they hold with high probability). For instance, for $\mu([t,\infty)) = (c / t)^\alpha \wedge 1$ 
for some $c > 0$, the quantile weights from \cref{rem:mu_weights} (ii) are given by
\begin{equation} \label{w_x_power_law}
w_x = c \pbb{\frac{N+1}{x}}^{1/\alpha}\,.
\end{equation}
If the weights $(w_x)$ have a power law distribution with exponent $\alpha$, then with high probability so do the degrees $(D_x)$ of the random graph; see e.g.\ \cite[Theorem 6.12]{hofstad_random_2016}. Hence, the choice \eqref{mu_power} yields random graphs with power-law degree distribution (often also called scale-free graphs).
\end{example}

\begin{example}[Exponential distribution]\label{ex:exponential}
Let $\mu$ be the exponential distribution with parameter $\alpha > 0$. Then Assumptions
  \ref{hyp:upperbd-weights} and \ref{hyp:behavior-m21} hold (in case (i) of \cref{rem:mu_weights}, they hold with high probability). In case (ii) of \cref{rem:mu_weights}, the weights are
  \begin{equation} \label{exp_quantiles}
    w_x = \frac{1}{\alpha}\log \frac{N+1}{x}\,,
  \end{equation}
  for $x \in [N]$.
\end{example}

\begin{remark} \label{rem:clique}
The condition from \cref{hyp:upperbd-weights} ensures that every edge in the graph has a probability $o(1)$ of being included in the graph. Without it, the graph may exhibit a large subset $S$ of vertices, such that $p_{xy} \asymp 1$ for all $x,y \in S$. This would lead to the presence of large cliques, to which our analysis does not apply without significant modifications.
\end{remark}

We conclude this section with the following quantitative notion of high probability that we use throughout the paper.
\begin{definition}\label{def:very-high-p}
  Let $\nu > 0$ be a constant. An event $E \equiv E_N$ holds with \emph{$\nu$-high probability} if there exists a constant $C > 0$ such that, for all $N$,
  \begin{equation*}
    \P(E) \geq 1 - CN^{-\nu}\,.
  \end{equation*}
\end{definition}

\subsection{Semilocalization}
\label{sec:stat-semi-local}

Our main result, semilocalization for the eigenvectors of $A$, pertains to vertices whose associated energies $\sqrt{D_x}$ lie in a interval of width $\eta > 0$ centered around an energy $\lambda \in \R$. We define the set of \emph{resonant vertices} by
\begin{equation*}
  \W_{\lambda, \eta} \deq \left\{x \in [N]\col \left|\sqrt{D_{x}} - \lambda\right| \leq \eta\right\}\,.
\end{equation*}

\begin{theorem}[Semilocalization]\label{thm:main-intro}
Suppose that Assumptions \ref{hyp:upperbd-weights} and \ref{hyp:behavior-m21} hold. For any $\nu > 0$ there exists $C_{\nu} > 0$ such that the following holds with $\nu$-high probability. For any normalized eigenvector $\b q$ of $A$ with associated eigenvalue $\lambda$ and for any $\eta \leq |\lambda|/2$ we have
\begin{equation} \label{semilocalization}
    \sum_{x\in \W_{\lambda, \eta}} \scalar{\b q}{\b u_{\sgn(\lambda)}(x)}^{2} \geq 1 - \frac{C_{\nu}}{\eta^{2}}\frac{\log N}{\log\log N}\,,
  \end{equation}
  where $\b u_{\pm}(x)$ is an explicit vector supported in $B_{2}(x)$ defined in \cref{prop:ps-ev} below.
\end{theorem}

The next result, \cref{prop:size-W}, shows that in the cases of
interest, the set of resonant vertices is of negligible size compared
to the size of the graph. More precise results leading to
complete localization for the extremal eigenvalues are
discussed in \cref{sec:biggest-eigenvalues}.
\begin{proposition}\label{prop:size-W}
  For any $\nu > 0$ there exists $C_{\nu} > 0$ such that with $\nu$-high
  probability
  \begin{equation*}
    \#\W_{\lambda, \eta} \leq 2 \Ec \left[ \#\W_{\lambda, \eta+1} \right] \vee \frac{2\nu \log N}{\log \log N}
  \end{equation*}
  whenever
 \begin{equation} \label{eta_condition}
  C_{\nu}\sqrt{\frac{\log N}{\log\log N}} < \eta < \frac\lambda2.
\end{equation}
Moreover, the expectation on the right-hand side can be estimated as follows.
\begin{enumerate}[label=(\roman*)]
\item
In general, a second moment method yields
  \begin{equation*}
    \Ec \left[ \#\W_{\lambda, \eta+1} \right] \leq \frac{m_{2}}{(\lambda - \eta)^{4}}N\,.
  \end{equation*}
  \item
     When the weights are taken to be the quantiles of the a power law distribution, as in Example \ref{ex:heavy-tail}, we have
       \begin{equation*}
    \Ec \left[ \#\W_{\lambda, \eta+1} \right] = \order\pbb{\frac{N\eta}{\lambda^{2\alpha+3-\iota}} + (\log N)^{2\delta}}\,,
  \end{equation*}
for any constant $\iota > 0$ if $\lim_{t\to \infty}L(t) = \infty$ and for $\iota = 0$ if $L$ is bounded.  
  \item
   When the weights are taken to be the quantiles of the exponential distribution, as in Example \ref{ex:exponential}, we have
  \begin{equation*}
    \Ec \left[ \#\W_{\lambda, \eta+1}  \right] = \order\pbb{\frac{N}{(\alpha+1)^{(\lambda-\eta)^{2}}} + (\log N)^{2\delta}}\,.
  \end{equation*}
\end{enumerate}
\end{proposition}

To illustrate \cref{thm:main-intro}, we apply it to the two examples from \cref{ex:heavy-tail} and \cref{ex:exponential}. For simplicity, we only focus on the largest eigenvalues, although analogous results hold for the smallest eigenvalues.

\begin{example}[\cref{ex:heavy-tail} continued]
Let $w_x$ be the $x$-th $(N+1)$-quantile of the power law distribution with exponent $\alpha > 2$, given in \eqref{w_x_power_law}. Abbreviate $t \deq N/x$ and suppose that $t \gg (\log N)^\alpha$. By \eqref{eq:d-weight} and \cref{lem:estimate-degree} below, $D_x = c \, t^{1/\alpha} (1 + o(1))$ with $\nu$-high probability for any fixed $\nu > 0$. By \cref{thm:eigenvalues} below, we conclude that the $x$-th largest eigenvalue satisfies $\lambda_x(A) = \sqrt{c} \, t^{1/2\alpha}(1 + o(1))$ with $\nu$-high probability. By \cref{thm:main-intro} and \cref{prop:size-W}, we therefore conclude that the eigenvector associated with the $x$-th largest eigenvalue is semilocalized in the sense of \eqref{semilocalization}, with $\eta = c' t^{1/2\alpha}$,  around at most
\begin{equation*}
\frac{N}{t^{1 + \frac{1}{\alpha}}} + \log N
\end{equation*}
vertices. 
In particular, we obtain nontrivial semilocalization for the $O(N (\log N)^{-\alpha})$ largest eigenvalues. We note that a slightly more careful analysis for for the largest eigenvalues allows one to upgrade semilocalization to complete localization (i.e.\ semilocalization around a single vertex), which is not expected to hold in the full semilocalized regime; see \cref{thm:localization} and \cref{ex:localization_power} below.

\end{example}

\begin{example}[\cref{ex:exponential} continued]
Let $w_x$ be the $x$-th $(N+1)$-quantile of the exponential distribution with parameter $\alpha > 0$, given in \eqref{exp_quantiles}. Let $\nu > 0$ and $\gamma \in (0,1)$. From \eqref{eq:d-weight}, \cref{lem:estimate-degree}, and a union bound, we find that with $\nu$-high probability for all $x \leq N^\gamma$ we have
\begin{equation} \label{D_x_bounds_exp}
w_x (1 - c) \leq D_x \leq w_x (1 + c)
\end{equation}
provided that
\begin{equation} \label{alpha_nu_constraint}
\frac{4\alpha (\nu + \gamma)}{1 - \gamma} < c^2 \leq \frac{1}{2}\,.
\end{equation}
From \cref{thm:eigenvalues} below, we therefore conclude that the $x$-th largest eigenvalue of $A$ is bounded from below by
\begin{equation*}
b \deq \sqrt{\frac{(1 - c - o(1))(1 - \gamma)}{\alpha}} \sqrt{\log N}\,.
\end{equation*}
Choosing $\eta = (1 - 1/\sqrt{2}) b$, we therefore conclude from \cref{thm:main-intro} and \cref{prop:size-W} that semilocalization in the sense of \eqref{semilocalization} with a right-hand side $1 - O_\epsilon\pb{\frac{1}{\log \log N}}$ holds, around at most $N^\beta$ vertices for any
\begin{equation*}
\beta > 1 - \log(1 + \alpha) \frac{(1 - c)(1 - \gamma)}{2 \alpha}\,,
\end{equation*}
under the constraint \eqref{alpha_nu_constraint}.
While this gives a nontrivial region of semilocalization for any $\alpha > 0$, the range of $\gamma$ and $\nu$ obtained is far from optimal and far from what \cref{thm:main-intro} yields for this example. Indeed, to simplify the presentation, we required \eqref{D_x_bounds_exp} to hold for each vertex $x \geq N^\gamma$, which imposes the strong constraint \eqref{alpha_nu_constraint}. Without this condition, the strong concentration of the degrees used above breaks down, but a more sparing analysis using the first and second moment method on the counting function of the degrees can still be applied, which yields a far larger region for semilocalization; for brevity, we do not carry this analysis out here. To conclude this example, we emphasize the importance of the optimal factor $\log \log N$ on the right-hand side of \eqref{semilocalization}, without which semilocalization for exponential weights could not be established.
\end{example}

\begin{remark}[Relaxation of the hypothesis $\delta < 1/3$.]
The condition $\delta < 1/3$ in \cref{hyp:behavior-m21} prevents the application of our results to graphs of high mean degree, such as $\log N$. However, we may relax the assumption that $\delta < 1/3$ and introduce a parameter
  $\beta \geq 1$ with $\delta < \beta/3$. By using the same proof as below, we may
  obtain a localization result with the term
  $\frac{\log N}{\log\log N}$ replaced by
  $\frac{(\log N)^{\beta}}{(\log \log N)^{\mu}}$ where $\mu = 1$ if $\beta = 1$ and
  any $\mu \geq 1$ if $\beta > 1$. Such an extension can for instance be used to establish semilocalization for sufficiently inhomogeneous random graphs with polylogarithmic mean degrees.
\end{remark}

\subsection{The degrees and the weights} \label{sec:degrees-weights}

We conclude the introduction with a few basic facts and tools on the relationship between weights and degrees, which we shall use throughout the proofs. In the sequel, we will mainly consider
vertices of high degree. In particular, vertices with weights greater
than $c\log N$, for a fixed $c > 0$, are easier to describe.

We denote by
\begin{equation*}
d_x \deq \E[D_x]
\end{equation*}
the expectation of the degree of $x$. Using
\begin{equation}\label{eq:proba-edge}
  p_{xy} = \frac{w_{x}w_{y}}{m_{1}N + w_{x}w_{y}} \leq \frac{w_{x}w_{y}}{m_{1}N}\,.
\end{equation}
we find
\begin{equation*}
  d_{x} = \sum_{y \in \vertices\setminus\{x\}}\frac{w_{x}w_{y}}{\sum_{z}w_{z} + w_{x}w_{y}} \leq w_{x}\,.
\end{equation*}
An asymptotically matching lower bound for $d_x$ follows from \cref{rem:asympt-expr-p}, which yields
\begin{equation}\label{eq:d-weight}
  d_{x} = \sum_{y \in \vertices\setminus \{x\}}\frac{w_{x}w_{y}}{m_{1}N}\Bigl(1+\order(N^{-\epsilon})\Bigr) = w_{x}\pbb{1-\frac{w_{x}}{m_{1}N}}\Bigl(1+\order(N^{-\epsilon})\Bigr) = w_{x}\Bigl(1 + \order(N^{-\epsilon})\Bigr)\,.
\end{equation}

As soon as the weight of a vertex $x$ is at least of order $\log N$,
the degree $D_{x}$ is as well at least of order $\log N$, with $\nu$-high probability.
\begin{lemma}\label{lem:estimate-degree}
  Let $\nu > 0$ and $x \in [N]$. With $\nu$-high probability,
  \begin{equation*}
    d_{x} - \sqrt{2\nu d_{x}\log N} \leq D_{x} \leq d_{x} + 2\sqrt{\nu\log N\left(d_{x}\vee\frac{4\nu}{9}\log N\right)}\,.
  \end{equation*}
\end{lemma}
\begin{proof}
  This is an application of Bennett's inequality (see
  \cite[Theorem 2.9]{BLM13}). We have for $M > 0$,
  \begin{equation*}
    \P(D_{x} \geq M+d_{x})
    = \P(D_{x} - d_{x} \geq M)
    \leq \exp\Biggl(-\frac{M^{2}}{2(w_{x} + M/3)}\Biggr)\,.
  \end{equation*}
  If we choose $M = 2\sqrt{\nu\log N (w_{x} \vee (2/3)\nu\log N)}$, we get
  the upper bound.

  For the lower bound, we use the slightly sharper bound (see
  \cite[Theorem 2.21]{hofstad_random_2016}),
  \begin{equation*}
    \P(D_{x} \leq d_{x} - \sqrt{2\nu d_{x}\log N})
    \leq \exp\Biggl(-\frac{(\sqrt{2\nu d_{x}\log N})^{2}}{2d_{x}}\Biggr)
    \leq N^{-\nu}\,. \qedhere
  \end{equation*}
\end{proof}

\begin{remark}[Lower bound for $D_{x}$]\label{rem:ratio-big-dx}
\cref{lem:estimate-degree} implies that as soon as $D_{x} \geq 1$
  and $d_{x} \geq 4\nu \log N$, we have that
  \begin{equation*}
    d_{x} \leq \frac{\sqrt{2}}{\sqrt{2}-1}D_{x} \quad \text{ with $\nu$-high probability. }
  \end{equation*}
  Otherwise, if $D_{x} \geq C \geq 1$ and $d_{x} < 4\nu \log N$, we only have the crude bound
  \begin{equation*}
    d_{x} \leq \frac{4\nu}{C}\log N D_{x}\,.
  \end{equation*}
  Hence, we have in any case
  \begin{equation*}
    d_{x} \leq \left( \frac{\sqrt{2}}{\sqrt{2} - 1} \vee \frac{4\nu}{C}\log N \right) D_{x}\,.
  \end{equation*}
\end{remark}
\begin{remark}[Upper bound for $D_{x}$]\label{rem:upper-bd-degree}
\cref{lem:estimate-degree} implies that if $d_{x} > \frac{4\nu}{9}\log N$, we have
  \begin{equation*}
    D_{x} \leq d_{x} + 2\sqrt{\nu\log N d_{x}} \leq d_{x} + 2\sqrt{\frac{9}{4}d_{x}^{2}} = 4d_{x}
  \end{equation*}
  with $\nu$-high probability. However, if $d_{x} \leq \frac{4\nu}{9}\log N$, we have
  \begin{equation*}
    D_{x} \leq d_{x} + \frac{4\nu}{3}\log N \leq \frac{16\nu}{9}\log N
  \end{equation*}
  with $\nu$-high probability. Hence, we have in all cases
  \begin{equation*}
    D_{x} \leq 4d_{x} \vee \frac{16 \nu}{9}\log N \quad \text{ with \(\nu\)-high probability.}
  \end{equation*}
\end{remark}

In the sequel, it will be convenient to order the vertices in terms of
their degree.
\begin{definition}\label{def:order}
  We define the strict order relation $\prec$ on the set of vertices $[N]$ as
  follows. For any two vertices $x, y \in [N]$,
  \begin{equation*}
    x \prec y \text{ if and only if } \Bigl( (D_{x} < D_{y}) \text{ or } (D_{x} = D_{y} \text{ and } x > y) \Bigr)\,.
  \end{equation*}

  We define the (random) permutation $\pi \in \Sym_{N}$ to be the
  unique permutation such that
  \begin{equation*}
    \pi(N) \prec \pi(N-1) \prec \cdots \prec \pi(2) \prec \pi(1)\,.
  \end{equation*}
\end{definition}
Note that in particular $D_{\pi(N)} \leq D_{\pi(N-1)} \leq \cdots \leq D_{\pi(2)} \leq  D_{\pi(1)}$.
This order allows us to define two notions of neighborhood and degree:
\begin{equation}\label{eq:def-S+-}
  \begin{split}
    S_{1}^{+}(x) &= \{y \in \vertices\col x \sim y, x \prec y\}~\text{ and }~D_{x}^{+} = \# S_{1}^{+}(x)\,,\\
    S_{1}^{-}(x) &= \{y \in \vertices\col x \sim y, x \succ y\}~\text{ and }~D_{x}^{-} = \# S_{1}^{-}(x)\,.
  \end{split}
\end{equation}
Thus, $S_1^+(x)$ and $S_1^-(x)$ partition $S_1(x)$ with $D_x = D_x^+ + D_x^-$.

\section{Pruning the graph}
\label{sec:pruning-graph}

Similarly as in \cite{ADK19, ADK20}, it is more convenient to work on a
pruned version of the graph. The GRG model is inhomogeneous, compared
to the Erd\H{o}s-R{\' e}nyi model: since the laws of the degrees in
the graph are governed by the weights $(w_{x})$, there are greater
differences of degrees in the graph. Because of this greater
heterogeneity, the pruning procedure has to be more subtle. As in the
Erd\H{o}s-R\'enyi case, we first prune the graph to remove cycles in
small balls. Then, instead of removing all edges connecting two
vertices of high degree, we remove edges appearing in a very specific
pattern. This procedure is key to simplifying the computations in
\cref{sec:finding-eigenvectors}.

Most of this section is devoted to estimating the error in operator norm we make when working with the adjacency matrix of the prunded graph rather than the adjacency matrix $A$. This is one of the main technical difficulty of our argument.

Throughout the following we fix a constant $r \geq 6$. For convenience,
we will sometimes omit it from the notation.

\begin{remark}
  The pruning of \cite{ADK20} (besides removing cycles) amounts to
  removing all the edges between vertices of high degree. Because of the
  inhomogeneity in the GRG case, this would mean cutting a number of
  vertices proportional to $w_{x}$ around a vertex $x$. This would
  prevent us from obtaining good bounds on the error we make when
  replacing the adjacency matrix of the original graph by the one of
  the pruned graph.

  The new pruning presented below is asymmetric, and based on the
  order $\prec$ introduced in \cref{def:order}. We orient each edge
  $\left\{ x, y \right\}$ in $G$ from $x$ to $y$ if $x \prec y$. We
  then remove special paths we call \emph{down-up paths}, paths of
  length 2 between vertices $x$ and $z$, going through a vertex $y$ so
  that $y \prec x \prec z$.

  In the Erd{\H o}s-R{\' e}nyi model, one can prune the graph so that
  in a ball of small radius around any vertex, the graph is a tree and
  contains at most one vertex of high degree. In the GRG model, by
  removing down-up paths, we get a pruned graph which is globally a
  forest. Furthermore, each connected component is a tree that is
  naturally rooted at the vertex of greatest degree in the connected
  component.
\end{remark}

\subsection{The pruning procedure}
\label{sec:pruning-procedure}

We now explain precisely the pruning procedure. It produces a pruned graph $G^{\pp}$. To give the construction, we
introduce notation pertaining to paths.
\begin{definition}
  A \emph{path} $\gamma$ in a graph $G$ is a sequence of vertices
  $\gamma = (\gamma_{0}, \gamma_{1}, \ldots, \gamma_{l})$, with
  $\left\{ \gamma_{i-1}, \gamma_{i} \right\} \in G$ for $i\in[n]$. The length of the
  path is $l(\gamma) = l$. A path is said to be \emph{simple} if $\gamma_{i} \neq \gamma_{j}$
  for $i \neq j, \{i, j\} \neq \{0, l\}$.

  The set of paths $\gamma$ in $G$ satisfying $\gamma_0 = x$ and
  $\gamma_{l(\gamma)} = y$ is denoted by $\mathcal{P}_{xy}(G)$, and
  its subset of simple paths is denoted by $\mathcal{P}^{*}_{xy}(G)$.
\end{definition}

We mentioned that we will consider a particular set of paths, the
down-up paths, which we define below.
\begin{definition}\label{def:down-up}
  A \emph{down-up path} between two distinct vertices $x\in \vertices$
  and $z \in \vertices$ is a path $(x, y, z)$ with
  $y \prec x \prec z$.\end{definition}

Consider a vertex $x \in [N]$. We define the two sets
\begin{equation}\label{eq:def-two-pruning-sets}
  \begin{split}
    \cycV(x) &= \{y \in S_{1}(x)\col \exists \gamma \in \mathcal{P}^{*}_{xx}(G), l(\gamma) \leq 2r + 1, \gamma_{1} = y\}\,,\\
    \Sdu(x) &= \{y \in S_1(x) \setminus \cycV(x) \col \exists z \in S_1(y) \setminus \cycV(y), y \prec x \prec z\}\,.
  \end{split}
\end{equation}
i.e.\ $\cycV(x)$ is the set of vertices connected to $x$ that are part
of a cycle which is a simple loop, and $\Sdu(x)$ is the set of
vertices part of an up-down path starting at $x$.

\begin{figure}[htp]
  \centering
  \includegraphics[width=0.5\textwidth]{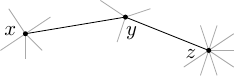}
  \caption{\label{fig:du} A down-up path.}
\end{figure}

To construct the pruned graph $G^{\pp}$, we proceed in two
steps. Firstly, we do a first pruning procedure to remove some cycles
in balls of radius $r$ around all vertices. Then, we remove the
down-up path, see \cref{def:down-up}.
\begin{definition}[Pruning procedure]\label{def:pruning}
The graph $G^\pp$ is defined by the following procedure.
  \begin{enumerate}
  \item\label{itm:remove-cycle} for each $x \in [N]$, and each $y \in \cycV(x)$, we
        remove from $G$ the edge $\{x, y\}$, and then
  \item\label{itm:remove-du} for each $x \in \vertices$, and each
        $y \in \Sdu(x)$ , we remove from $G$ the
        edge $\{x, y\}$.
\end{enumerate}
We denote the graph obtained after the first step by $\Gnc$.
\end{definition}
Note that the definition of $G^\nc$ and $G^\pp$ do not depend on the order in which the edges are removed.
We indicate by the superscript $\pp$ (respectively $\nc$) that the
adjacency matrix, degrees, spheres, balls, ... correspond to the
pruned graph $G^{\pp}$ (respectively $G^{\nc}$). For instance,
$D^{\pp}_{x}$ is the degree of the vertex $x$ in $G^{\pp}$. The event
that $\{x, y\} \in G^\nc$ is denoted by
$\{x \overset{\nc}{\sim} y\}$, and that $\{x, y\} \in G^\pp$ is denoted by
$\{x \overset{\pp}{\sim} y\}$.

\begin{proposition}\label{prop:pruning}
  Let $\nu > 0$. We define the threshold
  \begin{equation*}
    \xi = \xi_{\nu} = \frac{3(\nu + 1)(2 - 3\delta)}{(1 - \delta)(1 - 2\delta)}\frac{\log N}{\log\log N}\,.
  \end{equation*}
  The graph $G^{\pp}$ satisfies the following properties.
  \begin{enumerate}
    \item\label{item:diff-Dx} With $\nu$-high probability, for all $x \in \vertices$,
          $D_{x} - D^{\pp}_{x} \leq \xi/2$.
    \item There are no down-up paths in the graph $G^{\pp}$.
    \item The graph $G^{\pp}$ is a forest.
  \end{enumerate}
\end{proposition}
The proof of \cref{prop:pruning} is delayed to the end of
\cref{sec:removing-down-up}, and relies on the next two sections:
\cref{sec:removing-cycles} explains how many edges are removed when
constructing $G^{\nc}$ and \cref{sec:removing-down-up} explains how
many edges are removed when constructing $G^{\pp}$.

\begin{remark}\label{rem:prune-all}
  Note that we prune the graph around every vertex in the graph, and
  not only those of high degree as in \cite{ADK19,ADK20}. This ensures
  that the graph $G^{\pp}$ is globally, and not just locally, a
  forest. This proves useful in \cref{sec:bounds-overlinepi} when
  bounding the operator norm of the adjacency matrix of the graph
  $G^{\pp}$ restricted to the vertices of low degrees: we can
  immediately say that this operator norm is of the order of the
  square root of the maximal degree in this restricted graph.
\end{remark}

\subsection{Removing the cycles}
\label{sec:removing-cycles}
In this section, we estimate the number of edges to prune around each
vertex to remove all cycles in small balls centered around each
vertex. This will prove the part of \cref{prop:pruning} concerning step
\ref{itm:remove-cycle} of the pruning procedure. Recall that we fixed $r \geq 6$.

In \cref{sec:pruning-procedure}, we explained that we prune the edges
in the set $\left\{ \{x, y\}\col y \in \cycV(x) \right\}$ so as to remove
all cycles in all small balls. In this section, we give an upper bound
for the cardinality of the set $\cycV(x)$ defined in \eqref{eq:def-two-pruning-sets}.

\begin{proposition}\label{prop:edge-cyc}
  Fix $x \in \vertices$ and $\nu > 0$. There exists a constant $C_{\nu} \geq 0$
  such that with $\nu$-high probability,
  \begin{equation*}
    \#\cycV(x) \leq C_{\nu}\,.
  \end{equation*}
\end{proposition}

To prove \cref{prop:edge-cyc}, we will use the notion of edge-disjoint paths.
\begin{definition}\label{def:edge-inter}
  Two paths $\gamma$ and $\gamma'$ are \emph{edge-intersecting} if there exists $i$ and
  $j$ such that $\gamma_{i} = \gamma'_{j}$ and either $\gamma_{i-1} = \gamma'_{j-1}$ or
  $\gamma_{i-1} = \gamma'_{j+1}$. Two paths $\gamma$ and $\gamma'$ are
  \emph{edge-disjoint} if they are not edge-intersecting.
\end{definition}

\begin{figure}[htp]
  \centering
  \includegraphics[width=0.4\textwidth]{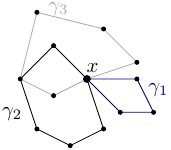}
  \caption{\label{fig:NEkr} A representation of the event $\mathrm{ED}_{k, r}$.
  The simple paths can share a vertex, but cannot be edge-intersecting.}
\end{figure}

We shall consider the following event concerning edge-disjoint paths
\begin{equation*}
  \mathrm{ED}_{k, r}(x) = \left\{\exists \gamma^{(1)}, \ldots, \gamma^{(k)} \in \mathcal{P}^{*}_{xx}(G\vert_{B_{r}(x)})\col~ \begin{aligned}
                                                                                                                                   &\forall i \neq j, \gamma^{(i)} \text{ and } \gamma^{(j)} \text{ are edge-disjoint},\\
                                                                                                                                                                                                                       &\sum_{i=1}^{k}l(\gamma^{(i)}) \leq 3k^{2}(2r+1)\\
                                                                                                              \end{aligned}
\right\}\,.
\end{equation*}
It is depicted in Figure \ref{fig:NEkr}.

We argue that $\#\cycV(x)$ can be bounded by the
number of non edge-intersecting paths in the graph
$G\vert_{B_{r}(x)\setminus\{x\}}$ between pairs of points of
$S_{1}(x)$.
\begin{lemma}\label{lem:non-edge-intersecting}
  Let $k \geq 1$ and $x \in \vertices$. We have the inclusion of events
  \begin{equation*}
      \{\#\cycV(x) \geq 3k\} \subset \mathrm{ED}_{k, r}(x)\,.
  \end{equation*}
\end{lemma}
\begin{proof}
  \begin{figure}[ht]
  \centering
  \includegraphics[width=0.6\textwidth]{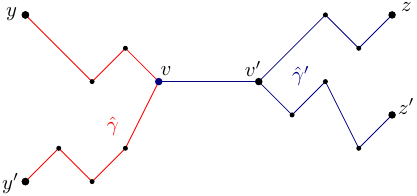}
  \caption{\label{fig:create-NEkr} Construction of the paths $\hat{\gamma}$ and $\hat{\gamma}'$.}
  The paths $\gamma$ and $\gamma'$ are edge-intersecting, respectively from $y$
  to $z$ and $y'$ to $z'$.
\end{figure}

Let $y_{1}, \ldots, y_{3k} \in \cycV(x)$ be distinct vertices. By definition,
there exists $z_{1}, \ldots, z_{3k} \in \cycV(x)$, with $z_{i} \neq y_{i}$, and
simple paths $\gamma^{(i)}, i = 1, \ldots, 3k$, each of them connecting $y_{i}$
to $z_{i}$ (without going through $x$). We now argue that we can find
a family of vertices
$\left\{ \tilde{y}_{i}, \tilde{z}_{i} \right\}_{1 \leq i \leq k}$ and a family
of paths $(\tilde{\gamma}_{1}, \ldots, \tilde{\gamma}_{k})$ such that
$\tilde{\gamma}^{(i)}$ is a path from $\tilde{y}_{i} \in S_{1}(x)$ to
$\tilde{z}_{i} \in S_{1}(x)$ and the vertices
$\left\{ \tilde{y}_{i}, \tilde{z}_{i} \right\}$ are all distinct. To see
this, consider the graph $\tilde{G}$ whose vertex set is
$\tilde{V} = \left\{ y_{i}, z_{i}; 1 \leq i \leq 3k \right\}$ and the edge set
is made of the pairs $\left\{ v_{1}, v_{2} \right\}$ such that
$v_{1} \in \tilde{V}$ and $v_{2} \in \tilde{V}$ are the endpoints of
one of the path in the family $(\gamma_{i})$. In each connected component
of $\tilde{G}$ made of the vertices $v_{1}, \ldots, v_{m}$ we choose a
perfect matching of the vertices $\left\{ v_{1}, \ldots, v_{m} \right\}$ if
$m$ is even, and a perfect matching of the vertices
$\left\{ v_{1}, \ldots, v_{m-1} \right\}$ if $m$ is odd. Note that $m \geq 2$ by
definition of the $\gamma_{i}$'s. For each couple $\left\{ v, v' \right\}$ in
a matching we can choose a simple path in $G$ going from $v$ to $v'$.
Each such path has length at most $3k(2r+1)$. Hence the total number
of pairs in the matching is
\begin{equation*}
  \frac{1}{2}\sum_{\tilde{G}_{c}\text{ connected component of } \tilde{G}} \# \tilde{V}_{c} - \ind_{\left\{ \# \tilde{V}_{c} \text{ odd } \right\}} = \frac{3k}{2} - \frac{1}{2}\# \left\{ \tilde{G}_{c} \text{ connected component with } \tilde{V}_{c} \text{ odd } \right\}\,,
\end{equation*}
where $\tilde{V}_{c}$ is the vertex set of the connected component
$\tilde{G}_{c}$ of $\tilde{G}$. Since each connected component has at least two vertices, we have
\begin{equation*}
  \# \left\{ \tilde{G}_{c} \text{ connected component with } \tilde{V}_{c} \text{ odd } \right\} \leq k.
\end{equation*}
This means that doing so, we find at least $k$ pairs: $2k$ distinct
vertices $\left\{ \tilde{y}_{i}, \tilde{z}_{i} \right\}$ and $k$ simple
paths $(\tilde{\gamma}_{i})$, with $\tilde{\gamma}_{i}$ going from
$\tilde{y}_{i}$ to $\tilde{z}_{i}$.

We now argue that from $(\tilde{\gamma}^{(i)})_{1 \leq i \leq k}$ we can produce
a family $(\hat{\gamma}^{(i)})_{1 \leq i \leq k}$ of edge-disjoint paths. To see
this, we explain how to produce from two paths $\tilde{\gamma}$ and
$\tilde{\gamma}'$ that do not share an endpoint a pair of paths $\hat{\gamma}$
and $\hat{\gamma}'$ that are edge-disjoint and do not share an endpoint.
Assume that $\left\{ v, v' \right\}$ is an edge shared by $\tilde{\gamma}$
and $\tilde{\gamma}'$. Without loss of generality, we can assume that
$\tilde{\gamma}$ and $\tilde{\gamma}'$ both first encounter $v$ and then
$v'$. Otherwise we replace $\tilde{\gamma}'$ by
\begin{equation*}
  (\tilde{\gamma}'_{l(\tilde{\gamma}')}, \tilde{\gamma}'_{l(\tilde{\gamma}')-1}, \ldots, \tilde{\gamma}'_{1}, \tilde{\gamma}'_{0}).
\end{equation*}
In that case, we can write
\begin{equation*}
  \begin{cases}
    \tilde{\gamma} &= (\tilde{\gamma}_{0}=y, \tilde{\gamma}_{1}, \tilde{\gamma}_{2}, \ldots, \tilde{\gamma}_{i}=v, \tilde{\gamma}_{i+1}=v', \tilde{\gamma}_{i+2}, \ldots, \tilde{\gamma}_{l-1}, z)\\
    \tilde{\gamma}' &= (\tilde{\gamma}'_{0}=y', \tilde{\gamma}'_{1}, \tilde{\gamma}'_{2}, \ldots, \tilde{\gamma}'_{j}=v, \tilde{\gamma}'_{j+1}=v', \tilde{\gamma}'_{j+2}, \ldots, \tilde{\gamma}'_{l'-1}, z').
  \end{cases}
\end{equation*}
We then set
\begin{equation*}
  \begin{cases}
    \hat{\gamma} &= (y, \tilde{\gamma}_{1}, \tilde{\gamma}_{2}, \ldots, \tilde{\gamma}_{i}=v = \tilde{\gamma}'_{j}, \tilde{\gamma}'_{j-1}, \ldots, \tilde{\gamma}'_{1}, y')\\
    \hat{\gamma}' &= (z, \tilde{\gamma}_{l-1}, \tilde{\gamma}_{l-2}, \ldots, \tilde{\gamma}_{i+1}=v' = \tilde{\gamma}'_{j+1}, \tilde{\gamma}'_{j+2}, \ldots, \tilde{\gamma}'_{l'-1}, z').
  \end{cases}
\end{equation*}
These new paths are depicted in \cref{fig:create-NEkr}. The new paths
$\hat{\gamma}$ and $\hat{\gamma}'$ do not share the edge
$\left\{ v, v' \right\}$, have disjoint endpoints, and all the edges in
$\hat{\gamma}$ and $\hat{\gamma}'$ appear in $\tilde{\gamma}$ or $\tilde{\gamma}'$ the
same number of times.

To construct a family of non edge-intersecting paths from the paths
$(\tilde{\gamma}^{(i)})_{1 \leq i \leq k}$, we consider first the path
$\tilde{\gamma}^{(1)}$ . We consider the first edge $e$ in $\gamma^{(1)}$ that
intersect another path $\gamma^{(i)}$, $i \neq 1$. We apply the procedure
described above, and obtain a new family of paths of size $k$:
\begin{equation*}
  (\hat{\gamma}^{(1)}, \tilde{\gamma}^{(2)}, \ldots, \tilde{\gamma}^{(i-1)}, \hat{\gamma}^{(i)}, \tilde{\gamma}^{(i+1)}, \ldots, \tilde{\gamma}^{(k)}),
\end{equation*}
such that the endpoints of all the paths appearing in the family are
distinct, and there is one less edge appearing in two paths than in
$(\tilde{\gamma}^{(i)})_{1 \leq i \leq k}$.

We keep applying this procedure on the first path of the family until
it is no longer edge-intersecting with any other path of the family.
We then consider the second path, and proceed as previously until all
the paths are edge-disjoint. Notice that this procedure terminates as
there is a finite number of edges that are part of two or more paths.
We end up with a family of $k$ edge-disjoint paths. Note that the
paths thus created are not necessarily simple but contain a simple
path between their endpoints (as their endpoints are distinct).
Replacing each $\hat{\gamma}^{(i)}$ by the simple path it contains yield
the result.

Notice that we have not added any edges in the procedure, thus the
total length of the non edge-intersecting paths thus created is less
than $3k^{2}(2r-1)$. This shows that the required inclusion holds.
\end{proof}
\cref{lem:non-edge-intersecting} then implies
\cref{prop:edge-cyc}.
\begin{proof}[Proof of \cref{prop:edge-cyc}]
  Let $k \geq 1$. \cref{lem:non-edge-intersecting} implies that
  \begin{equation*}
    \P \left( \#\cycV(x) \geq 3k\right) \leq \P \left( \mathrm{ED}_{k, r}(x)\right)\,.
  \end{equation*}

  The union bound then implies
  \begin{equation*}
    \P \left( \mathrm{ED}_{k, r}(x)\right)
    \leq \sum_{\substack{l_{1}, \ldots, l_{k}\geq 1\\\sum_{i} l_{i} \leq 3k^{2}(2r+1)}}\prod_{i=1}^{k} \left( \sum_{y_{1}, \ldots, y_{l_{i}-1}}p_{xy_{1}}p_{y_{1}y_{2}}\cdots p_{y_{l_{i}-2}y_{l_{i}-1}}p_{y_{l_{i-1}}x}\right)\,.
  \end{equation*}
Notice that we have independence of the edges because we ensured
  that the paths are edge-disjoint.

  We have using \eqref{eq:proba-edge} and then \eqref{eq:def-empirical-moment} that
  \begin{equation*}
    \sum_{y_{1}, \ldots, y_{l_{i}-1}}p_{xy_{1}}p_{y_{1}y_{2}}\cdots p_{y_{l_{i}-2}y_{l_{i}-1}}p_{y_{l_{i-1}}x}
    \leq \sum_{y_{1}, \ldots, y_{l_{i}-1}}\frac{w_{x}^{2}w_{y_{1}}^{2}\cdots w_{l_{i-1}}^{2}}{m_{1}^{l_{i}}N^{l_{i}}}
    = \frac{w_{x}^{2}}{m_{1}N} \left( \frac{m_{2}}{m_{1}} \right)^{l_{i}-1}\,.
  \end{equation*}

Assumptions \ref{hyp:upperbd-weights} and \ref{hyp:behavior-m21} then imply
  \begin{equation*}
    \sum_{y_{1}, \ldots, y_{l_{i}-1}}p_{xy_{1}}p_{y_{1}y_{2}}\cdots p_{y_{l_{i}-2}y_{l_{i}-1}}p_{y_{l_{i-1}}x} = \order\p{N^{-\epsilon/2}}
  \end{equation*}
  Finally, we have
  \begin{equation*}
    \P \left( \#\cycV(x) \geq 3k\right) \leq \order\p{N^{-k\epsilon/2}}.
  \end{equation*}
  Choosing $k$ big enough gives the result.
\end{proof}
\subsection{Coupling with a tree}
\label{sec:coupling-with-tree}

The problem of the graph $G^\nc$ is that its edges are not
independent. To solve this problem, we are going to introduce a new
graph $\Tree_x$, whose edges are independent. This graph $\Tree_x$
will be a forest. Actually, we introduce two versions of the forest,
$\Tree_{x}$ and $\check{\Tree}_{x}$ with
$\check{\Tree}_{x} \subset \Tree_{x}$. The reason is that while $\Tree_{x}$
has independent edges and is more convenient, it may have too many
edges for some purposes (see \cref{corol:bound-Dnc} below). The graph
$\check{\Tree}_{x}$ has fewer edges but its edges are independent
conditionally on an appropriate $\sigma$-algebra. The balls of small radius
in the graph $G^{\nc}$ around a fixed vertex $x \in \vertices$ may be
coupled with $\Tree_x$ and $\check{\Tree}_{x}$.

Let us explain briefly the construction of the enlarged graphs
$\Tree_{x}$ and $\check{\Tree}_{x}$. Their vertices are indexed by
families of vertices $\gamma = (x, y_{1}, \ldots, y_{d})$ for any length $d \geq 0$
of $G$. To define the edge set, we introduce Bernoulli random
variables $Z_{\gamma}$ such that there is an edge between
$(x, y_{1}, \ldots, y_{d-1})$ and $\gamma$ in $\Tree_{x}$ (respectively, in
$\check{\Tree}_{x}$) if and only if $Z_{\gamma} = 1$ (respectively,
$\check{Z}_{\gamma} = 1$). For $\Tree_{x}$ and $\check{\Tree}_{x}$ to be
coupled in a convenient way with the balls centered on $x$ in
$G^{\nc}$, we will have to make choice as to which edges of $G$ we
keep in $\Tree_{x}$ and $\check{\Tree}_{x}$. The precise construction
is as follows.

Let $x \in \vertices$. We introduce the family of independent random
variables
$(\hat{Z}_{xy_{1}\cdots y_{k}}; k \geq 2, y_{1}, \ldots, y_{k} \in \vertices)$ such
that for all $k \geq 2$ and $y_{1}, \ldots, y_{k} \in \vertices$,
$\hat{Z}_{xy_{1}\cdots y_{k}}$ is a Bernoulli random variable with
parameter $p_{y_{k-1}y_{k}}$. For $k \geq 2$ and $y_{k} \in S_{k}(x)$ we
define the path
\begin{equation*}
  \gamma^{*}(y_{k}) = \min \left\{ \gamma \in \mathcal{P}^{*}_{xy_{k}}(G) \colon l(\gamma) = k \right\}\,,
\end{equation*}
where the minimum is with respect to the lexicographic order (for the
usual order on $[N]$). Said otherwise, it is the path of length $k$
between $x$ and $y_k$ that is minimal for the lexicographic order. We
are now ready to introduce the variables describing the edges of the
random tree. We define $\check{Z}_{xy_{1}}$ and $Z_{xy_{1}}$ by
\begin{equation*}
  \check{Z}_{xy_{1}} = Z_{xy_{1}} = \ind_{\left\{ x \sim y_{1} \right\}} \quad \text{ for } y_{1} \in \vertices \setminus \left\{ x \right\}\,.
\end{equation*}
and for $k \geq 1$ and $y_{1}, \ldots, y_{k+1} \in \vertices$, $Z_{xy_{1}\cdots y_{k+1}}$ is defined by
\begin{equation*}
  Z_{xy_{1}\cdots y_{k +1}} \deq
  \begin{cases}
    \ind_{\left\{ y_{k} \sim y_{k+1} \right\}} & \text{ if } y_{k+1} \notin B_{k}(x), y_{k} \in S_{k}(x), \text{ and } (x, y_{1}, \ldots, y_{k}) = \gamma^{\star}(y_{k})\\
    \hat{Z}_{xy_{1}\cdots y_{k}y_{k+1}} &\text{ otherwise. }
  \end{cases}
\end{equation*}
The corresponding version for $\check{\Tree}_{x}$ is
\begin{equation*}
  \check{Z}_{xy_{1}\cdots y_{k +1}} \deq
  \begin{cases}
    \ind_{\left\{ y_{k} \sim y_{k+1} \right\}} & \text{ if } y_{k+1} \notin B_{k}(x), y_{k} \in S_{k}(x), \text{ and } (x, y_{1}, \ldots, y_{k}) = \gamma^{\star}(y_{k})\\
    \hat{Z}_{xy_{1}\cdots y_{k}y_{k+1}} & \text{ if } y_{k+1} \notin B_{k}(x), y_{k} \in S_{k}(x), \text{ and } (x, y_{1}, \ldots, y_{k}) \neq \gamma^{\star}(y_{k})\\
    0 &\text{ otherwise. }
  \end{cases}
\end{equation*}
We introduce the filtration $(\mathcal{F}_{k}(x))_{k \geq 1}$ given by
\begin{equation}\label{eq:def-filtration-F}
  \mathcal{F}_{k}(x) = \sigma \left( \left\{ \left\{ \gamma \text{ is a path in $G$} \right\} \colon \gamma \in \bigcup_{i=1}^{k} \left\{ x \right\}\times [N]^{i} \right\} \cup \left\{ \hat{Z}_{\gamma} \colon \gamma \in \bigcup_{i=1}^{k} \left\{ x \right\}\times [N]^{i} \right\} \right)\,,
\end{equation}
i.e.\ it the $\sigma$-algebra generated by the events that paths starting
from $x$ of length smaller that $k$ belong in $G$, and the random
variables $\hat{Z}_{\gamma}$ with $\gamma$ of size at most $k$. Note that given
$\gamma = (x, y_{1}, \ldots, y_{k})$, the event
$\left\{ y_{k} \in S_{k}(x), \gamma = \gamma^{*}(y_{k}) \right\}$ belongs to
$\mathcal{F}_{k}(x)$. Hence, we see that given $k \geq 1$, the random
variables $Z_{xy_{1}\ldots y_{k+1}}, y_{1}, \ldots, y_{k+1} \in \vertices$
(respectively, the random variables
$\check{Z}_{xy_{1}\ldots y_{k+1}}, y_{1}, \ldots, y_{k+1} \in \vertices$) are
independent conditionally to $\mathcal{F}_{k}(x)$. Furthermore, given
$y_{1}, \ldots, y_{k+1} \in \vertices$ and conditionally on
$\mathcal{F}_{k}(x)$,
\begin{itemize}
  \item $Z_{xy_{1}\cdots y_{k+1}}$ is a Bernoulli random variable of
        parameter $p_{y_{k}y_{k+1}}$.
  \item $\check{Z}_{xy_{1}\cdots y_{k+1}}$ is a Bernoulli random
        variable of parameter
        $p_{y_{k}y_{k+1}}\ind_{\left\{ y_{k+1}\notin B_{k}(x), y_k \in S_k(x)\right\}}$.
\end{itemize}

\begin{definition}[Forests $\mathcal{T}_{x}$ and $\check{\Tree}_{x}$]\label{def:tree-Tx}
  The graph $\Tree_{x}$ (respectively $\check{\Tree}_{x}$) is the
  graph
  \begin{itemize}
    \item with vertex set
          $V_{x} = \{x\} \cup \bigcup_{d \geq 1} \left\{ x \right\}\times \vertices^{d}$,
          and
    \item such that given
          $\gamma = (x, y_{1}, \ldots, y_{d})\in V_{x}$ and
          $\gamma' = (x, y'_{1}, \ldots, y'_{d'}) \in V_{x}$ with
          $d \leq d'$, we have
          $\left\{ \gamma, \gamma' \right\} \in \Tree_{x}$ if and only
          if $d' = d + 1$, $y_{1}=y'_{1}, \ldots, y_{d} = y'_{d}$, and
          $Z_{xy'_{1}\cdots y'_{d'}} = 1$ (respectively
          $\check{Z}_{xy'_{1}\cdots y'_{d'}} = 1$).
  \end{itemize}
\end{definition}
Note that both $\check{\Tree}_{x}$ and $\Tree_x$ are infinite forests,
while the connected components of $\check{\Tree}_x$ and $\Tree_{x}$
containing $x$ are trees naturally rooted at $x$. We indicate by an
exponent $\Tree_{x}$ (or $\check{\Tree}_{x}$) the spheres, degrees, etc.\ associated with $\Tree_{x}$ (or $\check{\Tree}_{x}$). For instance, we
write $D^{\Tree_{x}}_{\gamma}$ the degree of $\gamma \in V_{x}$ in $\Tree_{x}$.
For convenience, we also define for all
$\gamma = (x, y_{1}, \ldots, y_{d}, y_{d+1}) \in V_{x}$ the number of children of
$\gamma$:
\begin{equation}\label{eq:def-Dup}
  D^{\Tree_{x}\uparrow}_{\gamma} \deq D^{\Tree_{x}}_{\gamma} - \ind_{\bigl\{\gamma \neq (x) \simT{x} \gamma\bigr\}}\,.
\end{equation}

\begin{remark}
  Introducing the graph $\Tree_x$ is motivated by the fact that the
  edges of $G^\nc$ are not independent. At the cost of introducing a
  small number of additional edges, the graph $\check{\Tree}_x$ has
  edges at depth $k$ that are independent conditionally to the
  $\sigma$-algebra $\mathcal{F}_k(x)$. The graph $\Tree_{x}$ enjoys a
  stronger property: the edges of $\Tree_{x}$ are all independent as
  we show in \cref{lem:T-indep}.
\end{remark}

\begin{lemma}\label{lem:T-indep}
  The random variables $Z_{\gamma}$ for $\gamma \in V_{x}$ are all independent.
\end{lemma}
\begin{proof}
  It suffices to show that for every $k \geq 2$ and
  $\gamma_{1}, \ldots, \gamma_{k} \in V_{x}$ distinct vertices of $\Tree_{x}$, we have
  \begin{equation*}
    \P \Bigl(Z_{\gamma_{1}} = 1; \cdots; Z_{\gamma_{k}} = 1\Bigr) = \prod_{i=1}^{k}p_{y_{i,d_{i}}y_{i,d_{i}+1}}\,,
  \end{equation*}
  where for all $i$,
  $\gamma_{i} = (x, y_{i,1}, \ldots, y_{i, d_{i}}, y_{i, d_{i}+1})$. We may
  assume that $d_{1} \leq d_{2} \leq \cdots \leq d_{k}$, and that
  $d_{p} = d_{p+1} = \cdots = d_{k}$ for some $1 \leq p \leq k$. Since
  $\bigl\{Z_{\gamma_{1}} = 1; \cdots; Z_{\gamma_{p-1}} = 1\bigr\}$ is
  $\mathcal{F}_{d_{k}}(x)$-measurable, we have by conditional
  independence
  \begin{equation*}
    \P \Bigl(Z_{\gamma_{1}} = 1; \cdots; Z_{\gamma_{k}} = 1 \mid \mathcal{F}_{d_{k}}(x)\Bigr)
    = \ind_{\bigl\{Z_{\gamma_{1}} = 1; \cdots; Z_{\gamma_{p-1}} = 1\bigr\}}\prod_{i=p}^{k}p_{y_{i,d_{i}}y_{i,d_{i}+1}}\,.
  \end{equation*}
  Proceeding by induction, we get the result.
\end{proof}

\begin{remark}
  The forests $\Tree_{x}$ and $\check{\Tree}_{x}$ are constructed in
  such a way that if $\gamma$ and $\gamma$ are two distinct elements of $V_{x}$
  of the same length $d = d(x, y)$ from $x$ to $y$ that are present in
  $G$, we have that
  \begin{itemize}
    \item conditionally to $\mathcal{F}_{d}(x)$,
          $D^{\Tree_{x}\uparrow}_{\gamma}$ and
          $D^{\Tree_{x}\uparrow}_{\gamma'}$ are independent and
          identically distributed;
    \item $D^{\Tree_{x}\uparrow}_{\gamma}$ and
          $D^{\Tree_{x}\uparrow}_{\gamma'}$ are independent and
          identically distributed.
  \end{itemize}
\end{remark}

\begin{lemma}\label{lem:choice-gamma-star}
  Let $y \in B_{r}^{\nc}(x)$. Then, there exists a unique path
  $\gamma \in \mathcal{P}^{*}_{xy}(G)$ with $l(\gamma) \leq r$. Furthermore,
  $y \in S_{l(\gamma)}^{\nc}(x)$ and $\gamma = \gamma^{*}(y)$.
\end{lemma}
\begin{proof}
  The fact that there exists a unique path
  $\gamma \in \mathcal{P}^{*}_{xy}(G)$ with $l(\gamma) \leq r$ follows
  by construction of $G^{\nc}$. Indeed, if there existed
  $\gamma' \in \mathcal{P}^{*}_{xy}(G)$ with $\gamma' \neq \gamma$, we could
  consider the concatenation $\tilde{\gamma}$ of $\gamma$ and $\gamma'$ with its
  order reversed and construct a path contained in $B_{r}(x)$ from $x$
  to $x$. While $\tilde{\gamma}$ would not be necessarily simple, there
  would exists $x'$ such that $\tilde{\gamma}$ contains a simple path in
  $\mathcal{P}^{*}_{x'x'}(G)$ which is contained in $B_{r}(x')$.
  Hence, one vertex $y'$ in $\gamma$ would belong to $\cycV(x')$ and thus
  the edge $\left\{ x', y' \right\}$ would not appear in $G^{\nc}$.
  Hence, the subset of $\mathcal{P}^{*}_{xy}(G^{\nc})$ comprising
  paths of length at most $r$ must be empty, contradicting
  $y \in B_{r}^{\nc}$.

  Hence, we get that $d(x, y) = l(\gamma)$, i.e.\ $y \in S_{l(\gamma)}(x)$ and
  $y \in S_{l(\gamma)}^{\nc}(x)$. It then follows that $\gamma = \gamma^{*}(y)$.
\end{proof}

The graph $G^{\nc}\vert_{B^{\nc}_{r}(x)}$ can be embedded in
$\check{\Tree}_{x}$, and hence in $\Tree_{x}$, as follows. We define a
mapping $\iota \colon B^{\nc}_{r}(x) \to V_{x}$ by
\begin{equation*}
  \iota(x) = (x)\,.
\end{equation*}
Then, for all $y \in B^{\nc}_{r}(x) \setminus \left\{ x \right\}$ we use
\cref{lem:choice-gamma-star} and set
\begin{equation*}
  \iota(y) = \gamma^{*}(y)\,.
\end{equation*}
\begin{proposition}\label{prop:embedding-Tx}
  The mapping $\iota$ defines a graph embedding of
  $G^{\nc}\vert_{B^{\nc}_{r}(x)}$ into $\check{\Tree}_{x}$ and
  $\Tree_{x}$.
\end{proposition}
\cref{prop:embedding-Tx} will be used several times in the sequel. An
important implication of this result is the following corollary.
\begin{corollary}\label{corol:bound-Dnc}
  Let $y \in B_{r-1}^{\nc}(x)$. Then, we have
  \begin{equation*}
    D_{y}^{\nc} \leq D^{\check{\Tree}_{x}}_{\iota(y)} \leq D^{\Tree_{x}}_{\iota(y)}\,.
  \end{equation*}
  Furthermore, there exists a constant $C_{\nu} > 0$ such that with $\nu$-high probability
  \begin{equation*}
    D^{\check{\Tree}_{x}}_{\iota(y)} \leq D_{y}^{\nc} + C_{\nu}\,.
  \end{equation*}
\end{corollary}
\begin{proof}[Proof of \cref{corol:bound-Dnc}]
  Let $y \in B_{r-1}^{\nc}(x)$. The first inequality is a direct consequence of
  \cref{prop:embedding-Tx}: since $\iota$ is a graph embedding the set
  of neighbors of $\iota(y)$ in $B_{r-1}^{\nc}(x)$ contains the image
  of the set of neighbors of $y$.

  Let us prove the second inequality. By \cref{lem:choice-gamma-star}, we have that
  $d \deq d(x, y) = l(\gamma^{*}(y))$. Hence, $y \in S_{d}(x)$. We have
  \begin{equation*}
    D^{\check{\Tree}_{x}}_{\iota(y)} = \ind_{\{( \gamma^{*}(y))_{d-1} \sim y \}} + \sum_{z \notin B_{d}(x)} \ind_{\{ y \sim z \}} \leq D_{y} \leq D^{\nc}_{y} + C_{\nu}\,,
  \end{equation*}
  where the last inequality holds with $\nu$-high probability by
  \cref{prop:edge-cyc}.
\end{proof}

\begin{proof}[Proof of \cref{prop:embedding-Tx}]
  It is clear that $\iota$ is injective. Let
  $\left\{ y, y' \right\} \in G^{\nc}\vert_{B_{r}^{\nc}(x)}$, and let
  us show that $\left\{ \gamma^{*}(y), \gamma^{*}(y') \right\} \in \Tree_{x}$. We can
  assume without loss of generality that $d \deq l(\gamma^{*}(y)) \leq l(\gamma^{*}(y'))$. Since
  $G^{\nc}\vert_{B_{r}^{\nc}(x)}$ is a tree,
  \begin{equation*}
    \gamma^{*}(y') = \Bigl(x, \gamma^{*}(y)_{1}, \ldots, \gamma^{*}(y)_{l(\gamma^{*}(y))}, y'\Bigr)\,,
  \end{equation*}
  and in particular $l(\gamma^{*}(y')) = d+1$. It remains to show that
  $Z_{\gamma^{*}(y')} = \check{Z}_{\gamma^{*}(y')} = 1$. This follows from
  \cref{lem:choice-gamma-star}: indeed, we have $y \in S_{d}(x)$ and
  $y' \notin B_{d}(x)$ so
  $Z_{\gamma^{*}(y')} = \check{Z}_{\gamma^{*}(y')} = \ind_{\left\{ y \sim y' \right\}} = 1$.
\end{proof}

\subsection{Removing the down-up paths}
\label{sec:removing-down-up}

As explained in \cref{sec:pruning-procedure}, step \ref{itm:remove-du} of the pruning
procedure consists in removing the edges $\left\{ x, y \right\} \in G$
where $x \in \vertices$ and $y \in \Sdu(x) \setminus \cycV(x)$. When removing such
an edge, we change the degree of both $x$ and $y$. Controlling how
these two degrees change is the content of Lemmas \ref{lem:bound-ud-paths} and \ref{lem:bound-D+} below.
More precisely, if we consider a vertex $x \in \vertices$ and
$y \in S_{1}(x) \setminus \cycV(x)$, then
\begin{itemize}
  \item either $y \prec x$, and if $y \in \Sdu(x)$, we remove $\left\{ x, y \right\}$ from $G$;
  \item or $y \succ x$, and if there exists $y' \in S_{1}^{+}(x) \setminus \cycV(x)$ with $y \prec y'$, we remove $\left\{ x, y \right\}$ from $G$.
\end{itemize}
It implies that
\begin{equation*}
  D_{x}^{\pp} = D_{x} - \# \cycV(x) - \#\Sdu(x) - \left( \# \left( S_{1}^{+}(x) \setminus \cycV(x) \right) - 1 \right) \vee 0\,.
\end{equation*}
We have upper bounded the second term in \cref{prop:edge-cyc}. \cref{lem:bound-ud-paths} bounds the third term, and
\cref{lem:bound-D+} bounds the fourth one, as $\# S_1^+(x) = D_x^+$.
\begin{lemma}\label{lem:bound-ud-paths}
  Let $\nu > 0$, $c > 1$, and $x \in \vertices$. With \(\nu\)-high
  probability,
  \begin{equation*}
    \# \Sdu(x) \leq c\frac{\nu}{1-2\delta}\frac{\log N}{\log\log N}\,.
  \end{equation*}
\end{lemma}

\begin{lemma}\label{lem:bound-D+}
  Let $x \in [N]$, $c > 1$, and $\nu > 0$. Introduce
  \begin{equation}\label{eq:def-Dnc+}
    D^{\nc +}_{x} \deq \# \left( S_{1}^{+}(x)\setminus \cycV(x) \right)\,.
  \end{equation}
  With $\nu$-high probability,
  \begin{equation*}
    D_{x}^{\nc +} \leq c\frac{\nu}{1-\delta}\frac{\log N}{\log \log N}\,.
  \end{equation*}
\end{lemma}

Thus, Lemmas \ref{lem:bound-ud-paths} and \ref{lem:bound-D+} show that we remove roughly
$\log N/\log\log N$ edges around each vertex when removing the down-up
paths.

\begin{proof}[Proof of \cref{lem:bound-ud-paths}]
  We start by noticing that if $D_{x} < \frac{\log N}{\log\log N}$ the
  result is immediate as $\# \Sdu(x) \leq D_{x}$. We thus assume that
  $D_{x} \geq \frac{\log N}{\log\log N}$. Assuming this, we introduce
  $\chi = 4\nu\log \log N$ and use \cref{rem:ratio-big-dx} to get
  $\chi D_{x} \geq w_{x}$ with $\nu$-high probability. This gives us
  \begin{equation*}
    \# \Sdu(x)
    = \sum_{y \in \vertices} \ind_{\bigl\{x \simnc y\bigr\}} \ind_{\bigl\{ \exists z \neq x, y \simnc z, D_{z} \geq D_{x} \geq D_{y}\}}
    \leq \sum_{y \in \vertices} \ind_{\bigl\{x \simnc y\bigr\}} \ind_{\bigl\{ \exists z \neq x, y \simnc z, \chi D_{z} \geq w_{x}\}}\,,
  \end{equation*}
  with $\nu$-high probability.

  We now relate the quantity $\# \Sdu(x)$ to the tree $\Tree_{x}$. To
  do so we use \cref{prop:edge-cyc}: there exists a constant
  $C_{\nu} > 0$ such that with $\nu$-high probability
  \begin{equation*}
    \# \Sdu(x)
    \leq \sum_{y \in \vertices} \ind_{\bigl\{x \simnc y\bigr\}} \ind_{\bigl\{ \exists z \neq x, y \simnc z, \chi \bigl(D^{\nc}_{z}+C_{\nu}\bigr) \geq  w_{x}\bigr\}}\,.
  \end{equation*}
  \cref{prop:embedding-Tx} implies that
  \begin{equation*}
    \# \Sdu(x) \leq \sum_{y \in \vertices} \ind_{\bigl\{x \simT{x} y\bigr\}} \ind_{\bigl\{ \exists z \neq x, y \simT{x} z, \chi \bigl(D^{\Tree_{x}}_{z}+C_{\nu}\bigr) \geq w_{x}\bigr\}}\,.
  \end{equation*}
  To get the result, we are going to use Bennett's inequality
  \cite[Theorem 2.9]{BLM13}. The Bernoulli random variables
  \begin{equation*}
    \Bigl(\ind_{\bigl\{x \simT{x} y\bigr\}} \ind_{\bigl\{ \exists z \neq x, y \simT{x} z, \chi\bigl(D^{\Tree_{x}}_{z}+C_{\nu}\bigr) \geq  w_{x}\bigr\}}\Bigr)_{y \in \vertices}
  \end{equation*}
  are independent, since by \cref{lem:T-indep} the edges in
  $\Tree_{x}$ are independent. We compute
  \begin{equation*}
    v
    \deq \sum_{y \in \vertices}\E \left[ \ind_{\bigl\{x \simT{x} y\bigr\}} \ind_{\bigl\{ \exists z \neq x, y \simT{x} z, \chi\bigl(D^{\Tree_{x}}_{z}+C_{\nu}\bigr) \geq w_{x}\}} \right]
    = \sum_{y \in \vertices}p_{xy} \P \left( \exists z \neq x, y \simT{x} z, \chi \bigl(D^{\Tree_{x}}_{z}+C_{\nu}\bigr) \geq w_{x} \right)\,.
  \end{equation*}
  The union bound and Markov's inequality yield
  \begin{equation*}
    \begin{split}
      \P \left( \exists z \neq x, y \simT{x} z, \chi\bigl(D^{\Tree_{x}}_{z}+C_{\nu}\bigr) \geq  w_{x} \right)
      &\leq \sum_{z \neq x}p_{yz} \P \left( \chi\bigl(D^{\Tree_{x}}_{z} - \ind_{\bigl\{ y \simT{x} z \bigr\}}+1+C_{\nu}\bigr) \geq w_{x} \right)\\
      &\leq \sum_{z \neq x}p_{yz} \chi\frac{d_{z} +1 + C_{\nu}}{d_{x}}\,.
    \end{split}
  \end{equation*}
  Using \eqref{eq:def-empirical-moment}, \eqref{eq:proba-edge}, and
  the crude bound $w_{x}/2 \leq d_{x} \leq w_{x}$, we get
  \begin{equation*}
    \P \left( \exists z \neq x, y \simT{x} z, \chi\bigl(D^{\Tree_{x}}_{z}+C_{\nu}\bigr) \geq d_{x} \right)
    \leq \sum_{z} \frac{w_{y}w_{z}^{2} + (1+C_{\nu})w_{y}w_{z}}{m_{1}N} \frac{\chi}{d_{x}}
    \leq 2\chi\frac{w_{y}}{w_{x}} \frac{m_{2} + (C_{\nu}+1)m_{1}}{m_{1}}\,.
  \end{equation*}
  Using \eqref{eq:def-empirical-moment} and \eqref{eq:proba-edge} again, we get
  \begin{equation*}
    v \leq 2\chi \frac{m_{2}}{m_{1}} \left(\frac{m_{2}}{m_{1}} + C_{\nu} + 1\right)\,.
  \end{equation*}
  Bennett's inequality then implies the result:
  \begin{equation*}
    \begin{split}
      \P \Bigl( \# \Sdu &\geq \frac{c\nu}{1-2\delta} \frac{\log N}{\log\log N} \Bigr)\\
                          &\leq\P \Bigl( \sum_{y \in \vertices} \ind_{\bigl\{x \simT{x} y\bigr\}} \ind_{\bigl\{ \exists z \neq x, y \simT{x} z, \chi\bigl(D^{\Tree_{x}}_{z}+C_{\nu}\bigr) \geq w_{x}\}}\geq \frac{c \nu}{1-2\delta} \frac{\log N}{\log\log N} \Bigr) + \order\p{N^{-\nu}}\\
      &\leq \exp\pB{-(v + \frac{c \nu}{1 - 2\delta} \frac{\log N}{\log\log N})\ln(1 + \frac{c \nu \log N}{v(1-2\delta)\log\log N})(1 + o(1))} + \order\p{N^{-\nu}}\\
      &= \exp\p{-c\nu\log N(1 + o(1))} + \order\p{N^{-\nu}} = \order\p{N^{-\nu}}\,. \qedhere
    \end{split}
  \end{equation*}
\end{proof}

\begin{proof}[Proof of \cref{lem:bound-D+}]
  Notice that if $D_{x} < \frac{\log N}{\log\log N}$ the result is
  immediate. Hence, we assume that
  $D_{x} \geq \frac{\log N}{\log\log N}$. We introduce
  $\chi = 4\nu \log \log N$. By \cref{rem:ratio-big-dx} we have that
  with $\nu$-high probability that $\chi D_{x} \geq d_{x}$.

  Consider the random variable
  \begin{equation*}
    D^{\nc +}_{x} = \# \left( S_{1}^{+}(x)\setminus \cycV(x) \right) = \sum_{y \in \vertices}\ind_{\{x \simnc y, D_{x}\leq D_{y}\}}\,.
  \end{equation*}
  By the previous discussion, we have with $\nu$-high probability
  \begin{equation*}
    D^{\nc +}_{x} \leq \sum_{y \in \vertices}\ind_{\bigl\{x \simnc y, \chi D_{y} \geq d_{x}\bigr\}}\,.
  \end{equation*}
  By \cref{prop:edge-cyc}, there exists a constant $C_{\nu}$
  such that with $\nu$-high probability, we have
  \begin{equation*}
    D^{\nc +}_{x} \leq \sum_{y \in \vertices}\ind_{\bigl\{x \simnc y, \chi\bigl(D_{y}^{\nc} + C_{\nu}\bigr) \geq d_{x}\bigr\}}\,.
  \end{equation*}
  \cref{prop:embedding-Tx} allow us to bound this using the tree
  $\Tree_{x}$:
  \begin{equation*}
    D^{\nc +}_{x} \leq D^{\Tree_{x}+}_{x} \deq \sum_{y \in V_{x}}\ind_{\bigl\{x \simT{x} y, \chi\bigl(D_{y}^{\Tree_{x}\uparrow} + 1 + C_{\nu}\bigr) \geq d_{x}\bigr\}}\,.
  \end{equation*}

  As in the proof of \cref{lem:bound-ud-paths}, we conclude using
  Bennett's inequality (see \cite[Theorem 2.9]{BLM13}). The random variables
  \begin{equation*}
    \left( \ind_{\bigl\{x \simT{x} y\bigr\}}\ind_{\bigl\{\chi\bigl(D_{y}^{\Tree_{x}\uparrow} +1 + C_{\nu}\bigr) \geq d_{x}\bigr\}} \right)_{y \in V_{x}}
  \end{equation*}
  are independent by \cref{lem:T-indep} and bounded by $1$. Markov's
  inequality gives us:
  \begin{equation*}
      v
      \deq \sum_{y \in V_{x}} \E\Bigl[\ind_{\bigl\{x \simT{x} y, \chi\bigl(D_{y}^{\Tree_{x}\uparrow} +1 + C_{\nu}\bigr) \geq d_{x}\bigr\}}\Bigr]
      \leq \sum_{y \in V_{x}} p_{xy} \chi\frac{d_{y} + 1 + C_{\nu}}{d_{x}}\,.
  \end{equation*}
  Using \eqref{eq:def-empirical-moment}, \eqref{eq:proba-edge}, and
  the crude bound $w_{x}/2 \leq d_{x} \leq w_{x}$, we get
  \begin{equation*}
    v \leq 2\chi \frac{m_{2} + (C_{\nu}+1)m_{1}}{m_{1}}\,.
  \end{equation*}
  Bennett's inequality then implies
  \begin{equation*}
      \P \left( D^{\nc +}_{x} \leq \frac{c \nu}{1 - \delta}\frac{\log N}{\log\log N} \right)
      \leq \P \left( D^{\Tree_{x} +}_{x} \leq \frac{c \nu}{1 - \delta}\frac{\log N}{\log\log N} \right) + \order\p{N^{-\nu}}
      \leq \order\p{N^{-\nu}}\,.\qedhere
  \end{equation*}
\end{proof}

With the preceding results, \cref{prop:edge-cyc} and Lemmas \ref{lem:bound-ud-paths} and \ref{lem:bound-D+}, we are ready to prove \cref{prop:pruning}.
\begin{proof}[Proof of \cref{prop:pruning}.]
  Let $x \in \vertices$. During the pruning procedure, we remove at most
  \begin{equation*}
    D_{x}^{\nc +} + \# \Sdu(x) + \#\cycV(x)
  \end{equation*}
  edges around $x$. \cref{prop:edge-cyc} together with Lemmas \ref{lem:bound-D+} and \ref{lem:bound-ud-paths} yield
  claim \ref{item:diff-Dx}, since for every $2 < c \leq 3$ we have
  \begin{equation*}
    c\frac{\nu}{1 - 2\delta}\frac{\log N}{\log \log N} + c \frac{\nu}{1 - \delta} \frac{\log N}{\log\log N} = \frac{c \nu (2 - 3\delta)}{(1 - \delta)(1 - 2\delta)}\frac{\log N}{\log\log N} \leq \xi / 2\,.
  \end{equation*}
  The two other claims are consequence of the construction of
  $G^{\pp}$. Let us detail the third one. Assume that there is a
  simple loop in $G^{\pp}$ composed of the vertices
  $(\gamma_{0}, \gamma_{1}, \ldots, \gamma_{k} = \gamma_{0})$ with $\{\gamma_{i-1}, \gamma_{i}\} \in G^{\pp}$
  for all $i \in [k]$. Then there is a vertex, say $\gamma_{0}$, which is
  minimal for the total order $\prec$. Then
  $\gamma_{k-1} \succ \gamma_{k} = \gamma_{0} \prec \gamma_{1}$, and either
  $(\gamma_{k-1}, \gamma_{0}, \gamma_{1})$ or $(\gamma_{1}, \gamma_{0}, \gamma_{k-1})$ is a down-up
  path. As there is no such path in $G^{\pp}$, there are no cycle in
  $G^{\pp}$.
\end{proof}

\subsection{Estimate of $\|A - A^{\pp}\|$}
\label{sec:estimate-a-axi}

We now give estimates for the error we make when working with the
adjacency matrix of the pruned graph $A^{\pp}$ rather than the adjacency
matrix of the original graph $A$.

\begin{proposition}\label{prop:error-pruning}
  Let $\nu > 0$. There exists a constant $C_{\nu} > 0$ such that with
  $\nu$-high probability,
  \begin{equation*}
      \|A - A^{\pp}\| \leq C_{\nu} \sqrt{\frac{\log N}{\log\log N}}\,.
  \end{equation*}
\end{proposition}
The proof of \cref{prop:error-pruning} relies on
\cref{lem:bound-op-high} and \cref{lem:A-intermediate} stated below.
They are stated using a partition of the set of vertices
$\vertices = \hcV^{(\rm l)} \sqcup \hcV^{(\rm i)} \sqcup \hcV^{(\rm h)}$,
where
\begin{equation}\label{eq:def-set-vertices}
  \begin{aligned}
    \hcV^{(\rm l)}_{\nu} &= \left\{ x \in \vertices \col D_{x} < \xi_{\nu}, w_{x} \leq 4\xi_{\nu} \right\}&\quad \text{ (vertices of low degree) }\\
    \hcV^{(\rm i)}_{\nu} &= \left\{ x \in \vertices \col D_{x} < \xi_{\nu}, w_{x} > 4\xi_{\nu} \right\}&\quad \text{ (vertices of intermediate degree) }\\
    \hcV^{(\rm h)}_{\nu} &= \left\{ x \in \vertices \col \xi_{\nu} \leq D_{x} \right\}&\quad \text{ (vertices of high degree) }\,.
  \end{aligned}
\end{equation}
Recall that the threshold $\xi_{\nu}$ was defined in \cref{prop:pruning}.
The reason why we partition $[N]$ into three set of vertices will
become apparent in the proof of \cref{prop:error-pruning}. The key
point is that treating the vertices belonging in different sets
$\hcV_{\nu}^{(\circ)}, \circ \in \left\{ \rm l, \rm i, \rm h \right\}$ requires
different techniques. In particular, treating vertices in
$\hcV^{(\rm h)}_{\nu}$ requires using in a fine way the properties of
down-up paths, while treating vertices in $\hcV^{(\rm l)}_{\nu}$ is done
by using results in \cite{BGBK} about sparse random matrices with small
weights. The two lemmas concerning vertices of high and intermediate
degree are the following.
\begin{lemma}\label{lem:bound-op-high}
  Let $\nu > 0$. Then, for all
  $x_{1} \in \hcV_{\nu}^{(\rm h)}$,
  \begin{equation*}
    \sum_{\substack{x_{2}, x_{3} \in \hcV_{\nu}^{(\rm h)}\\x_{2} \succ x_{1}, x_{3}}} \left< \b 1_{\Sdu(x_{1})}, \b 1_{\Sdu(x_{2})} \right>\left< \b 1_{\Sdu(x_{2})\setminus \Sdu(x_{1})}, \b 1_{\Sdu(x_{3})} \right> \leq 2\nu \left( \frac{\log N}{\log\log N} \right)^{2}\,,
  \end{equation*}
  with $\nu$-high probability.
\end{lemma}

\begin{lemma}\label{lem:A-intermediate}
  For any $\nu > 0$ and
  $x_{1} \in \hcV^{(\rm i)}_{\nu}$ we have
  \begin{equation*}
    \sum_{\substack{x_{2} \in \hcV^{(\rm i)}_{\nu}\\x_{2} \neq x_{1}}} \left< \b 1_{\Sdu(x_{1})}, \b 1_{\Sdu(x_{2})} \right> \leq 2\nu \frac{\log N}{\log\log N}
  \end{equation*}
  with $\nu$-high probability.
\end{lemma}

The two previous lemmas are based on probabilistic estimates, while the proof of \cref{prop:error-pruning} below contains mainly algebraic arguments.

\begin{proof}[Proof of \cref{prop:error-pruning}]
  The matrix $A - A^{\pp}$ is the adjacency matrix of the graph made
  of the edges removed during the pruning. Consider the adjacency
  matrix $\Anc$ of the graph $\Gnc$ obtained after step
  \ref{itm:remove-cycle} of the pruning procedure \cref{def:pruning}. By \cref{prop:edge-cyc}, the maximum degree
  of a vertex in the graph described by $A - \Anc$ is bounded by a
  constant $C_{\nu}$, with $\nu$-high probability. It implies that
  $\|A -\Anc\| \leq C_{\nu}$, with $\nu$-high probability. Thus, it
  suffices to bound $\|\Anc - A^{\pp}\|$.

  Let us introduce for convenience the matrices
  \begin{equation*}
     \tilde{A}_{\circ} = \sum_{x \in \hcV^{(\circ)}}\b 1_{x}\b 1_{\Sdu(x)}^{*} \quad \text{ for } \circ \in \left\{ \rm l, \rm i, \rm h \right\}\,,
  \end{equation*}
  so that
  \begin{equation*}
    A^{\pp} - \Anc = \sum_{\circ \in \left\{ \rm l, \rm i, \rm h \right\}}\tilde{A}_{\circ} + \tilde{A}_{\circ}^{*}\,.
  \end{equation*}
  Let us now bound the norms of the operators $\tilde{A}_{\circ}$,
  starting with $\circ = \rm h$. Let us consider the matrix
  $\tilde{A}_{\rm h}\tilde{A}_{\rm h}^{*}$. We have
  \begin{equation*}
    \begin{split}
      \tilde{A}_{\rm h}\tilde{A}_{\rm h}^{*}
      &= \sum_{x_{1}, x_{2} \in\hcV_{\nu}^{(\rm h)}} \left< \b 1_{\Sdu(x_{1})}, \b 1_{\Sdu(x_{2})} \right> \b 1_{x_{1}} \b 1_{x_{2}}^{*}\\
      &= \sum_{x \in\hcV_{\nu}^{(\rm h)}} \left( \# \Sdu(x) \right) \b 1_{x} \b 1_{x}^{*} + \sum_{\substack{x_{1}, x_{2} \in\hcV_{\nu}^{(\rm h)}\\x_{1} \neq x_{2}}}\left< \b 1_{\Sdu(x_{1})}, \b 1_{\Sdu(x_{2})} \right> \b 1_{x_{1}} \b 1_{x_{2}}^{*}\,.
    \end{split}
  \end{equation*}
  The first term, a diagonal matrix, has its operator norm bounded by
  $\frac{2(\nu + 1)}{1-2\delta} \frac{\log N}{\log\log N}$ with $\nu$-high probability
  by \cref{lem:bound-ud-paths}. Let us concentrate on the second term,
  which we call $B$. We introduce $B_{\succ}$:
  \begin{equation*}
    B_{\succ} \deq \sum_{\substack{x_{1}, x_{2} \in\hcV_{\nu}^{(\rm h)}\\x_{1} \succ x_{2}}}\left< \b 1_{\Sdu(x_{1})}, \b 1_{\Sdu(x_{2})} \right> \b 1_{x_{1}} \b 1_{x_{2}}^{*} = \sum_{\substack{x_{1}, x_{2} \in\hcV_{\nu}^{(\rm h)}\\x_{1} \succ x_{2}}} \left( \# \Sdu(x_{1}) \cap \Sdu(x_{2}) \right) \b 1_{x_{1}} \b 1_{x_{2}}^{*}\,,
  \end{equation*}
  so that $B = B_{\succ} + B_{\succ}^{*}$.
  We then have
  \begin{equation*}
    \begin{split}
      B_{\succ}B_{\succ}^{*}
      &= \sum_{\substack{x_{1}, x_{2}, x_{3} \in\hcV_{\nu}^{(\rm h)}\\x_{2} \prec x_{1}, x_{3}}}\left< \b 1_{\Sdu(x_{1})}, \b 1_{\Sdu(x_{2})} \right>\left< \b 1_{\Sdu(x_{2})}, \b 1_{\Sdu(x_{3})} \right> \b 1_{x_{1}} \b 1_{x_{3}}^{*}\\
      &= \sum_{\substack{x_{1}, x_{2}, x_{3} \in\hcV_{\nu}^{(\rm h)}\\x_{2} \prec x_{1}, x_{3}}}\left< \b 1_{\Sdu(x_{1})}, \b 1_{\Sdu(x_{2})} \right>\left< \b 1_{\Sdu(x_{2})\setminus\Sdu(x_{1})}, \b 1_{\Sdu(x_{3})} \right> \b 1_{x_{1}} \b 1_{x_{3}}^{*}\\&\quad+\sum_{\substack{x_{1}, x_{2}, x_{3} \in\hcV_{\nu}^{(\rm h)}\\x_{2} \prec x_{1}, x_{3}}}\left< \b 1_{\Sdu(x_{1})}, \b 1_{\Sdu(x_{2})} \right>\left< \b 1_{\Sdu(x_{1})\cap\Sdu(x_{2})}, \b 1_{\Sdu(x_{3})} \right> \b 1_{x_{1}} \b 1_{x_{3}}^{*}\,.
    \end{split}
  \end{equation*}
  We thus have
  \begin{equation}\label{eq:prop-norm-intermediate}
    \begin{split}
      \| B_{\succ} \|^{2}
      &\leq \max_{\| \b u \| = 1}\sum_{\substack{x_{1}, x_{2}, x_{3} \in\hcV_{\nu}^{(\rm h)}\\x_{2} \prec x_{1}, x_{3}}}\left< \b 1_{\Sdu(x_{1})}, \b 1_{\Sdu(x_{2})} \right>\left< \b 1_{\Sdu(x_{2})\setminus\Sdu(x_{1})}, \b 1_{\Sdu(x_{3})} \right> u_{x_{1}} u_{x_{3}}\\
      &\quad+\max_{\| \b u \| = 1}\sum_{\substack{x_{1}, x_{2}, x_{3} \in\hcV_{\nu}^{(\rm h)}\\x_{2} \prec x_{1}, x_{3}}}\left< \b 1_{\Sdu(x_{1})}, \b 1_{\Sdu(x_{2})} \right>\left< \b 1_{\Sdu(x_{1})\cap\Sdu(x_{2})}, \b 1_{\Sdu(x_{3})} \right>u_{x_{1}} u_{x_{3}}\,.
    \end{split}
  \end{equation}
  The first term of \eqref{eq:prop-norm-intermediate} can be bounded using Young's inequality:
  \begin{equation*}
    \begin{split}
      \max_{\| \b u \| = 1}\sum_{\substack{x_{1}, x_{2}, x_{3} \in\hcV_{\nu}^{(\rm h)}\\x_{2} \prec x_{1}, x_{3}}}&\left< \b 1_{\Sdu(x_{1})}, \b 1_{\Sdu(x_{2})} \right>\left< \b 1_{\Sdu(x_{2})\setminus\Sdu(x_{1})}, \b 1_{\Sdu(x_{3})} \right> u_{x_{1}} u_{x_{3}}\\
      &\leq \max_{\| \b u \| = 1}\sum_{x_{1} \in \hcV^{(\rm h)}}u_{x_{1}}^{2}\sum_{\substack{x_{2}, x_{3} \in\hcV_{\nu}^{(\rm h)}\\x_{2} \prec x_{1}, x_{3}}}\left< \b 1_{\Sdu(x_{1})}, \b 1_{\Sdu(x_{2})} \right>\left< \b 1_{\Sdu(x_{2})\setminus\Sdu(x_{1})}, \b 1_{\Sdu(x_{3})} \right>\,.
    \end{split}
  \end{equation*}
  Then, \cref{lem:bound-op-high} allows to bound this by
  $2\nu \left( \frac{\log N}{\log\log N} \right)^{2}$. The second term of \eqref{eq:prop-norm-intermediate}
  can be treated as follows:
  \begin{equation*}
    \begin{split}
      \max_{\| \b u \| = 1}\sum_{\substack{x_{1}, x_{2}, x_{3} \in\hcV_{\nu}^{(\rm h)}\\x_{2} \prec x_{1}, x_{3}}}&\left< \b 1_{\Sdu(x_{1})}, \b 1_{\Sdu(x_{2})} \right>\left< \b 1_{\Sdu(x_{1})\cap\Sdu(x_{2})}, \b 1_{\Sdu(x_{3})} \right>u_{x_{1}} u_{x_{3}}\\
      &= \max_{\| \b u \| = 1}\sum_{\substack{x_{1}, x_{2}, x_{3} \in\hcV_{\nu}^{(\rm h)}\\x_{2} \prec x_{1}, x_{3}}} \sum_{y_{1}, y_{2}} \ind_{\Sdu(x_{1})\cap \Sdu(x_{2})}(y_{1})\ind_{\Sdu(x_{1})\cap \Sdu(x_{2})\cap \Sdu(x_{3})}(y_{2})u_{x_{1}} u_{x_{3}}\\
      &= \max_{\| \b u \| = 1}\sum_{\substack{x_{1}, x_{2}, x_{3} \in\hcV_{\nu}^{(\rm h)}\\x_{2} \prec x_{1}, x_{3}}} \sum_{y} \ind_{\Sdu(x_{1})\cap \Sdu(x_{2})\cap \Sdu(x_{3})}(y)u_{x_{1}} u_{x_{3}},
    \end{split}
  \end{equation*}
  where we could remove one of the sum on $y$ as if $y_{1} \neq y_{2}$
  we would be considering the case of having a cycle $x_1 \simnc y_1 \simnc x_2 \simnc y_2 \simnc x_1$
  in $G^{\nc}$, which contradicts the definition of $G^\nc$. We then notice that
  $\# \left\{ x_{2} \colon y \in \Sdu(x_{2}) \right\} \leq D_{y}^{\nc +}$, so that
  with $\nu$-high probability \cref{lem:bound-D+} gives
  \begin{equation*}
    \begin{split}
      \max_{\| \b u \| = 1}\sum_{\substack{x_{1}, x_{2}, x_{3} \in\hcV_{\nu}^{(\rm h)}\\x_{2} \prec x_{1}, x_{3}}}&\left< \b 1_{\Sdu(x_{1})}, \b 1_{\Sdu(x_{2})} \right>\left< \b 1_{\Sdu(x_{1})\cap\Sdu(x_{2})}, \b 1_{\Sdu(x_{3})} \right>u_{x_{1}} u_{x_{3}}\\
      &\leq \max_{\| \b u \| = 1}\sum_{x_{1}, x_{3} \in\hcV_{\nu}^{(\rm h)}}\sum_{y} D^{\nc +}_{y}\ind_{\Sdu(x_{1})\cap\Sdu(x_{3})}(y)u_{x_{1}} u_{x_{3}}\\
      &\leq \frac{2(\nu + 1)}{1-\delta} \frac{\log N}{\log\log N}\max_{\| \b u \| = 1}\sum_{x_{1}, x_{3} \in\hcV_{\nu}^{(\rm h)}}\# \Bigl(\Sdu(x_{1})\cap\Sdu(x_{3})\Bigr)u_{x_{1}} u_{x_{3}}\,.
    \end{split}
  \end{equation*}
  Note that since the entries of the matrices we consider are
  non-negative, we can assume that $\b u$ has positive coefficients.
  This allows the upper bound in the latter expression. We then
  recognize the norm of the symmetric matrix $B$ whose definition is
  \begin{equation*}
      B = \sum_{\substack{x_1, x_3 \in \hcV_\nu^{(\rm h)}\\x_1 \neq x_3}} \langle \b 1_{\Sdu(x_1)}, \b 1_{\Sdu(x_1)} \rangle \b 1_{x_1} \b 1_{x_3}^* = \sum_{\substack{x_1, x_3 \in \hcV_\nu^{(\rm h)}\\x_1 \neq x_3}} \# \Bigl(\Sdu(x_1)\cap \Sdu(x_3)\Bigr) \b 1_{x_1} \b 1_{x_3}^*\,,
  \end{equation*}
  so that with $\nu$-high probability
  \begin{equation*}
  \begin{split}
      \max_{\| \b u \| = 1}\sum_{\substack{x_{1}, x_{2}, x_{3} \in\hcV_{\nu}^{(\rm h)}\\x_{2} \prec x_{1}, x_{3}}}&\left< \b 1_{\Sdu(x_{1})}, \b 1_{\Sdu(x_{2})} \right>\left< \b 1_{\Sdu(x_{1})\cap\Sdu(x_{2})}, \b 1_{\Sdu(x_{3})} \right>u_{x_{1}} u_{x_{3}}\\
    &\leq \frac{2(\nu + 1)}{1-\delta}\frac{\log N}{\log\log N} \|B\| + \frac{2(\nu + 1)}{1-\delta} \frac{\log N}{\log\log N}\max_{\| \b u\| = 1}\sum_{x \in \hcV^{(\rm h)}}\pB{\# \Sdu(x)} u_x^2\\
    &\leq \frac{2(\nu + 1)}{1-\delta}\frac{\log N}{\log\log N} \|B\| + \frac{4(\nu + 1)^2}{(1-\delta)(1-2\delta)} \pB{\frac{\log N}{\log\log N}}^2\,.
  \end{split}
  \end{equation*}
  where we used a second time \cref{lem:bound-ud-paths}.
  Putting the two bounds together in \eqref{eq:prop-norm-intermediate} we have:
  \begin{equation*}
    \| B \|^{2} \leq 4 \| B_{\succ} \|^{2} \leq \pB{8\nu + \frac{16(\nu + 1)^2}{(1-\delta)(1-2\delta)}} \left( \frac{\log N}{\log\log N} \right)^{2} + \frac{8\nu}{1 - \delta} \frac{\log N}{\log\log N} \|B\|\,.
  \end{equation*}
  After solving a quadratic equation, it implies that for some constant $C_\nu > 0$, we have with $\nu$-high probability
  \begin{equation*}
    \| B\| \leq C_\nu\frac{\log N}{\log\log N}.
  \end{equation*}
  Hence, we have finally that there exists a constant $C_{\nu} > 0$ such that with $\nu$-high probability
  \begin{equation*}
    \| \tilde{A}_{\rm h} \|^{2} \leq C_{\nu} \frac{\log N}{\log\log N}\,.
  \end{equation*}

  Let us now consider $\tilde{A}_{\rm i}$. We have
  \begin{equation*}
    \left\| \tilde{A}_{\rm i} \right\|^{2}
    = \max_{\| \b u \| = 1}\sum_{x_{1}, x_{2} \in \hcV^{(\rm i)}_{\nu}} \left< \b 1_{\Sdu(x_{1})}, \b 1_{\Sdu(x_{2})} \right> u_{x_{1}} u_{x_{2}}\,.
  \end{equation*}
  Young's inequality followed by a use of \cref{lem:A-intermediate} then yields:
  \begin{equation*}
    \left\| \tilde{A}_{\rm i} \right\|^{2}
    \leq \max_{\| \b u \| = 1}\sum_{x_{1}, x_{2} \in \hcV^{(\rm i)}_{\nu}} \left< \b 1_{\Sdu(x_{1})}, \b 1_{\Sdu(x_{2})} \right>  u^{2}_{x_{1}}
    \leq 2\nu \frac{\log N}{\log\log N}
  \end{equation*}
  with $\nu$-high probability.

  It remains to bound the norm of $\tilde{A}_{\rm l}$. First, notice
  that this is the adjacency matrix of a graph $\tilde{G}_{l}$. Let us
  introduce the graph $G_{\rm l}$, the sub-graph of $G$ containing
  only edges between vertices $x$ and $y$ whose weights satisfy
  $w_{x}, w_{y} \leq 4\xi_\nu$. Denote by $A_{\rm l}$ the adjacency matrix of
  $G_{\rm l}$. We also introduce $\hat{A}_{\rm l}$, the adjacency
  matrix of the sub-graph $\hat{G}_{\rm l}$ of $G_{\rm l}$ in which we
  only kept edges between vertices $x, y \in \hcV^{(\rm l)}_{\nu}$. We
  see that we have the inclusion of graphs $\tilde{G}_{\rm l} \subset \hat{G}_{\rm l} \subset G_{\rm l}$. It implies in particular
  \begin{equation*}
    \| \tilde{A}_{\rm l} \| \leq \| \hat{A}_{\rm l} \| \leq \| A_{\rm l}\|.
  \end{equation*}
  We are going to bound $\| \hat{A}_{\rm l} \|$, the bound on
  $\| \tilde{A}_{\rm l} \|$ will immediately follow.

 Note that the entries of the matrix $A_{\rm l}$ are independent. We can thus use the results of \cite{BGBK} to bound the backtracking matrix of $G_{\rm l}$. This will in turn allow us to bound $\| \hat{A}_{\rm l}\|$. First, we observe that
  \begin{multline*}
    \| \E \hat{A}_{\rm l} \|^{2}
    \leq \max_{\| \b u \| = 1} \sum_{x_{1}, x_{2}, y}p_{x_{1}y}p_{x_{2}y}u_{x_{1}}u_{x_{2}}
    \leq \max_{\| \b u \| = 1} \sum_{x_{1}, x_{2}}\frac{w_{x_{1}}w_{x_{2}}}{m_{1}N}\frac{m_{2}}{m_{1}}u_{x_{1}}u_{x_{2}}
    \\
    \leq \max_{\| \b u \| = 1} \sum_{x_{1}, x_{2}}\frac{w_{x_{2}}^{2}}{m_{1}N}\frac{m_{2}}{m_{1}}u_{x_{1}}^{2}
    \leq \left( \frac{m_{2}}{m_{1}} \right)^{2}\,,
  \end{multline*}
  where we used Young's inequality and \eqref{eq:def-empirical-moment}. It remains to bound
  $\| \hat{A}_{\rm l} - \E \hat{A}_{\rm l}\|$. Set
  $\hat{H} = \frac{1}{\sqrt{4\xi}} \p{\hat{A}_{\rm l} - \E \hat{A}_{\rm l}}$
  and denote by $\hat{B}_{\rm l}$ and $B_{\rm l}$ the non-backtracking
  matrices of $\hat{A}_{\rm l}$ and $A_{\rm l}$ respectively. We refer
  to \cite{BGBK} for its definition and only notice that the inclusion of
  graphs implies $\rho(\hat{B}_{\rm l}) \leq \rho(B_{\rm l})$, and since
  \begin{equation*}
    \begin{split}
      \max_{x; w_{x} \leq 4\xi_\nu}\sum_{y}\E |\hat{H}_{xy}|^{2} &\leq \max_{x} \frac{w_{x}}{4\xi_\nu} \leq 1\\
      \max_{x,y; w_{x}, w_{y} \leq 4\xi_\nu} \E |\hat{H}_{xy}|^{2} &\leq  \max_{x,y; w_{x}, w_{y} \leq 4\xi_\nu} \frac{w_{x}w_{y}}{m_{1}N4\xi_\nu} \leq \frac{4\xi_\nu}{m_{1}} \frac{1}{N}\\
      \max_{x,y; w_{x}, w_{y} \leq 4\xi_\nu} |\hat{H}_{xy}| &\leq \frac{1}{\sqrt{4\xi_\nu}}\,,
    \end{split}
  \end{equation*}
  we have that \cite[Assumption 2.4]{BGBK} is satisfied. Hence, \cite[Theorem 2.5]{BGBK}
  implies that with $\nu$-high probability, the spectral radius
  $\rho(B_{\rm l})$ of $B_{\rm l}$ is bounded by $2$. This implies that
  \begin{equation*}
    \rho(\hat{B}_{\rm l}) \leq 2 \text{ with $\nu$-high probability. }
  \end{equation*}

  Then, introducing the norms
  \begin{equation*}
    \begin{split}
      \| \hat{H} \|_{2 \to \infty} &\deq \max_{x}\sqrt{\sum_{y}|\hat{H}_{xy}|^{2}} \leq \max_{x} \sqrt{\frac{D_{x}(1 + \order\p{N^{-\epsilon}})}{4\xi_\nu}}\\
      \| \hat{H} \|_{1 \to \infty} &\deq \max_{x, y}|\hat{H}_{xy}| \leq 1\,,
    \end{split}
  \end{equation*}
  we have by \cite[Theorem 2.2]{BGBK}:
  \begin{equation*}
    \| \hat{H} \| \leq \|\hat{H} \|_{2 \to \infty} f\Biggl(\frac{\rho(\hat{B}_{\rm l})}{\|\hat{H} \|_{2 \to \infty}}\Biggr) + 7\| \hat{H} \|_{1 \to \infty}\,,
  \end{equation*}
  where $f(x) = 2$ if $0 \leq x \leq 1$ and $f(x) = x + 1/x$ if $x > 1$.
  Since the degrees in $\hat{G}_{\rm l}$ are bounded by $\xi_\nu$, we get
  that $\| \hat{H} \|_{2 \to \infty} \leq 1$ so that
  \begin{equation*}
    \| \hat{A}_{\rm l} - \E \hat{A}_{\rm l} \| \leq
    \begin{cases}
      16\sqrt{\xi_\nu} & \text{ if } \rho(\hat{B}_{\rm l}) \leq \| \hat{H} \|_{2 \to \infty}\\
      2\sqrt{\xi_\nu} \left( \rho(\hat{B}_{\rm l}) + \frac{\| \hat{H} \|_{2 \to \infty}^{2}}{\rho(\hat{B}_{\rm l})}+ 7 \right) \leq 2\sqrt{\xi_\nu} \left( 2 + 1 + 7 \right) & \text{ if } \rho(\hat{B}_{\rm l}) > \| \hat{H} \|_{2 \to \infty}.
    \end{cases}
  \end{equation*}
  Hence, we get that
  \begin{equation*}
    \| \hat{A}_{\rm l} - \E \hat{A}_{\rm l} \| \leq 20 \xi\,.
  \end{equation*}
  This concludes the proof that
  \begin{equation*}
    \| \tilde{A}_{\rm l} \| = \order\Biggl(\sqrt{\frac{\log N}{\log\log N}}\Biggr)\,.
  \end{equation*}
  Putting the three bounds for $\circ \in \{ \rm h, \rm i, \rm l\}$ together,
  we get the result.
\end{proof}

\begin{proof}[Proof of \cref{lem:bound-op-high}]
We are going to bound the quantity
  \begin{equation*}
    \mathcal{P}^{(1)}_{x_{1}} \deq \sum_{\substack{x_{2}, x_{3} \in \hcV_{\nu}^{(\rm h)}\\x_{2} \prec x_{1}, x_{3}}} \left< \b 1_{\Sdu(x_{1})}, \b 1_{\Sdu(x_{2})} \right>\left< \b 1_{\Sdu(x_{2})\setminus \Sdu(x_{1}) }, \b 1_{\Sdu(x_{3})} \right>\,.
  \end{equation*}
  with $\nu$-high probability. By using the definition \eqref{eq:def-set-vertices} of
  the sets $\Sdu(x_{i}), i = 1, 2, 3$, $\mathcal{P}_{x_{1}}$ may be
  rewritten as
  \begin{equation*}
    \begin{split}
      \mathcal{P}^{(1)}_{x_{1}}
      = \sum_{\substack{x_{2}, x_{3} \in \hcV_{\nu}^{(\rm h)}\\x_{1}, x_{2}, x_{3}\text{ distinct }}}&\sum_{\substack{y_{1}, y_{2} \in \vertices\\ y_{1} \neq y_{2}}} \ind_{\left\{ x_{1} \simnc y_{1} \simnc x_{2} \simnc y_{2} \simnc x_{3} \right\}}\ind_{\bigl\{D_{y_{1}} \leq D_{x_{2}}\leq D_{x_{1}} , D_{y_{2}} \leq D_{x_{2}} \leq D_{x_{3}}\bigr\}}\\
      &\times\ind_{\left\{ \exists z_{1} \notin \{ x_{1}, x_{2}\}, y_{1} \simnc z_{1}, D_{x_{1}} \leq D_{z_{1}}  \right\}}\ind_{\left\{ \exists z_{2} \notin \{ x_{2}, x_{3}\}, y_{2} \simnc z_{2}, D_{x_{3}} \leq D_{z_{2}}  \right\}}\,.
    \end{split}
  \end{equation*}
  Note that it is so in particular because if
  $y_{1} \in \Sdu(x_{1}) \cap \Sdu(x_{2})$ and
  $y_{2} \in \p{\Sdu(x_{2}) \setminus \Sdu(x_{1})} \cap \Sdu(x_{3})$,
  we have necessarily $y_{1} \neq y_{2}$ and $x_{1}, x_{2}$, and
  $x_{3}$ must be distinct: indeed if we had $x_{1} = x_{3}$ there
  would be a cycle
  $x_{1} \simnc y_{1} \simnc x_{2} \simnc y_{2} \simnc x_{3}=x_{1}$ in
  $G^{\nc}$, in contradiction with the properties of $G^{\nc}$.
  Introduce the notation $\chi = \log\log N$. \cref{rem:ratio-big-dx}
  implies that with $\nu$-high probability, we have for every
  $x \in \vertices$:
  \begin{equation*}
    d_{x}  \leq \chi (D_{x}\vee \xi_{\nu})\,.
  \end{equation*}
  We use this fact and discard some unneeded events. We get that with $\nu$-high probability,
  \begin{equation*}
    \begin{split}
      \mathcal{P}^{(1)}_{x_{1}}
      \leq \sum_{\substack{x_{2}, x_{3} \in \vertices\\x_{1}, x_{2}, x_{3}\text{ distinct }}}&\sum_{\substack{y_{1}, y_{2} \in \vertices\\y_{1}\neq y_{2}}} \ind_{\bigl\{ x_{1} \simnc y_{1} \simnc x_{2} \simnc y_{2} \simnc x_{3}\bigr\}}\ind_{\bigl\{ d_{y_{1}} \vee \xi_{\nu} \leq \chi D_{x_{2}}, d_{y_{1}}\vee d_{y_{2}}\vee d_{x_{2}} \leq \chi D_{x_{3}} \bigr\}}\\
                                                                                 &\times\ind_{\left\{ \exists z \neq x_{1}, x_{2}, y_{1} \simnc z, d_{y_{1}} \vee d_{x_{1}} \leq \chi D_{z} \right\}}\,.
    \end{split}
  \end{equation*}
  The indicator function on the last line can be further simplified
  using \cref{rem:upper-bd-degree}. Indeed, if we assume that
  $d_{1} > 8\nu\log N$ and $D_{x_{1}} \leq D_{z}$,
  \cref{lem:estimate-degree} implies that with $\nu$-high probability
  \begin{equation*}
    4\nu\log N < \frac{d_{x_{1}}}{2} < d_{x_{1}} - \sqrt{2\nu\log N d_{1}} \leq D_{x_{1}} \leq D_{z} \leq d_{z} + 2\sqrt{\nu\log N \left( d_{z} \vee \frac{4\nu}{9}\log N \right)} \leq 4d_{z}\vee \frac{16\nu}{9}\log N\,.
  \end{equation*}
  The last inequality has been discussed in
  \cref{rem:upper-bd-degree}. This implies that with $\nu$-high
  probability, we have $D_{z} \leq 4d_{z}$. We thus have with $\nu$-high probability
  \begin{equation*}
    \begin{split}
      \mathcal{P}^{(1)}_{x_{1}}
      \leq \sum_{\substack{x_{2}, x_{3} \in \vertices\\x_{1}, x_{2}, x_{3}\text{ distinct }}}&\sum_{\substack{y_{1}, y_{2} \in \vertices\\y_{1}\neq y_{2}}} \ind_{\bigl\{ x_{1} \simnc y_{1} \simnc x_{2} \simnc y_{2} \simnc x_{3}\bigr\}}\ind_{\bigl\{ d_{y_{1}} \vee \xi_{\nu} \leq \chi D_{x_{2}}, d_{y_{1}}\vee d_{y_{2}}\vee d_{x_{2}} \leq \chi D_{x_{3}} \bigr\}}\\
                                                                                             &\times \left( \ind_{\bigl\{d_{1} \leq 8\nu \log N\bigr\}} + \ind_{\bigl\{d_{1} > 8\nu \log N\bigr\}}\ind_{\left\{ \exists z \neq x_{1}, y_{1} \simnc z, d_{y_{1}} \vee d_{x_{1}} \leq 4\chi d_{z} \right\}} \right)\,.
    \end{split}
  \end{equation*}

  Recall that by \cref{prop:edge-cyc}, there exists
  $C_{\nu} > 0$ such that with $\nu$-high probability, we have for all
  $x \in \vertices$, $D_{x} \leq D_{x}^{\nc} + C_{\nu}$. We thus get
  that with $\nu$-high probability
  \begin{equation*}
    \begin{split}
      \mathcal{P}^{(1)}_{x_{1}}
      \leq \sum_{\substack{x_{2}, x_{3} \in \vertices\\x_{1}, x_{2}, x_{3}\text{ distinct }}}&\sum_{\substack{y_{1}, y_{2} \in \vertices\\y_{1} \neq y_{2}}} \ind_{\bigl\{ x_{1} \simnc y_{1} \simnc x_{2} \simnc y_{2} \simnc x_{3}\bigr\}}\ind_{\bigl\{ d_{y_{1}} \vee \xi_{\nu} \leq \chi (D_{x_{2}}^{\nc} + C_{\nu}), d_{y_{1}}\vee d_{y_{2}}\vee d_{x_{2}} \leq \chi (D_{x_{3}}^{\nc} + C_{\nu}) \bigr\}}\\
                                                                                 &\times\Bigl( \ind_{\bigl\{d_{1} \leq 8\nu \log N\bigr\}} + \ind_{\bigl\{d_{1} > 8\nu \log N\bigr\}}\ind_{\bigl\{ \exists z \neq x_{1}, y_{1} \simnc z, d_{y_{1}} \vee d_{x_{1}} \leq 4\chi d_{z} \bigr\}} \Bigr)\,.
    \end{split}
  \end{equation*}
  We can now use the coupling with a rooted tree introduced in
  \cref{sec:coupling-with-tree}. To make notation lighter we write $\Tree \deq \Tree_{x_1}$. \cref{prop:embedding-Tx} implies that
  for all $\tilde{x} \in B^{\nc}_{r}(x)$,
  \begin{equation*}
    D_{\tilde{x}}^{\nc} \leq D_{\tilde{x}}^{\Tree} \leq D^{\Tree\uparrow}_{\tilde{x}} + 1\,,
  \end{equation*}
  so up to replacing $C_{\nu}$ by $C_{\nu} +1$ we have with $\nu$-high probability:
  \begin{equation*}
    \begin{split}
      \mathcal{P}^{(1)}_{x_{1}}
      \leq \sum_{\substack{x_{2}, x_{3} \in V_{x_1}\\x_{1}, x_{2}, x_{3}\text{ distinct }}}&\sum_{\substack{y_{1}, y_{2} \in V_{x_1}\\y_{1}\neq y_{2}}} \ind_{\bigl\{ x_{1} \simT{} y_{1} \simT{} x_{2} \simT{} y_{2} \simT{} x_{3}\bigr\}}\ind_{\bigl\{ d_{y_{1}} \vee \xi_{\nu} \leq \chi (D_{x_{2}}^{\Tree\uparrow} + C_{\nu}), d_{y_{1}} \vee d_{y_{2}}\vee d_{x_{2}} \leq \chi (D_{x_{3}}^{\Tree\uparrow} + C_{\nu}) \bigr\}}\\
                                                                                 &\times\Bigl( \ind_{\bigl\{d_{1} \leq 8\nu \log N\bigr\}} + \ind_{\bigl\{d_{1} > 8\nu \log N\bigr\}}\ind_{\bigl\{ \exists z \neq x_{1}, y_{1} \simT{} z, d_{y_{1}} \vee d_{x_{1}} \leq 4\chi d_{z} \bigr\}} \Bigr)\,.
    \end{split}
  \end{equation*}
  In the sequel, we are going to use the filtration
  $(\mathcal{F}_{k}(x_{1}))_{k \geq 1}$ defined in
  \eqref{eq:def-filtration-F}.

  Bounding $\mathcal{P}^{(1)}_{x_{1}}$ is based on a Chernoff-type
  bound. We have for all $\lambda > 0$ and $k \geq 0$:
  \begin{equation*}
    \P \left( \mathcal{P}^{(1)}_{x_{1}} \geq k \right)
    \leq \ee^{-\lambda k}\E \left[ \exp(\lambda \mathcal{P}^{(1)}_{x_{1}}) \right]\,.
  \end{equation*}
  We are going to use this bound with the particular choice $\lambda = \frac{\chi^{2}}{\log N}$ and $k = 2\nu \bigl(\frac{\log N}{\chi}\bigr)^2$. Let us bound the Laplace transform in the right-hand term with this choice $\lambda = \frac{\chi^{2}}{\log N}$.
  We write
  \begin{equation*}
    \mathcal{P}^{(1)}_{x_{1}} = \sum_{y_{1}\in V_{x}}X_{x_{1}y_{1}}^{(1)}\,,
  \end{equation*}
  with
  \begin{equation*}
    \begin{split}
      X_{x_{1}y_{1}x_{2}y_{2}x_{3}}^{(4)} &\deq \ind_{\bigl\{ y_{2} \simT{} x_{3}, d_{x_{2}}\vee d_{y_{1}}\vee d_{y_{2}} \leq \chi \left( D^{\Tree\uparrow}_{x_{3}} + C_{\nu} \right) \bigr\}}\\
      X_{x_{1}y_{1}x_{2}y_{2}}^{(3)} &\deq \ind_{\bigl\{ x_{2} \simT{} y_{2} \bigr\}}\sum_{x_{3} \neq x_{1}, x_{2}}X_{x_{1}y_{1}x_{2}y_{2}x_{3}}^{(4)}\\
      X_{x_{1}y_{1}x_{2}}^{(2)} &\deq \ind_{\bigl\{ y_{1} \simT{} x_{2} \bigr\}}\sum_{y_{2} \neq y_{1}}X^{(3)}_{x_{1}y_{1}x_{2}y_{2}}\\
      X_{x_{1}y_{1}}^{(1)} &\deq \ind_{\bigl\{ x_{1} \simT{} y_{1}\bigr\}}\Bigl( \ind_{\bigl\{d_{1} \leq 8\nu \log N\bigr\}} + \ind_{\bigl\{d_{1} > 8\nu \log N\bigr\}}\ind_{\bigl\{ \exists z \neq x_{1}, y_{1} \simT{} z, d_{y_{1}} \vee d_{x_{1}} \leq 4\chi d_{z} \bigr\}} \Bigr)\sum_{x_{2} \neq x_{1}}X^{(2)}_{x_{1}y_{1}x_{2}}\,.
    \end{split}
  \end{equation*}
  We will consider these different random variables starting from
  $X^{(4)}_{x_{1}y_{1}x_{2}y_{2}x_{3}}$ and use the independence
  properties of $\Tree$ to estimate the Laplace transform of each of
  those. In this computation, the tree structure of $\Tree$ is
  critical.

  Let us fix $y_{1}, x_{2}, y_{2}, x_{3}$ for now, and consider
  $X_{x_{1}y_{1}x_{2}y_{2}x_{3}}^{(4)}$. Conditionally to
  $\mathcal{F}_{4}(x_{1})$, we have
  \begin{equation*}
    \E \Bigl[ \ee^{\lambda X_{x_{1}y_{1}x_{2}y_{2}x_{3}}^{(4)}} \mid \mathcal{F}_{4}(x_{1})\Bigr] = 1 + \P \Bigl(y_{2} \simT{} x_{3}, d_{x_{2}}\vee d_{y_{1}}\vee d_{y_{2}} \leq \chi \bigl( D^{\Tree\uparrow}_{x_{3}} + C_{\nu} \bigr) \mid \mathcal{F}_{4}(x_{1})\Bigr) \left( \ee^{\lambda} - 1 \right).
  \end{equation*}
  Markov's inequality and the fact that $\E \Bigl[D^{\Tree\uparrow}_{x_{3}}\Bigr] \leq \sum_z p_{x_3 z}$ give us
  \begin{equation*}
    \E \Bigl[ \ee^{\lambda X_{x_{1}y_{1}x_{2}y_{2}x_{3}}^{(4)}} \mid \mathcal{F}_{4}(x_{1})\Bigr] = 1 + \ind_{\bigl\{y_{2} \simT{} x_{3}\bigr\}}\chi\frac{d_{x_{3}} + C_{\nu}}{d_{x_{2}}\vee d_{y_{1}}\vee d_{y_{2}}} \left( \ee^{\lambda} - 1 \right).
  \end{equation*}
  We choose for $\lambda$ the small value $\lambda = \frac{\chi^{2}}{\log N}$. Hence, we
  have for $1 < c < 2$ that for $N$ big enough:
  \begin{equation*}
    \ee^{\lambda} - 1 \leq c \frac{\chi^{2}}{\log N}.
  \end{equation*}
  This means that
  \begin{equation*}
    \E \Bigl[ \ee^{\lambda X_{x_{1}y_{1}x_{2}y_{2}x_{3}}^{(4)}} \mid \mathcal{F}_{4}(x_{1})\Bigr] \leq 1 + \ind_{\bigl\{y_{2} \simT{} x_{3}\bigr\}}c \frac{\chi^{3}}{\log N}\frac{d_{x_{3}} + C_{\nu}}{d_{x_{2}}\vee d_{y_{1}}\vee d_{y_{2}}}.
  \end{equation*}

  Still keeping $x_{2}, y_{1}, y_{2}$ fixed, we take the product on
  $x_{3} \neq x_{1}, x_{2}$ and condition on $\mathcal{F}_{3}(x_{1})$.
  The conditional independence property of $\Tree$ gives us
  \begin{equation*}
    \E \Bigl[ \ee^{\lambda \sum_{x_{3}\neq x_{1}, x_{2}} X_{x_{1}y_{1}x_{2}y_{2}x_{3}}^{(4)}} \mid \mathcal{F}_{3}(x_{1})\Bigr]
    \leq \prod_{x_{3}\neq x_{1}, x_{2}} \Biggl(1 + c p_{y_{2}x_{3}} \frac{\chi^{3}}{\log N}\frac{d_{x_{3}} + C_{\nu}}{d_{x_{2}}\vee d_{y_{1}}\vee d_{y_{2}}}\Biggr)\,.
  \end{equation*}
  Using \eqref{eq:proba-edge}, and $\ln (1 + u) \leq u$, we get
  \begin{equation*}
    \begin{split}
      \E \Bigl[ \ee^{\lambda \sum_{x_{3}\neq x_{1}, x_{2}} X_{x_{1}y_{1}x_{2}y_{2}x_{3}}^{(4)}} \mid \mathcal{F}_{3}(x_{1})\Bigr]
      &\leq \prod_{x_{3}\neq x_{1}, x_{2}} \Biggl(1 + c \frac{\chi^{3}}{m_{1}N\log N}\frac{w_{y_{2}}w_{x_{3}}\bigl(w_{x_{3}} + C_{\nu}\bigr)}{d_{x_{2}}\vee d_{y_{1}}\vee d_{y_{2}}}\Biggr)\\
      &\leq \exp\Biggl(\sum_{x_{3}\neq x_{1}, x_{2}} c \frac{\chi^{3}}{m_{1}N\log N}\frac{w_{y_{2}}w_{x_{3}}\bigl(w_{x_{3}} + C_{\nu}\bigr)}{d_{x_{2}}\vee d_{y_{1}}\vee d_{y_{2}}}\Biggr)\,.
    \end{split}
  \end{equation*}
  We can make the ratio of empirical moments $m_{2}/m_{1}$ appear
  using \eqref{eq:def-empirical-moment} and get with \cref{hyp:behavior-m21}
  \begin{equation*}
    \begin{split}
      \E \Bigl[ \ee^{\lambda \sum_{x_{3}\neq x_{1}, x_{2}} X_{x_{1}y_{1}x_{2}y_{2}x_{3}}^{(4)}} \mid \mathcal{F}_{3}(x_{1})\Bigr]
      &\leq \exp\Biggl(\frac{c \chi^{3}}{\log N} \frac{m_{2} + C_{\nu}m_{1}}{m_{1}}\frac{w_{y_{2}}}{d_{x_{2}}\vee d_{y_{1}}\vee d_{y_{2}}}\Biggr)\\
      &\leq \exp\Biggl(\frac{c^{2} \chi^{3}}{(\log N)^{1-\delta}} \frac{w_{y_{2}}}{w_{x_{2}}\vee w_{y_{1}}\vee w_{y_{2}}}\Biggr)\,.
    \end{split}
  \end{equation*}
  We could replace the average degrees
  $d_{x_{2}}\vee d_{y_{1}}\vee d_{y_{2}}$ by the weights at the cost
  of adding a small constant $c$ in the product. Finally, we can
  linearize the exponential again and get
  \begin{equation}\label{eq:P1-X4}
    \E \Bigl[ \ee^{\lambda \sum_{x_{3}\neq x_{1}, x_{2}} X_{x_{1}y_{1}x_{2}y_{2}x_{3}}^{(4)}} \mid \mathcal{F}_{3}(x_{1})\Bigr]
    \leq 1 + \frac{c^{3} \chi^{3}}{(\log N)^{1-\delta}} \frac{w_{y_{2}}}{w_{x_{2}}\vee w_{y_{1}}\vee w_{y_{2}}}.
  \end{equation}

  We are going to apply a similar argument to the one we used on
  $X^{(4)}_{x_{1}y_{1}x_{2}y_{2}x_{3}}$ on
  $X^{(3)}_{x_{1}y_{1}x_{2}y_{2}}$. We fix $y_{1}, x_{2}, y_{2}$. We have
  \begin{equation*}
    \ee^{\lambda X^{(3)}_{x_{1}y_{1}x_{2}y_{2}}} = 1 + \ind_{\bigl\{x_{2} \simT{} y_{2}\bigr\}} \left( \ee^{\lambda \sum_{x_{3}\neq x_{1}, x_{2}} X_{x_{1}y_{1}x_{2}y_{2}x_{3}}^{(4)}} - 1 \right).
  \end{equation*}
  Taking the expectation conditional to $\mathcal{F}_{3}(x_{1})$ and using \eqref{eq:P1-X4}, we get
  \begin{equation*}
    \E \Bigl[ \ee^{\lambda X_{x_{1}y_{1}x_{2}y_{2}}^{(3)}} \mid \mathcal{F}_{3}(x_{1})\Bigr] \leq 1 + \ind_{\bigl\{x_{2} \simT{} y_{2}\bigr\}}\frac{c^{3} \chi^{3}}{(\log N)^{1-\delta}} \frac{w_{y_{2}}}{w_{x_{2}}\vee w_{y_{1}}\vee w_{y_{2}}}.
  \end{equation*}
  Again, we take the product on $y_{2} \neq y_{1}$ and the expectation
  conditionally on $\mathcal{F}_{2}(x_{1})$. We get
  \begin{equation*}
    \E \Bigl[ \ee^{\lambda \sum_{y_{2} \neq y_{1}}X_{x_{1}y_{1}x_{2}y_{2}}^{(3)}} \mid \mathcal{F}_{2}(x_{1})\Bigr]
    \leq \prod_{y_{2} \neq y_{1}}\Biggl(1+ \frac{c^{3} \chi^{3}}{(\log N)^{1-\delta}}p_{x_{2}y_{2}}\frac{w_{y_{2}}}{w_{x_{2}}\vee w_{y_{1}}\vee w_{y_{2}}}\Biggr)\,.
  \end{equation*}
  Using $\log(1+u) \leq u$, \eqref{eq:proba-edge}, and \eqref{eq:def-empirical-moment}, we get
  \begin{equation*}
    \begin{split}
      \E \Bigl[ \ee^{\lambda \sum_{y_{2} \neq y_{1}}X_{x_{1}y_{1}x_{2}y_{2}}^{(3)}} \mid \mathcal{F}_{2}(x_{1})\Bigr]
      &\leq \exp\Bigl(\sum_{y_{2} \neq y_{1}}\frac{c^{3} \chi^{3}}{(\log N)^{1-\delta}}p_{x_{2}y_{2}}\frac{w_{y_{2}}}{w_{x_{2}}\vee w_{y_{1}}\vee w_{y_{2}}}\Bigr)\\
      &\leq \exp\Bigl(\sum_{y_{2} \neq y_{1}}\frac{c^{3} \chi^{3}}{m_{1}N(\log N)^{1-\delta}}\frac{w_{x_{2}}w_{y_{2}}^{2}}{w_{x_{2}}\vee w_{y_{1}}\vee w_{y_{2}}}\Bigr)\\
      &\leq \exp\Bigl(\frac{c^{3} \chi^{3}}{(\log N)^{1-2\delta}}\frac{w_{x_{2}}}{w_{x_{2}}\vee w_{y_{1}}}\Bigr)\,.
    \end{split}
  \end{equation*}
  By linearizing the exponential, we get
  \begin{equation*}
    \E \Bigl[ \ee^{\lambda \sum_{y_{2} \neq y_{1}}X_{x_{1}y_{1}x_{2}y_{2}}^{(3)}} \mid \mathcal{F}_{2}(x_{1})\Bigr]
    \leq 1 + \frac{c^{3} \chi^{3}}{(\log N)^{1-2\delta}}\frac{w_{x_{2}}}{w_{x_{2}}\vee w_{y_{1}}}\,.
  \end{equation*}
  Since we have
  \begin{equation*}
    \ee^{\lambda X^{(2)}_{x_{1}y_{1}x_{2}}} = 1 + \ind_{\bigl\{y_{1}\simT{}x_{2}\bigr\}} \left( \ee^{\lambda \sum_{y_{2}\neq y_{1}}X^{(3)}_{x_{1}y_{1}x_{2}y_{2}}} - 1\right)\,,
  \end{equation*}
  We finally get
  \begin{equation}\label{eq:P1-X2}
    \E\Bigl[\ee^{\lambda X^{(2)}_{x_{1}y_{1}x_{2}}}\mid \mathcal{F}_{2}(x_{1})\Bigr] \leq 1 + \ind_{\bigl\{y_{1}\simT{}x_{2}\bigr\}}\frac{c^{3} \chi^{3}}{(\log N)^{1-2\delta}}\frac{w_{x_{2}}}{w_{x_{2}}\vee w_{y_{1}}}\,.
  \end{equation}

  We may take the product on $x_{2} \neq x_{1}$ and the expectation
  conditionally on $\mathcal{F}_{1}(x_{1})$. We get using the conditional independence of the $\ind_{\bigl\{ y_{1} \simT{} x_{2}\bigr\}}$:
  \begin{equation*}
    \E\Bigl[\ee^{\sum_{x_{2}\neq x_{1}}\lambda X^{(2)}_{x_{1}y_{1}x_{2}}}\mid \mathcal{F}_{2}(x_{1})\Bigr] \leq \prod_{x_{2}\neq x_{1}} \Bigl(1 + p_{y_{1}x_{2}}\frac{c^{3} \chi^{3}}{(\log N)^{1-2\delta}}\frac{w_{x_{2}}}{w_{x_{2}}\vee w_{y_{1}}}\Bigr)\,.
  \end{equation*}
  Using again $\log(1 + u) \leq u$, \eqref{eq:proba-edge}, and \eqref{eq:def-empirical-moment}, we get
  \begin{equation*}
    \E\Bigl[\ee^{\sum_{x_{2}\neq x_{1}}\lambda X^{(2)}_{x_{1}y_{1}x_{2}}}\mid \mathcal{F}_{2}(x_{1})\Bigr] \leq \exp\Bigl( \frac{c^{3} \chi^{3}}{(\log N)^{1-3\delta}}\Bigr)\,.
  \end{equation*}
  Linearizing the exponential, we get
  \begin{equation}\label{eq:P1-prod-X2}
    \E\Bigl[\ee^{\sum_{x_{2}\neq x_{1}}\lambda X^{(2)}_{x_{1}y_{1}x_{2}}}\mid \mathcal{F}_{2}(x_{1})\Bigr] \leq 1 + \frac{c^{3} \chi^{3}}{(\log N)^{1-3\delta}}\,.
  \end{equation}

  The computation of the Laplace transform of $X^{(1)}_{x_{1}y_{1}}$
  is more involved than the preceding ones. By definition, we have for $x_{1}$ and $y_{1}$ fixed that if $d_{1} \leq 8\nu\log N$
  \begin{equation*}
    \ee^{\lambda X^{(1)}_{x_{1}y_{1}}} = 1 + \ind_{\bigl\{ x_{1} \simT{} y_{1}\bigr\}}\Bigl( \ee^{\lambda\sum_{x_{2}\neq x_{1}}X^{(2)}_{x_{1}y_{1}x_{2}}} - 1 \Bigr)\,.
  \end{equation*}
  In this case, it is straightforward to compute the expectation
  conditionally on $\mathcal{F}_{1}(x_{1})$ using \eqref{eq:P1-prod-X2}:
  \begin{equation*}
    \E\Bigl[\ee^{\lambda X^{(1)}_{x_{1}y_{1}}}\mid \mathcal{F}_{1}(x_{1})\Bigr] \leq 1 + \ind_{\bigl\{ x_{1} \simT{} y_{1}\bigr\}}\frac{c^{3} \chi^{3}}{(\log N)^{1-3\delta}}\,.
  \end{equation*}
  The more involved case is when $d_{x_{1}} > 8\nu\log N$. In this case:
  \begin{equation*}
    \ee^{\lambda X^{(1)}_{x_{1}y_{1}}} = 1 + \ind_{\bigl\{ x_{1} \simT{} y_{1}\bigr\}}\ind_{\bigl\{ \exists z \neq x_{1}, y_{1} \simT{} z, d_{y_{1}} \vee d_{x_{1}} \leq 4\chi d_{z} \bigr\}}\Bigl( \ee^{\lambda\sum_{x_{2}\neq x_{1}}X^{(2)}_{x_{1}y_{1}x_{2}}} - 1\Bigr)\,.
  \end{equation*}
  At this point, we notice the union bound
  \begin{equation*}
    \ind_{\bigl\{ \exists z \neq x_{1}, y_{1} \simT{} z, d_{y_{1}} \vee d_{x_{1}} \leq 4\chi d_{z}\bigr\}}
    \leq \sum_{z \neq x_{1}}\ind_{\bigl\{ y_{1} \simT{} z, d_{y_{1}} \vee d_{x_{1}} \leq 4\chi d_{z}\bigr\}}\,.
  \end{equation*}
  Together with
  $\ind_{\bigl\{ d_{y_{1}} \vee d_{x_{1}} \leq 4\chi d_{z}\bigr\}} \leq \frac{4\chi d_{z}}{d_{y_{1}} \vee d_{x_{1}}}$,
  this allows the bound
  \begin{equation*}
    \ee^{\lambda X^{(1)}_{x_{1}y_{1}}} \leq 1 + \sum_{z\neq x_{1}}\ind_{\bigl\{ x_{1} \simT{} y_{1}\simT{}z\bigr\}} \frac{4\chi d_{z}}{d_{y_{1}} \vee d_{x_{1}}}\Bigl( \ee^{\lambda\sum_{x_{2}\neq x_{1}}X^{(2)}_{x_{1}y_{1}x_{2}}} - 1\Bigr)\,.
  \end{equation*}
  We first take the expectation conditionally on
  $\mathcal{F}_{2}(x_{1})$ and use \eqref{eq:P1-X2}
  \begin{equation*}
    \begin{split}
      \E\Bigl[\ee^{\lambda X^{(1)}_{x_{1}y_{1}}}&\mid \mathcal{F}_{2}(x_{1})\Bigr]
      \leq 1 + \sum_{z\neq x_{1}}\ind_{\bigl\{ x_{1} \simT{} y_{1}\simT{}z\bigr\}} \frac{4\chi d_{z}}{d_{y_{1}} \vee d_{x_{1}}}\Biggl( \E\Bigl[\ee^{\lambda X^{(2)}_{x_{1}y_{1}z} + \lambda\sum_{x_{2}\neq x_{1}, z}X^{(2)}_{x_{1}y_{1}x_{2}}}\mid \mathcal{F}_{2}(x_{1})\Bigr] - 1\Biggr)\\
      &\leq 1 + \sum_{z\neq x_{1}}\ind_{\bigl\{ x_{1} \simT{} y_{1}\simT{}z\bigr\}} \frac{4\chi d_{z}}{d_{y_{1}} \vee d_{x_{1}}}\Biggl( \E\Bigl[\ee^{\lambda\sum_{x_{2}\neq x_{1}, z}X^{(2)}_{x_{1}y_{1}x_{2}}}(1 + \frac{c^{3}\chi^{3}}{(\log N)^{1 - 2\delta}})\mid \mathcal{F}_{2}(x_{1})\Bigr] - 1\Biggr)\,.
    \end{split}
  \end{equation*}
  We can now take the expectation conditionally on
  $\mathcal{F}_{1}(x_{1})$ and use the conditional independence and
  \eqref{eq:P1-prod-X2} to get
  \begin{equation*}
    \begin{split}
      \E\Bigl[\ee^{\lambda X^{(1)}_{x_{1}y_{1}}}&\mid \mathcal{F}_{1}(x_{2})\Bigr]\\
      &\leq 1 + \sum_{z\neq x_{1}}\ind_{\bigl\{ x_{1} \simT{} y_{1}\bigr\}}p_{y_{1}z} \frac{4\chi d_{z}}{d_{y_{1}} \vee d_{x_{1}}}\Bigl( (1 + \frac{c^{3}\chi^{3}}{(\log N)^{1-3\delta}})(1 + \frac{c^{3}\chi^{3}}{(\log N)^{1 - 2\delta}}) - 1\Bigr)\\
      &\leq 1 + \sum_{z\neq x_{1}}\ind_{\bigl\{ x_{1} \simT{} y_{1}\bigr\}}p_{y_{1}z} \frac{d_{z}}{d_{y_{1}} \vee d_{x_{1}}}\frac{4c^{4}\chi^{4}}{(\log N)^{1-3\delta}}\,.
    \end{split}
  \end{equation*}
  We use \eqref{eq:proba-edge} and \eqref{eq:def-empirical-moment} to get
  \begin{equation*}
    \E \Bigl[\ee^{\lambda X^{(1)}_{x_{1}y_{1}}}\mid \mathcal{F}_{1}(x_{1})\Bigr]
    \leq 1 + \ind_{\bigl\{ x_{1} \simT{} y_{1}\bigr\}}\frac{4c^{4}\chi^{4}}{(\log N)^{1-4\delta}}\frac{w_{y_{1}}}{d_{y_{1}} \vee d_{x_{1}}}\,.
  \end{equation*}
  Up to adding a factor $c$, we can replace $d_{y_{1}} \vee d_{x_{1}}$
  by $w_{y_{1}} \vee w_{x_{1}}$. We finally have
  \begin{equation}\label{eq:P1-X1}
    \E\Bigl[\ee^{\lambda X^{(1)}_{x_{1}y_{1}}}\mid \mathcal{F}_{1}(x_{1})\Bigr]
    \leq 1 + \ind_{\bigl\{ x_{1} \simT{} y_{1}\bigr\}}\frac{4c^{5}\chi^{4}}{(\log N)^{1-4\delta}}\frac{w_{y_{1}}}{w_{y_{1}} \vee w_{x_{1}}}\,.
  \end{equation}

  It remains to take the product on $y_{1}$ and the expectation conditional to $\mathcal{F}_{1}(x_{1})$. We consider first the case $d_{x_{1}} \leq 8\nu\log N$. We have by independence
  \begin{equation*}
    \E \Bigl[\ee^{\lambda \mathcal{P}^{(1)}_{x_{1}}}\Bigr] \leq \prod_{y_{1}} \Bigl(1 + p_{x_{1}y_{1}}\frac{c^{3} \chi^{3}}{(\log N)^{1-3\delta}}\Bigr)\,.
  \end{equation*}
  Using as before \eqref{eq:proba-edge} and \eqref{eq:def-empirical-moment} we get
  \begin{equation*}
    \E \Bigl[\ee^{\lambda \mathcal{P}^{(1)}_{x_{1}}}\Bigr] \leq \exp \Bigl(w_{x_{1}}\frac{c^{3} \chi^{3}}{(\log N)^{1-3\delta}}\Bigr) \leq \exp\Bigl(c^{4}\nu\chi^{3}(\log N)^{3\delta}\Bigr)\,.
  \end{equation*}
  In the case $d_{x_{1}} > 8\nu \log N$, we have by independence,
  \begin{equation*}
    \E \Bigl[\ee^{\lambda \mathcal{P}^{(1)}_{x_{1}}}\Bigr] \leq \prod_{y_{1}} \Bigl(1 + p_{x_{1}y_{1}}\frac{4c^{5}\chi^{4}}{(\log N)^{1-4\delta}}\frac{w_{y_{1}}}{w_{y_{1}} \vee w_{x_{1}}}\Bigr)\,.
  \end{equation*}
  We proceed as in the previous case and get
  \begin{equation*}
    \E \Bigl[\ee^{\lambda \mathcal{P}^{(1)}_{x_{1}}}\Bigr] \leq \exp\Biggl( \frac{4c^{5}\chi^{4}}{(\log N)^{1-5\delta}}\Biggr)\,.
  \end{equation*}
  We have shown in both cases that
  \begin{equation*}
    \E \Bigl[\ee^{\lambda \mathcal{P}^{(1)}_{x_{1}}}\Bigr] \leq N\,.
  \end{equation*}
  This implies
  \begin{equation*}
    \P \left( \mathcal{P}^{(1)}_{x_{1}} \geq 2\nu \left( \frac{\log N}{\log\log N} \right)^{2} \right) = \order\p{N^{-\nu}}\,. \qedhere
  \end{equation*}
\end{proof}
\begin{proof}[Proof of \cref{lem:A-intermediate}]
  Let $x_{1} \in \hcV^{(\rm i)}_{\nu}$. We are going to bound the quantity
  \begin{equation*}
    \mathcal{P}^{(2)}_{x_{1}} \deq \sum_{\substack{x_{2} \in \hcV^{(\rm i)}_{\nu}\\x_{2} \neq x_{1}}} \left< \b 1_{\Sdu(x_{1})}, \b 1_{\Sdu(x_{2})} \right>\,.
  \end{equation*}

  Let us start by noticing that by \cref{lem:estimate-degree}, we have
  the following property with $\nu$-high probability: if
  $z \in \hcV^{(\rm l)}_{\nu} \cup \hcV^{(\rm i)}_{\nu}$ then
  $w_{z} < 4\nu\log N$.

  We thus observe that with $\nu$-high probability, we have $w_{x_{1}} <  4\nu\log N$ and
  \begin{equation*}
    \mathcal{P}^{(2)}_{x_{1}} \leq \sum_{\substack{x_{2} \in \vertices\\x_{2}\neq x_{1}\\4\xi \leq w_{x_{2}} \leq 4\nu\log N}}\sum_{\substack{y \in \vertices\\w_{y} \leq 4\nu\log N}} \ind_{\left\{ x\simnc y \simnc z \right\}}\ind_{\bigl\{ D_{z} < \xi \bigr\}}\,.
  \end{equation*}
  Let us relate the right-hand side to the rooted tree
  $\check{\Tree} \deq \check{\Tree}_{x_{1}}$ of
  \cref{sec:coupling-with-tree}. We have by \cref{prop:embedding-Tx}
  and \cref{corol:bound-Dnc} that there exists a constant $C_{\nu} > 0$
  such that
  \begin{equation*}
    \begin{split}
      \mathcal{P}^{(2)}_{x_{1}}
      &\leq \sum_{\substack{x_{2} \in \vertices\\x_{2}\neq x_{1}\\4\xi \leq w_{x_{2}} \leq 4\nu\log N}}\sum_{\substack{y \in \vertices\\w_{y} \leq 4\nu\log N}} \ind_{\left\{ x\simnc y \simnc z \right\}}\ind_{\bigl\{ D^{\nc}_{z} < \xi \bigr\}}\\
      &\leq \sum_{\substack{x_{2} \in V_{x_{1}}\\x_{2}\neq x_{1}\\4\xi \leq w_{x_{2}} \leq 4\nu\log N}}\sum_{\substack{y \in V_{x_{1}}\\w_{y} \leq 4\nu\log N}} \ind_{\bigl\{ x\simTc{} y \simTc{} z \bigr\}}\ind_{\bigl\{ D^{\check{\Tree}\uparrow}_{z} - C_{\nu} < \xi \bigr\}}\,,
    \end{split}
  \end{equation*}
  with $\nu$-high probability.

  Before computing $\E[\exp(\lambda \mathcal{P}^{(2)}_{x_{1}})]$, let us now
  compute
  $\P \left( D^{\check{\Tree}\uparrow}_{x_{2}} - C_{\nu} \leq \xi \right)$
  using the version of Bennett's inequality given in \cite[Theorem
  2.21]{hofstad_random_2016}. Since $\sum_{z \in V_{x_{1}}} \E \Bigl[\ind_{\bigl\{ x \simTc{} z\bigr\}}\Bigr]
    \leq w_{x}$, we have
  \begin{equation*}
    \P \left( D^{\check{\Tree}\uparrow}_{x_{2}} - C_{\nu} \leq \xi\right)
    = \P \left( D^{\check{\Tree}\uparrow}_{x_{2}} - w_{x_{2}} \leq -(w_{x_{2}} - \xi - C_{\nu})\right)
    \leq \exp\biggl(-\frac{(w_{x_{2}} - \xi - C_{\nu})^{2}}{2w_{x_{2}}}\biggr)\,.
  \end{equation*}
  Hence, we get that
  \begin{equation*}
    \P \left( D^{\check{\Tree}\uparrow}_{x_{2}} + C_{\nu} \leq \xi\right) \leq \exp\bigl(-\xi(1+o(1))\bigr)\,.
  \end{equation*}
  For us it suffices that for all $\beta > 0$ we have for $N$ big enough
  \begin{equation}\label{eq:bound-DT-inter}
    \P \left( D^{\check{\Tree}\uparrow}_{x_{2}} + C_{\nu} \leq \xi\right) \leq \frac{1}{(\log N)^{\beta}}\,.
  \end{equation}

  We have that for $x_{1}$ and $y$ fixed
  \begin{equation*}
    \exp\biggl(\lambda \sum_{\substack{x_{2} \neq x_{1}\\4\xi_{\nu} \leq w_{x_{2}} \leq 4\nu \log N}} \ind_{\bigl\{y \simTc{} x_{2} \bigr\}}\ind_{\bigl\{D^{\check{\Tree}\uparrow}_{x_{2}} - C_{\nu} \leq \xi\bigr\}}\biggr)
    = \prod_{\substack{x_{2} \neq x_{1}\\4\xi_{\nu} \leq w_{x_{2}} \leq 4\nu \log N}}\Biggl( 1 + \ind_{\bigl\{y \simTc{} x_{2} \bigr\}}\ind_{\bigl\{D^{\check{\Tree}\uparrow}_{x_{2}} - C_{\nu} \leq \xi\bigr\}} (\ee^{\lambda} - 1)\Biggr)\,.
  \end{equation*}
  Taking the expectation conditionally on $\mathcal{F}_{2}(x_{1})$, using the conditional independence, and injecting \eqref{eq:bound-DT-inter} yields
  \begin{equation*}
    \E\biggl[\exp\Bigl(\lambda \sum_{x_{2} \neq x_{1}} \ind_{\bigl\{y \simT{} x_{2} \bigr\}}\ind_{\bigl\{D^{\check{\Tree}\uparrow}_{x_{2}} - C_{\nu} \leq \xi\bigr\}}\Bigr) \,\bigg\vert\, \mathcal{F}_{2}(x_{1})\biggr]
    \leq \prod_{x_{2} \neq x_{1}}\Biggl( 1 + \ind_{\bigl\{y \simTc{} x_{2} \bigr\}}\frac{1}{(\log N)^{\beta}} (\ee^{\lambda} - 1)\Biggr)\,,
  \end{equation*}
  where the sum on the left-hand size and product on the right-hand
  side are over $x_{2} \in V_{x_{1}}$ satisfying $x_{2}\neq x_{1}$ and
  $4\xi \leq w_{x_{2}} \leq 4\nu\log N$.

  We choose $\lambda = \log\log N$ and take the expectation conditionally on $\mathcal{F}_{1}(x_{1})$:
  \begin{equation*}
    \E\biggl[\exp\Bigl(\lambda \sum_{x_{2} \neq x_{1}} \ind_{\bigl\{y \simT{} x_{2} \bigr\}}\ind_{\bigl\{D^{\check{\Tree}\uparrow}_{x_{2}} - C_{\nu} \leq \xi\bigr\}}\Bigr) \,\bigg\vert\, \mathcal{F}_{1}(x_{1})\biggr]
    \leq \prod_{x_{2} \neq x_{1}}\Biggl( 1 + p_{yx_{2}}\frac{1}{(\log N)^{\beta-1}}\Biggr)\,.
  \end{equation*}
  Using $\log(1 + u) \leq u$ and \eqref{eq:proba-edge}, we get
  \begin{equation*}
    \E\biggl[\exp\Bigl(\lambda \sum_{x_{2} \neq x_{1}} \ind_{\bigl\{y \simT{} x_{2} \bigr\}}\ind_{\bigl\{D^{\check{\Tree}\uparrow}_{x_{2}} - C_{\nu} \leq \xi\bigr\}}\Bigr) \,\bigg\vert\, \mathcal{F}_{1}(x_{1})\biggr]
    \leq \exp\Biggl(\frac{w_{y}}{(\log N)^{\beta-1}}\Biggr)\,.
  \end{equation*}
  Assuming $w_{y} \leq 4\nu \log N$ and $\beta > 2$, we can linearize the exponential and get
  \begin{equation*}
    \E\biggl[\exp\Bigl(\lambda \sum_{x_{2} \neq x_{1}} \ind_{\bigl\{y \simT{} x_{2} \bigr\}}\ind_{\bigl\{D^{\check{\Tree}\uparrow}_{x_{2}} - C_{\nu} \leq \xi\bigr\}}\Bigr) \,\bigg\vert\, \mathcal{F}_{1}(x_{1})\biggr]
    \leq 1 + \frac{c w_{y}}{(\log N)^{\beta-1}}
  \end{equation*}
  for $1 < c < 2$.
  We thus have
  \begin{equation*}
    \E \Bigl[\ee^{\lambda \mathcal{P}^{(2)}_{x_{1}}}\Bigr] \leq \E \Biggl[\prod_{y} \left( 1 + \ind_{\bigl\{x_{1} \simT{} y\bigr\}} \frac{c w_{y}}{(\log N)^{\beta-1}} \right)\Biggr]\,.
  \end{equation*}
  By independence, we get
  \begin{equation*}
    \E \Bigl[\ee^{\lambda \mathcal{P}^{(2)}_{x_{1}}}\Bigr] \leq \E \Biggl[\prod_{y} \left( 1 + p_{x_{1}y} \frac{c w_{y}}{(\log N)^{\beta-1}} \right)\Biggr]\,.
  \end{equation*}
  Using $\log (1 + u) \leq u$, \eqref{eq:proba-edge}, and
  \cref{hyp:behavior-m21}, we get
  \begin{equation*}
    \E \Bigl[\ee^{\lambda \mathcal{P}^{(2)}_{x_{1}}}\Bigr]
    \leq \exp \Biggl( \frac{c w_{x_{1}}}{(\log N)^{\beta-1-\delta}} \Biggr)\,.
  \end{equation*}
  Since we assumed that $w_{x_{1}} \leq 4\nu \log N$, we have for $\beta > 3$,
  \begin{equation*}
    \E \Bigl[\ee^{\lambda \mathcal{P}^{(2)}_{x_{1}}}\Bigr] \leq \exp \Biggl(\frac{4c\nu}{(\log N)^{\beta-2-\delta}} \Biggr) \leq N\,.
  \end{equation*}
  Chernoff's bound implies
  \begin{equation*}
    \P \left( \mathcal{P}^{(2)}_{x_{1}} \geq 2\nu \frac{\log N}{\log\log N} \right) = \order\p{N^{-\nu}}\,,
  \end{equation*}
  which is the wanted result.
\end{proof}

\section{Construction of approximate eigenvectors}
\label{sec:finding-eigenvectors}

In this section, we construct a family of orthonormal vectors that
will be close to eigenvectors of $A$. The main intuition to show that
the eigenvectors associated to the greatest eigenvalues are
(semi-)localized, is that these eigenvectors are with high probability close to the
vectors
\begin{equation*}
 \b v_{\sigma}(x) = \frac{1}{\sqrt{2}}\left(\b 1_{x} + \frac{\sigma}{\sqrt{D^{\pp -}_{x}}} \, \b 1_{S^{\pp -}_{1}(x)}\right)\,,
\end{equation*}
where $x \in \vertices$ and $\sigma \in \{\pm 1\}$. This particular
choice is motivated by the fact that such vectors are eigenvectors of
the adjacency matrix of a star of degree $D^{\pp -}_{x}$, i.e.\ a tree
with one vertex connected to $D_{x}^{\pp -}$ leaves.

Notice that these vectors are normalized, but in the pruned graph
$G^{\pp}$, we have (recall \cref{def:order})
\begin{equation*}
\scalar{\b v_{\rho}(x)}{\b v_{\sigma}(y)} = \frac{\delta_{xy}}{2}\left(1 + \rho\sigma\right) + \frac{\ind_{\{x\sim y\}}}{2}\left(\frac{\rho\ind_{\{x \prec y\}}}{\sqrt{D_{x}^{\pp -}}} + \frac{\sigma\ind_{\{x \succ y\}}}{\sqrt{D^{\pp -}_{y}}}\right)\,,
\end{equation*}
that is, we only have $\scalar{\b v_{-}(x)}{\b v_{+}(x)} = 0$ for all $x \in \vertices$, but
in general $\scalar{\b v_{\rho}(x)}{\b v_{\sigma}(y)}$ is not zero for $x \sim y$.

We are going to define a family of pseudo-eigenvectors that are
supported near the vertices in $\hcV_{\nu}^{(\rm h)}$, with the greatest
degrees in the graph $G$.
\begin{remark}\label{rem:one-parent}
  In the pruned graph $G^{\pp}$, for each vertex $x$, there is at most
  one vertex in $S_{1}^{\pp +}(x)$. Indeed, if it were not the case, i.e.\
  if there existed two distinct vertices $y, z \in S_{1}^{\pp +}(x)$, then
  one of $(y, x, z)$ or $(z, x, y)$ would be a down-up path. There are
  no such paths in the pruned graph.
\end{remark}

This remark allows the following definition.
\begin{definition}
  Let $x \in \vertices$. The unique element $y \in S_{1}^{\pp +}(x)$, if it
  exists, is called the \emph{parent} of $x$, and denoted by $\hat{x}$.
  Conversely, the \emph{children} of $x$ are the vertices in
  $S_{1}^{\pp -}(x)$. The elements of $S_{1}^{\pp -}(\hat{x})$ are called
  the \emph{siblings} of $x$. The set of siblings of a vertex $x$ is denoted
  by $\Sib(x)$ and is empty if $x$ has no parent. The set of smaller
  siblings is
  \begin{equation*}
    \Sib^{-}(x) = \Sib(x) \cap \{y \in \vertices \col y \prec x\}.
  \end{equation*}
\end{definition}

We use the following conventions.
  \begin{enumerate}
    \item $\frac{\b 1_{\emptyset}}{0} = 0$,
    \item Any term where the symbol $\hat{x}$ appears is $0$ if the
          vertex $x$ has no parent.
  \end{enumerate}

The orthonormal family we shall use is defined in the following proposition.
\begin{proposition}\label{prop:ps-ev}
  For all $x \in \hcV^{(\rm h)}_{\nu}$, $\sigma \in \{\pm 1\}$, define
  \begin{equation*}
    Z_{x} = 2 + \frac{2}{\#\Sib^{-}(x)}\ind_{\{\#\Sib^{-}(x) \neq 0\}}\,,
  \end{equation*}
  and
  \begin{equation*}
    \b u_{\sigma}(x) = \frac{1}{\sqrt{Z_{x}}} \Biggl( \b 1_{x} + \sigma\frac{\b 1_{S_{1}^{\pp -}(x)}}{\sqrt{D_{x}^{\pp -}}} - \frac{\b 1_{\Sib^{-}(x)}}{\#\Sib^{-}(x)}\Biggr)\,.
  \end{equation*}
  The family $(\b u_{\sigma}(x))_{x \in \hcV^{(\rm h)}_{\nu}, \sigma\in \{\pm 1\}}$ is orthonormal.
\end{proposition}

The family $(\b u_{\sigma}(x))$ is convenient as it is an orthonormal
family of vectors that are localized around the vertices of
$\hcV^{(\rm h)}_{\nu}$. In the sequel, we will consider the operator
\begin{equation}\label{eq:op-u}
  \sum_{\substack{x \in \hcV^{(\rm h)}_{\nu}\\\sigma\in\{\pm 1\}}}\sigma\sqrt{D^{\pp -}_{x}}\b u_{\sigma}(x)\b u_{\sigma}(x)^{*}\,.
\end{equation}
In fact, \cref{prop:approx-hat} stated in the sequel implies that this
is a good approximation of $A^{\pp}$ restricted to the eigenvectors of
its greatest eigenvalues. The operator \eqref{eq:op-u} will be further
approximated by
\begin{equation}\label{eq:op-v}
  \sum_{\substack{x \in \hcV^{(\rm h)}_{\nu}\\\sigma\in\{\pm 1\}}}\sigma\sqrt{D^{\pp -}_{x}}\b v_{\sigma}(x)\b v_{\sigma}(x)^{*}\,,
\end{equation}
in \cref{prop:op-u-v-close}. It is convenient to define
\eqref{eq:op-v} as it has good geometric properties: it is the
adjacency matrix of the pruned graph restricted to the neighborhood of
the vertices in $\hcV^{(\rm h)}$. This geometric insight will be useful in
\cref{sec:bounds-overlinepi}. We now show that the matrices \eqref{eq:op-u} and \eqref{eq:op-v}
are similar.
\begin{proposition}\label{prop:op-u-v-close}
  Let $\nu > 0$. With $\nu$-high probability,
  \begin{equation*}
    \left\|\sum_{x \in \hcV^{(\rm h)}_{\nu}} \sum_{\sigma\in\{\pm 1\}} \sigma\sqrt{D^{\pp -}_{x}}\Bigl(\b u_{\sigma}(x)\b u_{\sigma}(x)^{*} - \b v_{\sigma}(x)\b v_{\sigma}(x)^{*}\Bigr)\right\| \leq \sqrt{9\nu \frac{\log N}{\log\log N}}\,.
  \end{equation*}
\end{proposition}

We now turn to the proofs. Before showing that the vectors
$({\b u}_\sigma(x))$ form an orthogonal family, we explain heuristically
how their expression was found. Below, we discuss a Gram-Schmidt
orthonormalization procedure. The vector obtained by this procedure
will provide us with the ansatz that allowed us to define the
$(\b u_{\sigma}(x))_{x\in\hcV^{(\rm h)}_{\nu}, \sigma \in \left\{ \pm 1 \right\}}$.

Introduce the families of vectors
$(\b V_{0}(x))_{x \in \hcV^{(\rm h)}_{\nu}}, (\b V_{1}(x))_{x \in \hcV^{(\rm h)}_{\nu}}$, defined
by
\begin{equation*}
  \b V_{0}(x) = \b 1_{x}\,, \quad \text{ and }\quad  \b V_{1}(x) = \frac{\b 1_{S_{1}^{\pp -}(x)}}{\sqrt{D^{\pp -}_{x}}}\,.
\end{equation*}
In particular, we have $\b v_{\sigma}(x) = (\b V_{0}(x) + \sigma \b V_{1}(x))/\sqrt{2}$.

The family of vectors we shall consider is the family
$(\b V_{1}(x), \b U_{0}(x))_{x \in \hcV^{(\rm h)}_{\nu}}$ obtained after applying
the Gram-Schmidt orthonormalization procedure on
$(\b V_{1}(x), \b V_{0}(x))_{x \in \hcV^{(\rm h)}_{\nu}}$, starting with the vectors
of $(\b V_{1}(x))$, and ordering the vectors of $(\b V_{0}(x))$
decreasingly according to the strict order $\prec$. Working in the pruned
graph makes this procedure simpler. \cref{prop:pruning} implies that
the family $(\b V_{1}(x))_{x\in\hcV^{(\rm h)}_{\nu}}$ is orthonormal. Indeed, if
$x \neq y$, then $\left< \b V_{1}(x), \b V_{1}(y) \right>$ is nonzero if
and only if there is a down-up path between $x\in \hcV^{(\rm h)}_\nu$ and $y\in \hcV^{(\rm h)}_\nu$,
and in the pruned graph there are no such paths. Thus, we only need to consider the vectors $\b U_0(x), x \in \hcV^{(\rm h)}_{\nu}$.

The vectors $\b U_{0}(x), x \in \hcV^{(\rm h)}_{\nu}$ resulting from the
Gram-Schmidt procedure, are defined by
\begin{equation*}
    \b U_{0}(x) \deq \frac{\tilde{\b U}_{0}(x)}{\|\tilde{\b U}_{0}(x)\|}\,, \qquad
    \tilde{\b U}_{0}(x) \deq \b V_{0}(x) - \sum_{y \in \hcV^{(\rm h)}_\nu} \left< \b V_{0}(x), \b V_{1}(y) \right> \b V_{1}(y) - \sum_{\substack{y \in \hcV^{(\rm h)}_\nu\\y \succ x}} \left< \b V_{0}(x), \b U_{0}(y) \right> \b U_{0}(y)\,.
    \end{equation*}

We then have
\begin{equation}\label{eq:U-tilde}
  \begin{split}
    \tilde{\b U}_{0}(x)
    &= \b 1_{x} - \sum_{y \in \hcV^{(\rm h)}_\nu} \frac{ \left< \b 1_{x}, \b 1_{S_{1}^{\pp -}(y)} \right>}{D^{\pp -}_{y}}\b 1_{S_{1}^{\pp -}(y)} - \sum_{\substack{y \in \hcV^{(\rm h)}_\nu\\y \succ x}} \left< \b 1_{x}, \b U_{0}(y) \right>\b U_{0}(y)\\
    &= \b 1_{x} - \sum_{y \in \hcV^{(\rm h)}_\nu} \frac{\ind_{\{x \sim y, x \prec y\}}}{D^{\pp -}_{y}}\b 1_{S_{1}^{\pp -}(y)} - \sum_{\substack{y \in \hcV^{(\rm h)}_\nu\\y \succ x}} \left< \b 1_{x}, \b U_{0}(y) \right>\b U_{0}(y)\\
    &= \b 1_{x} - \frac{1}{D^{\pp -}_{\hat{x}}}\b 1_{S_{1}^{\pp -}(\hat{x})} - \sum_{y \succ x} \left< \b 1_{x}, \b U_{0}(y) \right>\b U_{0}(y)\,.
  \end{split}
\end{equation}

\begin{lemma}\label{lem:support-U+}
  For all $x \in \hcV^{(\rm h)}_{\nu}$, the vector $\b U_{0}(x)$ is supported on
  $\{x\} \cup \Sib(x)$.
\end{lemma}
\begin{proof}
  We proceed by induction, starting from the vertices that are the
  greatest for the order $\prec$. We first notice that for all
  $x \in \hcV^{(\rm h)}_{\nu}$ which has no parent, we have
  $\tilde{\b U}_{0}(x) = \b 1_{x}$.

  Then, considering \eqref{eq:U-tilde}, we see that
  \begin{equation*}
    \b 1_{x} -  \frac{1}{D^{\pp -}_{\hat{x}}}\b 1_{S_{1}^{\pp -}(\hat{x})}
  \end{equation*}
  is supported on $\{x\} \cup S_{1}^{\pp -}(\hat{x})$.

  By the induction hypothesis $\left\langle \b 1_{x}, \b U_{0}(y)\right\rangle$ with $y \succ x$, is
  non-zero only if $x \in S_{1}^{\pp -}(\hat{y})$, i.e.\ $x$ and $y$ are
  siblings. Thus,
  \begin{equation*}
    \sum_{\substack{y \in \hcV^{(\rm h)}_{\nu}\\y \succ x}} \left< \b 1_{x}, \b U_{0}(y) \right>\b U_{0}(y)
  \end{equation*}
  is supported on $\{x\} \cup S_{1}^{-}(\hat{x})$. Indeed, the siblings of
  the siblings $y$ of $x$ are the siblings of $x$.
\end{proof}

\cref{lem:support-U+} serves as heuristics to prove \cref{prop:ps-ev}.
\begin{proof}[Proof of \cref{prop:ps-ev}]
  Motivated by \cref{lem:support-U+}, we look for vectors
  $\hat{\b U}_{0}(x), x \in \hcV^{(\rm h)}_{\nu}$ of the form
  \begin{equation*}
    \hat{\b U}_{0}(x) = a_{x}\b 1_{x} + \sum_{y \in S_{1}^{\pp -}(\hat{x})}b_{x}(y)\b 1_{y}\,.
  \end{equation*}
  We use again the vectors
  \begin{equation*}
    \b V_{1}(x) = \frac{1}{\sqrt{D^{\pp -}_{x}}}\b 1_{S_{1}^{\pp -}(x)}\,.
  \end{equation*}
  We are going to find necessary and sufficient conditions for the family
  $(\hat{\b U}_{0}(x), \b V_{1}(x))_{x\in \hcV^{(\rm h)}_{\nu}}$ to be
  orthogonal. If it were the case then, for all $x, y \in \hcV^{(\rm h)}_{\nu}$,
  \begin{equation}\label{eq:ortho+-}
    0 = \left\langle\hat{\b U}_{0}(x), \b V_{1}(y)\right\rangle = \delta_{y \hat{x}} \left(\frac{a_{x}}{\sqrt{D_{y}^{\pp -}}} + \sum_{z \in S_{1}^{\pp -}(\hat{x})}\frac{b_{x}(z)}{\sqrt{D_{y}^{\pp -}}}\right)\,.
  \end{equation}
  Furthermore, for all $x, y\in \hcV^{(\rm h)}_{\nu}$,
  \begin{equation}\label{eq:ortho++}
    \delta_{x y} = \left< \hat{\b U}_{0}(x), \hat{\b  U}_{0}(y) \right> = a_{x}^{2}\delta_{xy} + \delta_{\hat{x} \hat{y}} \left( a_{x}b_{y}(x) + a_{y}b_{x}(y) \right) + \delta_{\hat{x} \hat{y}}\sum_{z \in S_{1}^{\pp -}(\hat{x})}b_{x}(z)b_{y}(z)\,.
  \end{equation}
  This implies than if $x = y$
  \begin{equation*}
    1 = a_{x}^{2} + 2a_{x}b_{x}(x) +  \sum_{z \in S_{1}^{\pp -}(\hat{x})}b_{x}(z)^{2}\,,
  \end{equation*}
  and if $x \neq y$ with $\hat{x} = \hat{y}$,
  \begin{equation*}
    0 = a_{x}b_{y}(x) + a_{y}b_{x}(y) + \sum_{z \in S_{1}^{\pp -}(\hat{x})}b_{x}(z)b_{y}(z)\,.
  \end{equation*}

 Consider the particular choice
  \begin{equation*}
    a_{x} = \frac{1}{\sqrt{Z_{x}/2}} \quad \text{ and } \quad b_{x}(y) = -\frac{\ind_{\{y \prec x\}}}{\sqrt{Z_{x}/2}\#\Sib^{-}(x)},
  \end{equation*}
  if $x$ has a parent and $b_{y}(x) = 0$ otherwise. With this choice
  of coefficients, \eqref{eq:ortho+-} and \eqref{eq:ortho++} are satisfied. The family
  \begin{equation*}
    (\b V_{1}(x), \hat{\b U}_{0}(x))_{x\in \hcV^{(\rm h)}_{\nu}} = \left(\frac{\b 1_{S_{1}^{-}(x)}}{\sqrt{D^{\pp -}_{x}}}\,,\, a_{x}\b 1_{x} + \sum_{y \in S_{1}^{\pp -}(\hat{x})}b_{x}(y)\b 1_{y}\right)_{x \in \hcV^{(\rm h)}_{\nu}}
  \end{equation*}
  is then orthonormal. Note however that we do not claim that it is
  the family obtained from the orthonormalization of
  $(\b V_{1}(x), \b V_{0}(x))$. This does not matter: defining
  \begin{equation*}
    \b u_{\sigma}(x) = \frac{\hat{\b U}_{0}(x) + \sigma \b V_{1}(x)}{\sqrt{2}}
  \end{equation*}
  yields an orthonormal family.
\end{proof}

We now turn to the proof of \cref{prop:op-u-v-close}. We first need to prove
Lemmas \ref{lem:descending-ball} and \ref{lem:bound-siblings} to
bound the number of vertices in a ball around a vertex $x$, and of
siblings of a vertex $x$. \cref{lem:descending-ball} will also be
used in the proof of \cref{prop:approx-hat}.
\begin{lemma}\label{lem:descending-ball}
  Let $x \in \vertices$, and $\eta > 0$ and $\nu > 0$ two constants. With
  $\nu$-high probability, for all $x \in \hcV^{(\rm h)}_{\nu}$ we have
  \begin{equation*}
    \frac{1}{D^{\pp}_{x}}\sum_{y \in S_{1}^{\pp -}(x)}D_{y}^{\pp -} \leq 3\nu \frac{\log N}{\log\log N}\,.
  \end{equation*}
\end{lemma}
\begin{lemma}\label{lem:bound-siblings}
  Let $\nu > 0$. With $\nu$-high probability, for all $x \in \hcV^{(\rm h)}_{\nu}$ that
  has a parent, we have
  \begin{equation*}
    \#\Sib^{-}(x) \geq \frac{1}{2}D_{\hat{x}}^{\pp}\,.
  \end{equation*}
\end{lemma}

\begin{proof}[Proof of \cref{lem:descending-ball}]
  We start by noticing that by the union bound
  \begin{equation*}
    \begin{split}
    \P \biggl( \forall x \in \hcV^{(\rm h)}_{\nu}, \frac{1}{D^{\pp}_{x}}\sum_{y \in S_{1}^{\pp -}(x)}D_{y}^{\pp -} &\leq (\log N)^{\eta + \delta} \biggr)
    = 1 - \P \left( \exists x \in \hcV^{(\rm h)}_{\nu}, \frac{1}{D^{\pp}_{x}}\sum_{y \in S_{1}^{\pp -}(x)}D_{y}^{\pp -} > (\log N)^{\eta + \delta} \right)\\
    &\geq 1 - \sum_{x \in \vertices}\P \biggl( x \in \hcV^{(\rm  h)}_{\nu}, \frac{1}{D^{\pp}_{x}}\sum_{y \in S_{1}^{\pp -}(x)}D_{y}^{\pp -} > (\log N)^{\eta + \delta} \biggr).
    \end{split}
  \end{equation*}
  It suffices to upper bound the term in the sum by $CN^{-\nu-1}$ to get the result.

  We introduce the notation
  \begin{equation*}
    \mathcal{P}^{(3)}_{x} \deq \frac{1}{D_x}\sum_{y \in S_{1}^{-}(x)\setminus\cycV(x)}(D^{\nc}_{y} - 1)\,.
  \end{equation*}
  It suffices to bound $\mathcal{P}^{(3)}_{x}$ to bound
  $\sum_{y \in S_{1}^{\pp -}(x)}D_{y}^{\pp -}$. We prove below that
    \begin{equation}\label{eq:intermediate-obj}
      \mathcal{P}^{(3)}_{x} \leq \frac{5\nu}{4(1-\delta)}\frac{\log N}{\log \log N} \quad \text{ with $\nu$-high probability.}
  \end{equation}
    Assuming \eqref{eq:intermediate-obj}, by
  \cref{prop:pruning} we get
  \begin{equation*}
    \frac{1}{D^{\pp}_{x}}\sum_{y \in S_{1}^{\pp -}(x)}D_{y}^{\pp -}
    \leq 1 + \mathcal{P}^{(3)}_{x}\frac{D_{x}}{D_{x}^{\pp}} \leq 1 + \frac{5\nu\log N}{4(1-\delta)\log\log N} \frac{D_{x}}{D_{x} - \xi/2 } \leq  1+2\nu\frac{\log N}{\log\log N}\,.
  \end{equation*}
  which gives the result.

  We turn to the proof of \eqref{eq:intermediate-obj}. Firstly, recall that using
  \cref{rem:ratio-big-dx}, we have with $\nu$-high probability that
  $w_{x} \leq \chi D_{x}$ where $\chi = \log\log N$. If furthermore
  $y \prec x$, we have
  \begin{equation*}
    \frac{D^{\nc}_{y}}{D_{x}} \leq 1 \wedge \frac{\chi D^{\nc}_{y}}{D_{x}}\,.
  \end{equation*}
  Secondly, using \cref{prop:embedding-Tx} and
  \cref{corol:bound-Dnc}, we have
  \begin{equation*}
    \mathcal{P}^{(3)}_{x} \leq \sum_{y \neq x}\ind_{\bigl\{ x \simT{x} y\bigr\}} \left( 1 \wedge \frac{\chi D^{\Tree_{x}\uparrow}_{y}}{w_{x}} \right)\,,
  \end{equation*}
  with $\nu$-high probability.

  We are going to prove \eqref{eq:intermediate-obj} using Bennett's inequality \cite[Theorem 2.9]{BLM13}. Note
  that the variables
  \begin{equation*}
    \ind_{\bigl\{ x \simT{x} y\bigr\}} \left( 1 \wedge \frac{\chi D^{\Tree_{x}\uparrow}_{y}}{w_{x}} \right)
  \end{equation*}
  are independent by \cref{lem:T-indep}, and bounded by $1$. Furthermore, by using \cref{lem:T-indep}  and noticing that $D_y^{\Tree_x\uparrow}$ and $D_y - \ind_{\{ x \sim y\}}$ have the same distribution by construction, we have
  \begin{equation*}
    v \deq \sum_{y\in\vertices}\E\Biggl[\ind_{\bigl\{ x \simT{x} y\bigr\}}\Bigl( 1 \wedge \frac{\chi D^{\Tree_{x}\uparrow}_{y}}{w_{x}} \Bigr) \Biggr] \leq \sum_{y\in\vertices}p_{xy}\frac{\chi w_{y}}{w_{x}} \leq \chi \frac{m_{2}}{m_{1}}\,.
  \end{equation*}
  Hence, we have for $t > 0$
  \begin{equation*}
    \P \left( \mathcal{P}^{(3)}_{x} \geq t + v \right) \leq \exp \left( -(v + t)\log \frac{t}{v} + t \right) + \order\p{N^{-\nu}}\,.
  \end{equation*}
  Taking $t = \frac{5\nu}{4(1 - \delta)} \frac{\log N}{\log\log N}$ proves
  the claim \eqref{eq:intermediate-obj}, and thus the result.
\end{proof}
\begin{proof}[Proof of \cref{lem:bound-siblings}]
  To show this, we shall rather show that with $\nu$ high probability,
  we have the following property. For all $x \in \hcV^{(\rm h)}_{\nu}$, the random
  variable
  \begin{equation*}
    D_{x}^{\nc, \hcV} \deq \# \left( S^{\nc}_{1}(x)\cap \hcV^{(\rm h)}_{\nu} \right)= \sum_{y \neq x}\ind_{\{x \simnc y, \xi < D_{y}\}}
  \end{equation*}
  is bounded by $D_{x}/4$.

  We see that by using \cref{prop:edge-cyc} and then \cref{prop:embedding-Tx}, we get
  \begin{equation*}
    D_{x}^{\nc, \hcV}
    \leq \sum_{y \neq x}\ind_{\{x \simnc y, \xi < D^{\nc}_{y} + C_{\nu}\}}
    \leq \sum_{y \neq x}\ind_{\{x \simT{x} y, \xi \leq D^{\Tree_{x}\uparrow}_{y} + C_{\nu}\}}\,,
  \end{equation*}
  with $\nu$-high probability for some constant $C_{\nu} > 0$.

  Let us use Bennett's inequality to bound the right-hand side. We
  proceed as in the proof of \cref{lem:bound-D+}. The random variables
  \begin{equation*}
    \left( \ind_{\{x \simT{x} y\}}\ind_{\{\xi \leq D^{\Tree_{x}\uparrow}_{y} + C_{\nu}\}} \right)_{y \in \vertices}
  \end{equation*}
  are independent by \cref{lem:T-indep}, and bounded by $1$. We compute
  \begin{equation*}
    v \deq \sum_{y \neq x}\E \left[ \ind_{\{x \simT{x} y, \xi \leq D^{\Tree_{x}\uparrow}_{y} + C_{\nu}\}} \right]
    \leq \sum_{y \neq x} p_{xy} \frac{d_{y}}{\xi - C_{\nu}}
    \leq 2\frac{w_{x}}{\xi}\frac{m_{2}}{m_{1}}\,.
  \end{equation*}
  We get that
  \begin{equation*}
    \P \left( D_{x}^{\nc, \hcV} \geq t\right) \leq \exp\pB{-t\ln \frac{t}{v} + t} + \order\p{N^{-\nu}}\quad \text{ where } t = \frac{w_{x}}{4\chi} \vee \frac{2\nu \log N}{\log\log N}\,.
  \end{equation*}

  If $\frac{w_{x}}{4\chi} \geq \frac{2\nu\log N}{\log\log N}$,
  \begin{equation*}
    \P \left( D_{x}^{\nc, \hcV} \geq t \right)
    \leq \exp\pB{-2\nu(1-\delta)\ln N(1 + o(1))}  + \order\p{N^{-\nu}} = \order\p{N^{-\nu}}\,,
  \end{equation*}
  as $2(1-\delta) \geq 4/3 > 1$. Recall that by \cref{rem:ratio-big-dx} we
  have with $\nu$-high probability that $w_{x} \leq \chi D_{x}$. We then
  have
  \begin{equation*}
    D_{x}^{\nc, \hcV} \leq \frac{w_{x}}{4\chi} \leq \frac{D_{x}}{4}
  \end{equation*}
  with $\nu$-high probability.

  If $\frac{w_{x}}{4\chi} \leq \frac{2\nu\log N}{\log\log N}$,
  \begin{equation*}
    \P \left( D_{x}^{\nc, \hcV} \geq t \right)
    \leq \exp\pB{-2\nu\ln N(1 + o(1))}  + \order\p{N^{-\nu}} = \order\p{N^{-\nu}}\,,
  \end{equation*}
  and then
  \begin{equation*}
    D_{x}^{\nc, \hcV} \leq 2\nu \frac{\log N}{\log\log N} \leq \frac{\xi}{4} \leq \frac{D_{x}}{4}\,,
  \end{equation*}
  with $\nu$-high probability. We have proved the claim.

  For $x \in \hcV^{(\rm h)}_{\nu}$, the neighbors of $\hat{x}$ are in $\Sib^{-}(x)$,
  or in $S_{1}(\hat{x})\cap \hcV^{(\rm h)}_{\nu}$. Thus,
  \begin{equation*}
    \#\Sib^{-}(x) + D_{\hat{x}}^{\nc, \hcV} \geq D_{\hat{x}}^{\pp}\,.
  \end{equation*}
  The previous result on $D_{\hat{x}}^{\nc, \hcV}$ implies that with
  $\nu$-high probability,
  \begin{equation*}
    \#\Sib^{-}(x) \geq D_{\hat{x}}^{\pp} - D_{\hat{x}}^{\nc, \hcV} \geq  D_{\hat{x}}^{\pp} - \frac{1}{4}D_{\hat{x}} \geq \frac{1}{2}D_{\hat{x}}^{\pp}\,,
  \end{equation*}
  where we used that when $x \in \hcV^{(\rm h)}_{\nu}$, by
  \cref{prop:pruning}, $D_{x} - D_{x}^{\pp} \leq \xi/2$ which implies
  $\frac{1}{2}D^{\pp}_{\hat{x}} - \frac{1}{4}D_{\hat{x}}\geq \frac{1}{4}D_{\hat{x}} - \frac{1}{4}\xi\geq 0$.
\end{proof}

\begin{proof}[Proof of \cref{prop:op-u-v-close}]
  First, notice that
\begin{equation*}
    \sum_{\sigma \in \{\pm 1\}}\sigma\sqrt{D_{x}^{\pp -}}\b u_{\sigma}(x)\b u_{\sigma}(x)^{*}
    = \frac{2}{Z_{x}} \left(\sum_{\sigma \in \{\pm 1\}}\sigma\sqrt{D_{x}^{\pp -}}\b v_{\sigma}(x)\b v_{\sigma}(x)^{*} - \frac{\b 1_{\Sib^{-}(x)}\b 1_{S_{1}^{-}(x)}^{*} + \b 1_{S_{1}^{-}(x)}\b 1_{\Sib^{-}(x)}^{*}}{\#\Sib^{-}(x)} \right)\,,
  \end{equation*}
  because of the cancellations occurring when summing the contributions of $\sigma = +1$ and $\sigma = -1$.

  We consider first
  \begin{equation*}
    \sum_{x \in \hcV^{(\rm h)}_{\nu}}\left(\frac{2}{Z_{x}} - 1 \right) \sum_{\sigma \in \{\pm 1\}}\sigma\sqrt{D_{x}^{\pp -}}\b v_{\sigma}(x)\b v_{\sigma}(x)^{*}\,,
  \end{equation*}
  which is zero if $x$ has no parent. Otherwise, its operator norm satisfies
  \begin{equation*}
    \begin{split}
      \Biggl\|\sum_{x \in \hcV^{(\rm h)}_{\nu}}&\left(\frac{2}{Z_{x}} - 1 \right) \sum_{\sigma \in \{\pm 1\}}\sigma\sqrt{D_{x}^{\pp -}}\b v_{\sigma}(x)\b v_{\sigma}(x)^{*}\Biggr\|^{2}
      = \Biggl\|\sum_{x \in \hcV^{(\rm h)}_{\nu}} \frac{\b 1_{x}\b 1_{S_{1}^{\pp -}(x)}^{*} + \b 1_{S_{1}^{\pp -}(x)}\b 1_{x}^{*} }{\#\Sib^{-}(x) + 1}\Biggr\|^{2}\\
                                       &\leq\max_{\|\b u\| = 1}\sum_{x, x' \in \hcV^{(\rm h)}_{\nu}} \frac{1}{(\#\Sib^{-}(x) + 1)^{2}}\Biggl( \delta_{x x'} \left< \b 1_{S_{1}^{\pp -}(x)}, \b u \right>^{2} + \delta_{x x'}D_{x}^{\pp -}u_{x}^{2}\\
      &\quad\quad\quad+ 2\ind_{\{x \sim x', x' \prec x\}}u_{x} \left< \b 1_{S_{1}^{\pp -}(x')}, \b u \right> \Biggr)\,.
    \end{split}
  \end{equation*}

  \cref{lem:bound-siblings} implies that with $\nu$-high probability,
  for all $x \in \hcV^{(\rm h)}_{\nu}$,
  $\frac{1}{2}D_{x}^{\pp -} \leq \#\Sib^{-}(x)$. Together with Young's
  lemma and the fact that $\b 1_{S_{1}^{\pp -}(x)}$ form an orthogonal
  family, it implies
  \begin{equation*}
    \left\|\sum_{x \in \hcV^{(\rm h)}_{\nu}}\left(\frac{2}{Z_{x}} - 1 \right) \sum_{\sigma \in \{\pm 1\}}\sigma\sqrt{D_{x}^{\pp -}}\b v_{\sigma}(x)\b v_{\sigma}(x)^{*}\right\|^{2} = \order\p{1/\xi}\,.
  \end{equation*}

  Consider now the operator norm
  \begin{equation*}
    \begin{split}
      \Bigg\|\sum_{x \in \hcV^{(\rm h)}_{\nu}}\frac{\b 1_{\Sib^{-}(x)}\b 1_{S_{1}^{\pp -}(x)}^{*}}{\#\Sib^{-}(x)}\Bigg\|^{2}
                                   &\leq \max_{\|\b u\| = 1}\sum_{x, x' \in \hcV^{(\rm h)}_{\nu}}\delta_{x x'}\frac{D_{x}^{\pp -} \left<  \b 1_{\Sib^{-}(x)}, \b u\right> \left<  \b 1_{\Sib^{-}(x')}, \b u \right>}{\left( \#\Sib^{-}(x) \right)^{2}}\\
                                   &+ \max_{\|\b u\| = 1}2\sum_{x, x' \in \hcV^{(\rm h)}_{\nu}}\delta_{\hat{x} \hat{x}'}\ind_{\Sib^{-}(x)}(x')\frac{ \left<  \b 1_{S_{1}^{\pp -}(x)}, \b u\right> \left< \b 1_{S_{1}^{\pp -}(x')}, \b u  \right>}{\#\Sib^{-}(x)}.
    \end{split}
  \end{equation*}
The first term is bounded by
  \begin{equation*}
    \begin{split}
      \sum_{x, x' \in \hcV^{(\rm h)}_{\nu}}\delta_{x x'}\frac{D_{x}^{\pp -} \left< \b 1_{\Sib^{-}(x)}, \b u \right> \left< \b 1_{\Sib^{-}(x')}, \b u  \right>}{\left( \#\Sib^{-}(x) \right)^{2}}
      &= \sum_{x \in \hcV^{(\rm h)}_{\nu}}\frac{D_{x}^{\pp -}}{\#\Sib^{-}(x)^{2}}\sum_{y, y' \in \Sib^{-}(x)}u_{y}u_{y'}\\
      &\leq 2\sum_{x \in \hcV^{(\rm h)}_{\nu}}\frac{D_{x}^{\pp -}}{D_{\hat{x}}^{\pp -}}\sum_{y\in \Sib^{-}(x)}u_{y}^{2}\,,
    \end{split}
  \end{equation*}
  with $\nu$-high probability. We used \cref{lem:bound-siblings} and Young's
  inequality. This can be bounded as follows:
  \begin{equation*}
      \sum_{x, x' \in \hcV^{(\rm h)}_{\nu}}\delta_{x x'}\frac{D_{x}^{\pp -} \left< \b 1_{\Sib^{-}(x)}, \b u \right> \left< \b 1_{\Sib^{-}(x')}, \b u \right>}{\left( \#\Sib^{-}(x) \right)^{2}}
      \leq 2\sum_{y}u_{y}^{2}\sum_{x \in \hcV^{(\rm h)}_{\nu}\cap S_{1}^{\pp -}(\hat{y})}\frac{D_{x}^{\pp -}}{D_{\hat{y}}^{\pp -}}\,.
  \end{equation*}
  \cref{lem:descending-ball} then gives that this is upper bounded by
  $2\nu \frac{\log N}{\log\log N}$ with $\nu$-high probability.

  Similarly, using Young's inequality, \cref{lem:bound-siblings},
  and then \cref{lem:descending-ball} we see that the second term
  is bounded by
  \begin{equation*}
    \begin{split}
      \sum_{x, x' \in \hcV^{(\rm h)}_{\nu}}\delta_{\hat{x} \hat{x}'}\ind_{\Sib^{-}(x)}(x')\frac{ \left< \b 1_{S_{1}^{\pp -}(x)}, \b u \right> \left<  \b 1_{S_{1}^{\pp -}(x')}, \b u \right>}{\#\Sib^{-}(x)}
      &\leq 2\sum_{x\in \hcV^{(\rm h)}_{\nu}}\left( \frac{ \left< \b 1_{S_{1}^{\pp -}(x)}, \b u \right>}{\sqrt{D_{x}^{\pp -}}} \right)^{2}\sum_{x' \in S_{1}^{\pp -}(\hat{x})\cap\hcV^{(\rm h)}_{\nu}}\frac{D_{x'}^{\pp -}}{D_{\hat{x}}^{\pp -}}\\
      &\leq 2\nu \frac{\log N}{\log\log N}\sum_{x\in \hcV^{(\rm h)}_{\nu}}\left( \frac{ \left< \b 1_{S_{1}^{\pp -}(x)}, \b u \right>}{\sqrt{D_{x}^{\pp -}}} \right)^{2}\,,
    \end{split}
  \end{equation*}
  with $\nu$-high probability. The fact
  that $(V_{1}(x))_{x \in \hcV^{(\rm h)}_{\nu}}$ is an orthonormal family allow us
  to conclude that
  \begin{equation*}
    \Bigg\|\sum_{x \in \hcV^{(\rm h)}_{\nu}}\frac{\b 1_{\Sib^{-}(x)}\b 1_{S_{1}^{\pp -}(x)}^{*} + \b 1_{S_{1}^{\pp -}(x)}\b 1_{\Sib^{-}(x)}^{*}}{\#\Sib^{-}(x)}\Bigg\| \leq 2\Bigg\|\sum_{x \in \hcV^{(\rm h)}_{\nu}}\frac{\b 1_{\Sib^{-}(x)}\b 1_{S_{1}^{\pp -}(x)}^{*}}{\#\Sib^{-}(x)}\Bigg\| \leq \sqrt{8\nu \frac{\log N}{\log\log N}}
  \end{equation*}
  with $\nu$-high probability. Putting together the two bounds, we get
  the result.
\end{proof}

\section{The spectral gap and the semilocalization phenomenon}
\label{sec:spectral-gap-semi}

In this Section, we prove \cref{thm:main-intro}. To do so, we construct a
block-diagonal approximation $\hat{A}$ of $A$, whose eigenvectors
associated with the extremal eigenvalues are $\b u_{\sigma}(x)$ for
$\sigma \in \{\pm 1\}$ and $x \in \hcV^{(\rm h)}_{\nu}$. A spectral gap property is then proved
for $\hat{A}$, and can be transferred to a spectral gap property for
$A$.

\subsection{The block-diagonal approximation}
\label{sec:block-diag-appr}

We introduce some notation. We define the orthogonal
projections
\begin{equation}\label{eq:projections-u}
    \Pi^{\pp} = \sum_{\substack{x\in\hcV^{(\rm h)}_{\nu}\\\sigma \in \{\pm 1\}}}\b u_{\sigma}(x)\b u_{\sigma}(x)^{*} \quad \text{ and }\quad
    \overline{\Pi}^{\pp} = \Id - \Pi^{\pp}\,,
\end{equation}
and the block-diagonal approximation of $A^{\pp}$,
\begin{equation}\label{eq:block-approx}
  \hat{A} = \sum_{\substack{x\in\hcV^{(\rm h)}_{\nu}\\\sigma \in \{\pm 1\}}}\sigma\sqrt{D^{\pp -}_{x}}\b u_{\sigma}(x)\b u_{\sigma}(x)^{*} + \overline{\Pi}A^{\pp}\overline{\Pi}\,.
\end{equation}

\begin{proposition}\label{prop:approx-hat}
Let $\nu > 0$. There exists $C_{\nu} > 0$ such that with $\nu$-high probability,
  \begin{equation*}
      \|\hat{A} - A^{\pp}\| \leq C_{\nu} \sqrt{\frac{\log N}{\log \log N}} \quad \text{ and } \quad \|\hat{A} - A\| \leq C_{\nu}\sqrt{\frac{\log N}{\log \log N}}\,.
  \end{equation*}
\end{proposition}

To prove \cref{prop:approx-hat} we use the following Lemma.
\begin{lemma}\label{lem:expr-error}
  Define for all $x \in \hcV^{(\rm h)}_{\nu}, \sigma\in\{\pm 1\}$,
  \begin{equation*}
    \delta_{\sigma}(x) = A^{\pp}\b u_{\sigma}(x) - \sigma\sqrt{D_{x}^{\pp-}}\b u_{\sigma}(x)\,.
  \end{equation*}
  This vector can also be expressed as
  \begin{equation*}
    \delta_{\sigma}(x) = \frac{1}{\sqrt{Z_{x}}} \left(\sigma\sum_{y \in S_{1}^{\pp -}(x)}\frac{\b 1_{S_{1}^{\pp}(y)\setminus \{x\}}}{\sqrt{D_{x}^{\pp -}}} - \sum_{y\in S_{1}^{\pp -}(\hat{x})}\frac{\ind_{\{y \prec x\}}}{\#\Sib^{-}(x)}\b 1_{S^{\pp-}_{1}(y)} +  \sigma\sqrt{D_{x}^{\pp -}}\frac{\b 1_{\Sib^{-}(x)}}{\#\Sib^{-}(x)}\right)\,.
  \end{equation*}
\end{lemma}

\begin{proof}
  We compute $A^{\pp}\b u_{\sigma}(x)$:
  \begin{equation*}
    \begin{split}
      A^{\pp}\b u_{\sigma}(x)
      &= \frac{1}{\sqrt{Z_{x}}}\left(\b 1_{S_{1}^{\pp}(x)} + \frac{\sigma}{\sqrt{D^{\pp -}_{x}}}\Biggl(D_{x}^{\pp -}\b 1_{x} + \sum_{y \in S_{1}^{\pp -}(x)}\b 1_{S_{1}^{\pp}(y)\setminus \{x\}}\Biggr) - \sum_{y \in S_{1}^{\pp -}(\hat{x})}\frac{\ind_{\{y \prec x\}}}{\#\Sib^{-}(x)}\b 1_{S_{1}^{\pp}(y)}\right)\\
      &= \sigma\sqrt{D_{x}^{-}}\b u_{\sigma}(x)\\ &+ \frac{\sigma}{\sqrt{Z_{x}}} \left(  \sum_{y \in S_{1}^{-}(x)}\frac{\b 1_{S_{1}^{\pp}(y)\setminus \{x\}}}{\sqrt{D^{-}_{x}}} - \sigma\sum_{y \in S_{1}^{\pp -}(\hat{x})}\frac{\ind_{\{y \prec x\}}}{\#\Sib^{-}(x)}\b 1_{S_{1}^{\pp}(y)\setminus\{\hat{x}\}} +\sqrt{D_{x}^{\pp -}}\frac{\b 1_{\Sib^{-}(x)}}{\# \Sib^{-}(x)}\right)\,.
    \end{split}
  \end{equation*}
  The result follows when noticing that in the pruned graph
  $S_{1}^{\pp}(y)\setminus \{\hat{y}\} = S_{1}^{\pp -}(y)$.
\end{proof}

\begin{proof}[Proof of \cref{prop:approx-hat}]
  The triangular inequality yields
  \begin{equation*}
    \|\hat{A} - A\| \leq \|\hat{A} - A^{\pp}\| + \|A^{\pp} - A\|\,.
  \end{equation*}
\cref{prop:error-pruning} allows us to bound the second
  part. The first part can be bounded as
  \begin{equation*}
    \|\hat{A} - A^{\pp}\|
    \leq \Biggl\|\sum_{\substack{x\in\hcV^{(\rm h)}_{\nu}\\\sigma \in \{\pm 1\}}}\sigma\sqrt{D^{-}_{x}}\b u_{\sigma}(x)\b u_{\sigma}(x)^{*} - \Pi^{\pp}A^{\pp}\Pi^{\pp}\Biggr\|\\
    + \|\overline{\Pi}^{\pp}A^{\pp}\Pi^{\pp} + \Pi^{\pp}A^{\pp}\overline{\Pi}^{\pp}\|\,.
  \end{equation*}

Recall that in \cref{lem:expr-error}, we defined the vector
  \begin{equation*}
    \delta_{\sigma}(x) = A^{\pp}\b u_{\sigma}(x) - \sigma\sqrt{D^{\pp -}_{x}}\b u_{\sigma}(x) \quad \text{ for all }x\in \hcV^{(\rm h)}_{\nu}, \sigma \in \{\pm 1\}\,.
  \end{equation*}
  We have
\begin{equation*}
    A^{\pp}\Pi^{\pp}
    = \sum_{\substack{x\in\hcV^{(\rm h)}_{\nu}\\\sigma \in \{\pm 1\}}}A^{\pp}\b u_{\sigma}(x)\b u_{\sigma}(x)^{*}
    = \sum_{\substack{x\in\hcV^{(\rm h)}_{\nu}\\\sigma \in \{\pm 1\}}}\sigma\sqrt{D^{\pp -}_{x}}\b u_{\sigma}(x)\b u_{\sigma}(x)^{*} + B\,,
  \end{equation*}
  where $B = \sum_{\substack{x\in\hcV^{(\rm h)}_{\nu}\\\sigma \in \{\pm 1\}}}\delta_{\sigma}(x)\b u_{\sigma}(x)^{*}$. We
  now bound the operator norm of $B$. Using the expression of
  $\delta_{\sigma}(x)$ from \cref{lem:expr-error}, we write the operator as a sum
  $B = B_{1} + B_{2} + B_{3} + B_{4}$, with
  \begin{equation*}
    \begin{split}
      B_{1} &= \sum_{x\in\hcV^{(\rm h)}_{\nu}} \frac{2}{Z_{x}}\sum_{y \in S_{1}^{-}(x)}\frac{\b 1_{S_{1}^{\pp -}(y)}\b 1_{S_{1}^{\pp -}(x)}^{*}}{D_{x}^{\pp -}}\\
      B_{2} &= \sum_{x\in\hcV^{(\rm h)}_{\nu}} \frac{2}{Z_{x}}\sum_{\substack{y \in S_{1}^{\pp -}(\hat{x})\\y \prec x}}\frac{\b 1_{y}\b 1_{S_{1}^{\pp -}(x)}^{*}}{\#\Sib^{-}(x)}\\
      B_{3} &= -\sum_{x\in\hcV^{(\rm h)}_{\nu}} \frac{2}{Z_{x}}\sum_{\substack{y \in S_{1}^{\pp -}(\hat{x})\\y \prec x}}\frac{\b 1_{S_{1}^{\pp -}(y)}\b 1_{x}^{*}}{\#\Sib^{-}(x)}\\
      B_{4} &= \sum_{x\in\hcV^{(\rm h)}_{\nu}} \frac{2}{Z_{x}}\sum_{\substack{y, z \in S_{1}^{\pp -}(\hat{x})\\y, z \prec x}}\frac{\b 1_{S_{1}^{\pp -}(y)}\b 1_{z}^{*}}{(\#\Sib^{-}(x))^{2}}\,.
    \end{split}
  \end{equation*}
  We now bound the operator
  norm of each of the four operators. The first one is
  \begin{equation*}
    \begin{split}
      \|B_{1}\|^{2}
      &= \max_{\|\b u\| = 1}(\b u^{*}B_{1}^{*}B_{1}\b u)\\
      &= \max_{\|\b u\| = 1}\sum_{x, x'\in\hcV^{(\rm h)}_{\nu}}\frac{4}{Z_{x} Z_{x'}}\frac{\left< \b u, \b 1_{S_{1}^{\pp -}(x)} \right> \left< \b 1_{S_{1}^{\pp -}(x')}, \b u \right>}{D_{x}^{\pp -}D_{x'}^{\pp -}} \sum_{\substack{y \in S_{1}^{\pp -}(x)\\y' \in S_{1}^{\pp -}(x')}} \left< \b 1_{S_{1}^{\pp -}(y)}, \b 1_{S_{1}^{\pp -}(y')} \right>\,.
    \end{split}
  \end{equation*}
  The orthogonality of the vectors
  $(\b 1_{S_{1}^{\pp -}(y)})_{y \in \vertices} = (\sqrt{D_{y}^{\pp -}}V_{1}(y))_{y \in \vertices}$
  implies that we must take $y = y'$ in the sum above. As
  $x = \hat{y}$ and $x' = \hat{y}'$ this means that the contributions
  for $x \neq x'$ vanish. We are left with
  \begin{equation*}
    \|B_{1}\|^{2} = \max_{\|\b u\| = 1}\sum_{x\in\hcV^{(\rm h)}_{\nu}}\frac{4}{Z_{x}^{2}}\frac{\left< \b u, \b 1_{S_{1}^{\pp -}(x)} \right>^{2}}{D_{x}^{\pp -}} \leq  \max_{\|\b u\| = 1}\sum_{x\in\hcV^{(\rm h)}_{\nu}}\left< \b u, \b V_{1}(x) \right>^{2} \leq 1\,.
  \end{equation*}

  The other bounds are proved similarly, in some cases with the help
  of \cref{lem:descending-ball}, in \cref{sec:bounds-proof-prop}.
  Thus, there exists a constant $C_{\nu}' > 0$ such that with $\nu$-high
  probability,
  \begin{equation*}
    \|A^{\pp}\Pi^{\pp}\| \leq C'_{\nu} \sqrt{\frac{\log N}{\log\log N}}\,.
  \end{equation*}
  The result follows from the fact that $\Pi^{\pp}$ and
  $\overline{\Pi}^{\pp}$ are orthogonal projections.
\end{proof}

\subsection{Bounds on $\overline{\Pi}A^{\pp}\overline{\Pi}$}
\label{sec:bounds-overlinepi}

The last step before proving \cref{thm:main-intro} is to bound the operator norm
of $\overline{\Pi}A^{\pp}\overline{\Pi}$. To do so, we use the fact that the
pruned graph is actually a forest.

Let us consider $\overline{\Pi}A^{\pp}\overline{\Pi}$. By definition,
see \eqref{eq:projections-u}, we have
\begin{multline}\label{eq:three-terms}
    \overline{\Pi}A^{\pp}\overline{\Pi}
    = \hat{A} - \sum_{\substack{x\in\hcV^{(\rm h)}_{\nu}\\\sigma \in \{\pm 1\}}}\sigma\sqrt{D^{\pp -}_{x}}\b u_{\sigma}(x)\b u_{\sigma}(x)^{*}
    = (\hat{A} - A^{\pp})\\
    - \sum_{\substack{x\in\hcV^{(\rm h)}_{\nu}\\\sigma \in \{\pm 1\}}}\sigma\sqrt{D^{\pp -}_{x}}\Bigl(\b u_{\sigma}(x)\b u_{\sigma}(x)^{*} - \b v_{\sigma}(x)\b v_{\sigma}(x)^{*}\Bigr)
    + \Bigl(A^{\pp} - \sum_{\substack{x\in\hcV^{(\rm h)}_{\nu}\\\sigma \in \{\pm 1\}}}\sigma\sqrt{D^{\pp -}_{x}}\b v_{\sigma}(x)\b v_{\sigma}(x)^{*}\Bigr)\,.
\end{multline}
\cref{prop:approx-hat} implies that the error
$\|A^{\pp} - \hat{A}\|$ is at most of order
$\sqrt{\frac{\log N}{\log \log N}}$, and \cref{prop:op-u-v-close} that
the second part is at most of order
$\sqrt{\frac{\log N}{\log\log N}}$. We now prove that the third part
is of order $\sqrt{\frac{\log N}{\log\log N}}$.

\begin{lemma}\label{lem:bound-rest-graph}
  We have
  \begin{equation*}
    \Biggl\|A^{\pp} - \sum_{\substack{x\in\hcV^{(\rm h)}_{\nu}\\\sigma \in \{\pm 1\}}}\sigma\sqrt{D^{\pp -}_{x}}\b v_{\sigma}(x)\b v_{\sigma}(x)^{*}\Biggr\| \leq 2\sqrt{\xi_{\nu}}\,.
  \end{equation*}
\end{lemma}
\begin{proof}
  The matrix
  \begin{equation*}
    A' = A^{\pp} - \sum_{\substack{x\in\hcV^{(\rm h)}_{\nu}\\\sigma \in \{\pm 1\}}}\sigma\sqrt{D^{\pp -}_{x}}\b v_{\sigma}(x)\b v_{\sigma}(x)^{*}
  \end{equation*}
  is the adjacency matrix of the graph
  $\tilde{G} = G^{\pp}\vert_{( \hcV^{(\rm h)}_{\nu})^{c}}$. The degree of the vertices of
  this graph are bounded by $\xi_{\nu}$, and the graph is actually a forest by
\cref{prop:pruning}.
Thus, a standard estimate -- see for instance \cite[Lemma
  A.4]{ADK20} -- implies that
  \begin{equation*}
    \|A'\| \leq 2\sqrt{\xi_{\nu}}\,. \qedhere
  \end{equation*}
\end{proof}

Together with \cref{prop:approx-hat}, \cref{prop:op-u-v-close}, and \eqref{eq:three-terms},  \cref{lem:bound-rest-graph} yields the following bound on
$\|\overline{\Pi}^{\pp}A^{\pp}\overline{\Pi}^{\pp}\|$.
\begin{proposition}\label{prop:bound-block-error}
  Let $\nu > 0$. There exists a constant $C_{\nu} > 0$ such that with
  $\nu$-high probability, we have
  \begin{equation*}
    \|\overline{\Pi}^{\pp}A^{\pp}\overline{\Pi}^{\pp}\| \leq C_{\nu}\sqrt{\frac{\log N}{\log\log N}}\,.
  \end{equation*}
  \end{proposition}

\subsection{Semilocalization}
\label{sec:semilocalization}

We now prove the main result. First, we introduce
  some notation.
  Let $\eta > 0$ and $\lambda > 0$. We define the sets
\begin{equation*}
  \W^{\pp}_{\lambda, \eta} = \Bigl\{x \in [N]\col \bigl|\sqrt{D_{x}^{\pp -}} - \lambda\bigr| \leq \eta\Bigr\} \quad \text{ and } \quad \W_{\lambda, \eta} = \Bigl\{x \in [N]\col \bigl|\sqrt{D_{x}} - \lambda\bigr| \leq \eta\Bigr\}\,,
\end{equation*}
and the orthogonal projections
\begin{equation*}
  \Pi_{\lambda, \eta}^{\pp} = \sum_{x \in \W^{\pp}_{\lambda, \eta}}\b u_{+}(x)\b u_{+}(x)^{*} \quad \text{ and }\quad \overline{\Pi}_{\lambda, \eta}^{\pp} = 1 - \Pi_{\lambda, \eta}^{\pp}\,.
\end{equation*}

The proof of our main theorem, \cref{thm:main-intro}, is from now on very close
to the one of \cite[Theorem 3.4]{ADK20}.
\begin{proof}[Proof of \cref{thm:main-intro}]
  As in \cite[Theorem 3.4]{ADK20}, the core of the proof is the spectral
  gap property of the form
  \begin{equation*}
    \Spec\Bigl(\overline{\Pi}^{\pp}_{\lambda, \eta}A\overline{\Pi}^{\pp}_{\lambda, \eta}\Bigr) \subset \Real \setminus [\lambda - \eta, \lambda + \eta]\,.
  \end{equation*}

Consider first the block-diagonal approximation $\hat{A}$. The
  orthogonal projections $\Pi^{\pp}$ and $\Pi^{\pp}_{\lambda, \eta}$ commute and we
  have the inclusion property
  \begin{equation*}
    \Pi^{\pp}\Pi^{\pp}_{\lambda, \eta} = \Pi^{\pp}_{\lambda, \eta}\,.
  \end{equation*}
  Note that we also have
  \begin{equation*}
    \overline{\Pi}^{\pp}_{\lambda, \eta} = \Id - \Pi^{\pp}\Pi^{\pp}_{\lambda, \eta} = \overline{\Pi}^{\pp} + \overline{\Pi}^{\pp}_{\lambda, \eta}\Pi^{\pp}\,.
  \end{equation*}
  These properties allow us to rewrite $\overline{\Pi}^{\pp}_{\lambda, \eta}\hat{A}\overline{\Pi}^{\pp}_{\lambda, \eta}$ as
  \begin{equation}\label{eq:split-hat}
    \overline{\Pi}^{\pp}_{\lambda, \eta}\hat{A}\overline{\Pi}^{\pp}_{\lambda, \eta} = \overline{\Pi}^{\pp}_{\lambda, \eta}\Pi^{\pp}\hat{A}\Pi^{\pp}\overline{\Pi}^{\pp}_{\lambda, \eta} + \overline{\Pi}^{\pp}\hat{A}\overline{\Pi}^{\pp}\,.
  \end{equation}

The spectral gap property can be shown for
  $\Pi^{\pp}_{\lambda, \eta}\hat{A}\Pi^{\pp}_{\lambda, \eta}$ by showing it for each of the
  two terms in \eqref{eq:split-hat}: they are the two blocks of a block decomposition of
  the operator. By definition, we immediately have
  \begin{equation*}
    \Spec\Bigl(\overline{\Pi}^{\pp}_{\lambda, \eta}\Pi^{\pp}\hat{A}\Pi^{\pp}\overline{\Pi}^{\pp}_{\lambda, \eta}\Bigr) \subset \Real \setminus [\lambda-\eta, \lambda + \eta]\,.
  \end{equation*}
  For the second part, \cref{prop:bound-block-error} implies that there exists a
  constant $C_{\nu} > 0$ such that with $\nu$-high probability,
\begin{equation*}
    \Spec\Bigl(\overline{\Pi}^{\pp}\hat{A}\overline{\Pi}^{\pp}\Bigr) = \Spec\Bigl(\overline{\Pi}^{\pp}A^\pp\overline{\Pi}^{\pp}\Bigr) \subset \Biggl[-\frac{C_{\nu}}{2}\sqrt{\frac{\log N}{\log\log N}}\,, \frac{C_{\nu}}{2}\sqrt{\frac{\log N}{\log\log N}}\,\Biggr] \subset \Real \setminus [\lambda-\eta, \lambda+\eta]\,,
  \end{equation*}
  as we chose $\eta$ such that $|\lambda \pm \eta| \geq |\lambda| / 2 > C_{\nu}\sqrt{\frac{\log N}{\log\log N}}/2$.

Since the two summands in \eqref{eq:split-hat} commute, we have proved
  \begin{equation*}
    \Spec\Bigl(\overline{\Pi}^{\pp}_{\lambda, \eta}\hat{A}\overline{\Pi}^{\pp}_{\lambda, \eta}\Bigr)\subset \Real \setminus [\lambda-\eta, \lambda+\eta]\,.
  \end{equation*}
\cref{prop:approx-hat} allows to upgrade this to a spectral gap property for
  $\overline{\Pi}^{\pp}_{\lambda, \eta}A\overline{\Pi}^{\pp}_{\lambda, \eta}$: there exists
  $c_{\nu} > 0$ such that with $\nu$-high probability
  \begin{equation*}
    \Spec\Bigl(\overline{\Pi}^{\pp}_{\lambda, \eta}A\overline{\Pi}^{\pp}_{\lambda, \eta}\Bigr)
    \subset \Real \setminus \Biggl[\lambda - \Biggl(\eta - c_{\nu}\sqrt{\frac{\log N}{\log\log N}}\,\Biggr), \lambda + \Biggl(\eta - c_{\nu}\sqrt{\frac{\log N}{\log\log N}}\,\Biggr)\Biggr]\,.
  \end{equation*}
  By convention, if $\eta - c_{\nu}\sqrt{\frac{\log N}{\log\log N}} > 0$, then the interval
  is $\emptyset$. Assume that $ \eta > c_{\nu}\sqrt{\frac{\log N}{\log\log N}}$, otherwise the
  result is vacuous.

  We now conclude as follows. Let $\lambda$ be an eigenvalue associated to
  a normalized eigenvector $\b q$, we have $(A - \lambda)\b q = 0$. Multiplying
  by $\overline{\Pi}^{\pp}_{\lambda, \eta}$ and introducing
  $\Id = \Pi^{\pp}_{\lambda, \eta}+\overline{\Pi}^{\pp}_{\lambda, \eta}$, we get
  \begin{equation*}
    \overline{\Pi}^{\pp}_{\lambda, \eta}(A - \lambda)\overline{\Pi}^{\pp}_{\lambda, \eta}\b q + \overline{\Pi}^{\pp}_{\lambda, \eta}(A - \lambda)\Pi^{\pp}_{\lambda, \eta}\b q = 0\,,
  \end{equation*}
  which simplifies to
  \begin{equation*}
    (\overline{\Pi}^{\pp}_{\lambda, \eta}A\overline{\Pi}^{\pp}_{\lambda, \eta} - \lambda)\overline{\Pi}^{\pp}_{\lambda, \eta}\b q = -\overline{\Pi}^{\pp}_{\lambda, \eta}A\Pi^{\pp}_{\lambda, \eta}\b q\,.
  \end{equation*}
  Finally, we have
  \begin{equation*}
    \overline{\Pi}^{\pp}_{\lambda, \eta}\b q = -(\overline{\Pi}^{\pp}_{\lambda, \eta}A\overline{\Pi}^{\pp}_{\lambda, \eta} - \lambda)^{-1}\overline{\Pi}^{\pp}_{\lambda, \eta}A\Pi^{\pp}_{\lambda, \eta}\b q\,.
  \end{equation*}

  The spectral gap property and
\cref{prop:approx-hat} imply the bounds
  \begin{equation*}
    \begin{split}
      \|(\overline{\Pi}^{\pp}_{\lambda, \eta}A\overline{\Pi}^{\pp}_{\lambda, \eta} - \lambda)^{-1}\| &\leq \frac{1}{\eta - c_{\nu}\sqrt{\frac{\log N}{\log\log N}}}\\
      \|\overline{\Pi}^{\pp}_{\lambda, \eta}A\Pi^{\pp}_{\lambda, \eta}\| &\leq \|\hat{A} - A\| \leq c_{\nu}\sqrt{\frac{\log N}{\log\log N}}\,.
    \end{split}
  \end{equation*}
  They allow us to deduce
  \begin{equation}\label{eq:last-eq-main-thm}
      \|\overline{\Pi}^{\pp}_{\lambda, \eta}\b q\| \leq \frac{c_{\nu}\sqrt{\frac{\log N}{\log\log N}}}{\eta - c_{\nu}\sqrt{\frac{\log N}{\log\log N}}}\wedge 1 \leq \frac{2c_{\nu}\sqrt{\frac{\log N}{\log\log N}}}{\eta}\,,
    \end{equation}
  which can be restated as
   \begin{equation*}
    \sum_{x\in \W^{\pp}_{\lambda, \eta}}\langle \b q, \b u_{\sigma}(x) \rangle^{2} \geq 1 - \left(\frac{c_{\nu}}{\eta}\sqrt{\frac{\log N}{\log\log N}}\,\right)^{2}\,,
  \end{equation*}
  if $\sigma = \sign \lambda$.

  To go from a result on the set $\W^\pp_{\lambda, \eta}$ to the result on $\W_{\lambda, \eta}$, we notice that by \cref{prop:pruning}, with $\nu$-high probability,
  $D_{x} - \xi/2 -1 \leq D_{x}^{\pp -} = D_{x}^{\pp} - 1 \leq D_{x}$. Thus, assuming $\lambda, \eta \geq \sqrt{\xi}/2$
  we get
  \begin{equation}\label{eq:inclusion-W}
    \W^{\pp}_{\lambda, \eta - \sqrt{\xi/2}} \subset \W_{\lambda, \eta}\,.
  \end{equation}
  Set $\tilde{\eta} = \eta - \sqrt{\frac{\xi}{2}}$. The inclusion \eqref{eq:inclusion-W}
  implies that for all vectors $\b q$,
  \begin{equation*}
    \normBB{\Biggl( \Id - \sum_{x \in \W_{\lambda, \eta}}\b u_{\sigma}(x)\b u_{\sigma}(x)^{*} \Biggr)\b q} \leq \|\overline{\Pi}^{\pp}_{\lambda, \tilde{\eta}}\b q\|\,.
  \end{equation*}
  where $\sigma = \sign \lambda$.

  We have by \eqref{eq:last-eq-main-thm} that with $\nu$-high probability,
  \begin{equation*}
    \normBB{ \Biggl( \Id - \sum_{x \in \W_{\lambda, \eta}}\b u_{\sigma}(x)\b u_{\sigma}(x)^{*} \Biggr)\b q} \leq \|\overline{\Pi}^{\pp}_{\lambda, \tilde{\eta}}\b q\| \leq \frac{c_{\nu}\sqrt{\frac{\log N}{\log\log N}}}{\tilde{\eta}} = \frac{c_{\nu}\sqrt{\frac{\log N}{\log\log N}}}{\eta- \sqrt{\xi/2}}\,.
  \end{equation*}
  Note that the result is vacuous if $\eta$ is not of order $\sqrt{\frac{\log N}{\log\log N}}$: either $\eta > \sqrt{\xi}$ and we have reached the wanted result up to increasing the constant $c_\nu$, or $\eta \leq \sqrt{\xi}$ and we can bound the left-hand term by any constant $c_\nu \geq 1$.
\end{proof}

\section{The extremal eigenvalues and the localization phenomenon}
\label{sec:biggest-eigenvalues}

We now turn to the extremal eigenvalues. We show that they are close to
the square roots of the degrees of some vertices in $G$. Furthermore,
assuming that the average degrees $d_{x}$ are well separated, we
obtain the localization of some eigenvectors around single vertices.

We use the notation $\lambda_1(A) \geq \lambda_2(A) \geq \cdots \geq \lambda_N(A)$ for the ordered eigenvalues of $A$. Moreover, we recall the permutation $\pi$ from \cref{def:order}.

\begin{theorem}\label{thm:eigenvalues}
  There exists a constant $C_{\nu}> 0$ such that with $\nu$-high
  probability, for all $i \in [N]$, if $\lambda_{i}(A) > C_{\nu}\sqrt{\frac{\log N}{\log\log N}}$, then
  \begin{equation*}
    \left|\lambda_{i}(A) - \sqrt{D_{\pi(i)}}\right| \leq C_{\nu}\sqrt{\frac{\log N}{\log\log N}}\,.
  \end{equation*}
  or if $\lambda_{i}(A) < -C_{\nu}\sqrt{\frac{\log N}{\log\log N}}$, then
  \begin{equation*}
    \left|\lambda_{i}(A) + \sqrt{D_{\pi(N+1 - i)}}\right| \leq C_{\nu}\sqrt{\frac{\log N}{\log\log N}}\,.
  \end{equation*}
\end{theorem}
\begin{proof}
  This is a consequence of \cref{prop:approx-hat}. There
  exists $C'_{\nu} > 0$ such that with $\nu$-high probability
  \begin{equation*}
    |\lambda_{i}(A) - \lambda_{i}(\hat{A})| \leq \|A - \hat{A}\| \leq C'_{\nu}\sqrt{\frac{\log N}{\log\log N}}\,,
  \end{equation*}
  for all $i \in [N]$.

  The $i$-th eigenvalue of $\hat{A}$ satisfies
  $\lambda_{i}(\hat{A}) = \pm\sqrt{D_{\pi(i)}^{\pp -}}$ if
  $\lambda_{i}(\hat{A}) > \|\overline{\Pi}^{\pp}A^{\pp}\overline{\Pi}^{\pp}\|$.
  Furthermore, \cref{prop:pruning} implies that with $\nu$-high probability, for
  all $x \in \hcV^{(\rm h)}$,
  \begin{equation*}
    |D_{x} - D_{x}^{\pp -}| \leq \xi/2\,.
  \end{equation*}
  Hence, there exists a constant $C_{\nu}$ such that with $\nu$-high
  probability: for all $i$ such that
  $\lambda_{i}(\hat{A}) > \|\overline{\Pi}^{\pp}A^{\pp}\overline{\Pi}^{\pp}\|$,
  \begin{equation*}
    |\lambda_{i}(A) - \sqrt{D_{x}}| \leq C_{\nu} \sqrt{\frac{\log N}{\log\log N}}\,.
  \end{equation*}
  We have a similar result when
  $\lambda_{i}(\hat{A}) < - \|\overline{\Pi}^{\pp}A^{\pp}\overline{\Pi}^{\pp}\|$.
\end{proof}

We now consider the phenomenon of localization around a single vertex,
in general a stronger result than the semilocalization. According to
\cref{thm:main-intro}, it occurs when
$\#\mathcal{W}_{\lambda, \eta} = 1$ for an appropriate pair
$(\lambda, \eta)$. We fix $\nu > 0$ and $\eta > 0$.

We introduce the set of isolated vertices:
\begin{equation}\label{eq:def-set-loc}
  \hcV^{*}_{\nu, \eta} = \left\{ x \in \vertices \col
    \begin{aligned}
      &\bullet \, \forall y \in \vertices \setminus \{x\} , |d_{x} - d_{y}| \geq \left( 4\sqrt{\nu\log N d_{x}} + 4\sqrt{\nu\log N d_{y}} \right)\vee 16 \eta^{2}\\
      &\bullet \, d_{x} \geq \frac{4\nu}{9}\log N\\
    \end{aligned}
  \right\}\,.
\end{equation}
Note that while the definition of the set $\hcV^{*}_{\nu, \eta}$ is
somewhat involved, it only depends on the sequence of weights. We
shall show that the eigenvectors associated to the vertices in
$\hcV^{*}_{\nu, \eta}$ are localized with $\nu$-high probability.
\begin{theorem}[Localization]\label{thm:localization}
  There exists $C_{\nu} > 0$ such that with $\nu$-high
  probability, for any eigenvalue $\lambda > C_{\nu}\sqrt{\log N}$ of $A$, with associated
  eigenvector $\b q$, and all $\eta \leq \lambda/2$, we have the following property.

  If $\mathcal{W}_{\lambda, \eta}\cap \hcV^{*}_{\nu, \eta} \neq \emptyset$ then there exists
  $x \in \hcV^{*}_{\nu, \eta}$ such that
  \begin{equation*}
    \langle \b q, \b u_{+}(x) \rangle^{2} \geq 1 - \Biggl(\frac{C_\nu}{\eta}\sqrt{\frac{\log N}{\log\log N}}\Biggr)^2\,.
  \end{equation*}
\end{theorem}
\begin{proof}
  We shall show that if
  $\mathcal{W}_{\lambda, \eta}\cap \hcV^{*}_{\nu, \eta} \neq \emptyset$,
  then $\#\mathcal{W}_{\lambda, \eta} = 1$.
  \cref{thm:localization} is then a consequence of \cref{thm:main-intro}.

  Let $x \in \mathcal{W}_{\lambda, \eta}\cap \hcV^{*}_{\nu, \eta}$ and
  $y \in \mathcal{W}_{\lambda, \eta}$. We first notice that assuming
  $\lambda > C_{\nu}\log N$, we get thanks to \cref{lem:estimate-degree} that
  \begin{equation*}
    \frac{C_{\nu}}{2}\log N \ \frac{\lambda}{2} \leq D_{y} \leq d_{y} + 2 \sqrt{\nu \log N (d_{y}\vee \frac{4\nu}{9}\log N)}\,.
  \end{equation*}
  By choosing $C_{\nu} > 0$ big enough, we can ensure
  $d_{y} \geq \frac{4\nu}{9}\log N$. A second use of
  \cref{lem:estimate-degree} implies that with $\nu$-high probability,
  \begin{equation*}
    d_{x}-d_{y}-2\sqrt{\nu\log N d_{x}}-2\sqrt{\nu\log N d_{y}} \leq D_{x} - D_{y}\leq d_{x}-d_{y}+2\sqrt{\nu\log N d_{x}}+2\sqrt{\nu\log N d_{y}}\,.
  \end{equation*}
  Using \eqref{eq:def-set-loc}, we have
  \begin{equation*}
    d_{x}-d_{y} - \frac{1}{2}|d_{x} - d_{y}| \leq D_{x} - D_{y} \leq d_{x}-d_{y} + \frac{1}{2}|d_{x} - d_{y}|\,.
  \end{equation*}
  Notice that if $D_{x} > D_{y}$ then $d_{x} \geq d_{y}$ as otherwise we would have
  \begin{equation*}
    D_{x} - D_{y} \leq - \frac{1}{2}|d_{x} - d_{y}| \leq 0\,,
  \end{equation*}
  a contradiction. We can show similarly that if $D_{x} < D_{y}$ then
  $d_{x}\leq d_{y}$. Thus, we have
  \begin{equation}\label{eq:lowerbd-D}
    |D_{x} - D_{y}| \geq \frac{1}{2}|d_{x} - d_{y}|\,.
  \end{equation}

  Finally, if $x \neq y$, we have by \eqref{eq:lowerbd-D} and \eqref{eq:def-set-loc}:
  \begin{equation*}
    \Bigl|\sqrt{D_{x}} - \sqrt{D_{y}}\Bigr| \geq \sqrt{\frac{|D_{x} - D_{y}|}{2}}\geq \frac{\sqrt{|d_{x} - d_{y}|}}{2} \geq 2\eta\,.
  \end{equation*}
  Thus, either $x = y$ or $y \notin \mathcal{W}_{\lambda, \eta}$. Hence, $\mathcal{W}_{\lambda, \eta}$ is a singleton.
\end{proof}

\begin{example} \label{ex:localization_power}
  Consider \cref{ex:heavy-tail}. The weights are chosen as the
  quantiles of a power law. Let $\alpha > 2$. We choose for simplicity
  \begin{equation*}
    w_{i} = \left( \frac{N}{i} \right)^{1/\alpha}\,.
  \end{equation*}

  In that case, if $i \leq N^{1/(2\alpha + 2)}$, we have
  \begin{equation*}
    |w_{i} - w_{i-1}| \geq \left( \frac{N}{i} \right)^{1/\alpha}\frac{c}{i}\,,
  \end{equation*}
  for some constant $c > 0$. As the sequence $(|w_{i+1} - w_{i}|)_{i \geq 1}$ is
  decreasing, we get that
  $\{1, \ldots, \lfloor N^{1/(2\alpha + 2)}\rfloor\} \subset \hcV^{*}_{\nu, \eta}$ with
  $\eta = \frac{\sqrt{\alpha}}{4}N^{1/(4\alpha)}$. On the other hand,
  \cref{thm:eigenvalues} implies that for $i \leq \lfloor N^{1/(2\alpha+2)}\rfloor$,
  \begin{equation*}
    \left| \lambda_{i}(A)- \sqrt{D_{\pi(i)}} \right| \ll \eta
  \end{equation*}
  with $\nu$-high probability, that is for all such $i$,
  $\pi(i) \in \W_{\lambda_{i}(A), \eta}$. Bennett's inequality shows
  that $D_{i} > D_{i+1}$ with $\nu$-high probability for
  $i \leq \lfloor N^{1/(2\alpha+2)}\rfloor$. Hence, $\pi(i) = i$ for all such $i$'s.

  It then follows from \cref{thm:localization} that with $\nu$-high probability, the
  eigenvectors corresponding to the $N^{1/(2\alpha + 2)}$ first eigenvalues
  are localized.
\end{example}

\appendix

\section{Bounds for the proof of \cref{prop:approx-hat}}
\label{sec:bounds-proof-prop}
The norm of
\begin{equation*}
  B_{2} \deq \sum_{x\in\hcV^{(\rm h)}_{\nu}} \frac{2}{Z_{x}}\sum_{\substack{y \in S_{1}^{\pp -}(\hat{x})\\y \prec x}}\frac{\b 1_{S_{1}^{\pp -}(x)}\b 1_{y}^{*}}{\sqrt{D_{x}^{\pp -}}\#\Sib^{-}(x)}
\end{equation*}
is equal to
\begin{equation*}
  \begin{split}
    \|B_{2}\|^{2} = \max_{\|\b u\| = 1}(\b u^{*}B_{2}B_{2}^{*}\b u)
    &= \max_{\|\b u\| = 1}\sum_{x, x'\in\hcV^{(\rm h)}_{\nu}} \frac{4 \left< \b u, \b 1_{S_{1}^{\pp -}(x)} \right> \left< \b 1_{S_{1}^{\pp -}(x')}, \b u \right>}{Z_{x} Z_{x'}\sqrt{D_{x}^{\pp -}D_{x'}^{\pp -}}\#\Sib^{-}(x)\Sib_{x'}^{-}}\sum_{\substack{y \in S_{1}^{\pp -}(\hat{x})\\y \prec x\\y' \in S_{1}^{\pp -}(\hat{x}')\\y' \prec x'}}\b 1_{y}^{*}\b 1_{y'}\\
    &= \max_{\|\b u\| = 1}\sum_{x, x'\in\hcV^{(\rm h)}_{\nu}} \frac{4 \left< \b u, \b 1_{S_{1}^{\pp -}(x)} \right> \left< \b 1_{S_{1}^{\pp -}(x')}, \b u \right>}{Z_{x} Z_{x'}\sqrt{D_{x}^{\pp -}D_{x'}^{\pp -}}\sqrt{\#\Sib^{-}(x)\Sib_{x'}^{-}}}\delta_{\hat{x} \hat{x}'}\\
    &= \max_{\|\b u\| = 1}\sum_{x\in\hcV^{(\rm h)}_{\nu}} \frac{4 \left<  \b u, \b 1_{S_{1}^{\pp -}(x)}\right> \left< \b 1_{S_{1}^{\pp -}(x)}, \b u \right>}{Z_{x}^{2} D_{x}^{\pp -}\#\Sib^{-}(x)}\,.
  \end{split}
\end{equation*}
At this point, we use the orthonormal family
$\left( \b V_{1}(x) \right)_{x \in \hcV^{(\rm h)}_{\nu}} = \left(\frac{1}{\sqrt{D_{x}^{\pp -}}}\b 1_{S_{1}^{\pp -}(x)}\right)_{x \in \hcV^{(\rm h)}_{\nu}}$.
There exists a constant $C_{\nu} > 0$ such that $\frac{1}{\# \Sib^{-}(x)} \leq \frac{C_{\nu}}{\xi_{\nu}}$ so that
\begin{equation*}
  \|B_{2}\|^{2} \leq \frac{C_{\nu}}{\xi_{\nu}}\max_{\|\b u\| = 1}\sum_{x\in\hcV^{(\rm h)}_{\nu}} \left< \b u, \b V_{1}(x) \right>^{2} \leq \frac{C_{\nu}}{\xi_{\nu}}\,.
\end{equation*}

We now turn to the operator
\begin{equation*}
  B_{3} \deq -\sum_{x\in\hcV^{(\rm h)}_{\nu}} \frac{2}{Z_{x}^{2}}\sum_{y \in \Sib^{-}(x)}\frac{\b 1_{S_{1}^{\pp -}(y)}\b 1_{x}^{*}}{\#\Sib^{-}(x)}\,,
\end{equation*}
whose norm is
\begin{equation*}
  \|B_{3}\|^{2} = \max_{\|\b u\|= 1}\b u^{*}B_{3}^{*}B_{3}\b u = \sum_{x, x'\in\hcV^{(\rm h)}_{\nu}} \frac{4u_{x}u_{x'}}{Z_{x} Z_{x'}\#\Sib^{-}(x)\Sib_{x'}^{-}}\sum_{\substack{y \in \Sib^{-}(x)\\ y' \in \Sib^{-}(x')}} \left< \b 1_{S_{1}^{\pp -}(y)}, \b 1_{S_{1}^{\pp -}(y')} \right>\,.
\end{equation*}
The orthogonality of the vectors
$(\b 1_{S_{1}^{\pp -}(y)})_{y \in \vertices} = (\sqrt{D_{y}^{\pp -}}V_{1}(y))_{y \in \vertices}$ yields
\begin{equation*}
  \|B_{3}\|^{2} = \sum_{x, x'\in\hcV^{(\rm h)}_{\nu}} \frac{4\delta_{\hat{x} \hat{x}'}u_{x}u_{x'}}{Z_{x} Z_{x'}\#\Sib^{-}(x)\# \Sib^{-}(x')}\sum_{y \in \Sib^{-}(x) \cap \Sib^{-}(x')}D_{y}\,.
\end{equation*}
We apply Young's inequality to replace
$\frac{u_{x}u_{x'}}{\#\Sib^{-}(x)\#\Sib^{-}(x')}$ by
$\frac{u_{x}^{2}}{(\#\Sib^{-}(x))^{2}}$:
\begin{equation*}
  \|B_{3}\|^{2} \leq \sum_{x, x'\in\hcV^{(\rm h)}_{\nu}} \frac{4\delta_{\hat{x} \hat{x}'}u_{x}^{2}}{Z_{x} Z_{x'} \left( \#\Sib^{-}(x) \right)^{2}}\sum_{y \in \Sib^{-}(x) \cap \Sib^{-}(x')}D_{y} \leq \sum_{x\in\hcV^{(\rm h)}_{\nu}} \frac{4u_{x}^{2}D_{\hat{x}}}{Z_{x} \left( \#\Sib^{-}(x) \right)^{2}}\sum_{y \in \Sib^{-}(x)}D^{\pp}_{y}\,.
\end{equation*}
Lemmas \ref{lem:bound-siblings} and \ref{lem:descending-ball} give
that with $\nu$-high probability
\begin{equation*}
  \|B_{3}\|^{2} \leq 8\nu\sum_{x, x'\in\hcV^{(\rm h)}_{\nu}} \frac{4u_{x}^{2}}{Z_{x}} \frac{\log N}{\log\log N} = \order\pbb{\frac{\log N}{\log\log N}}\,.
\end{equation*}

Finally, consider the operator
\begin{equation*}
  B_{4} \deq \sum_{x\in\hcV^{(\rm h)}_{\nu}} \frac{2}{Z_{x}}\sum_{y, z \in \Sib^{-}(x)}\frac{\b 1_{S_{1}^{\pp -}(y)}\b 1_{z}^{*}}{(\#\Sib^{-}(x))^{2}}\,,
\end{equation*}
and define for $x \in \hcV^{(\rm h)}_{\nu}$,
\begin{equation*}
  \begin{split}
    \b u(x) &= \frac{1}{\#\Sib^{-}(x)}\b 1_{\Sib^{-}(x)}\\
    \b v(x) &= \frac{1}{\#\Sib^{-}(x)}\sum_{y \in \Sib^{-}(x)} \b 1_{S_{1}^{\pp -}(y)}\,.
  \end{split}
\end{equation*}

Notice that
\begin{equation*}
  B_{4}\b u(x) = \frac{2}{Z_{x}(\#\Sib^{-}(x))^{3}}\sum_{y, z \in \Sib^{-}(x)}\b 1_{S_{1}^{\pp -}(y)} = \frac{2}{Z_{x}(\#\Sib^{-}(x))^{2}}\sum_{y \in \Sib^{-}(x)}\b 1_{S_{1}^{\pp -}(y)} = \frac{2}{Z_{x}\#\Sib^{-}(x)}\b v(x)\,.
\end{equation*}
Furthermore,
\begin{equation*}
\b  v(x)^{*}B_{4} = \frac{2}{Z_{x}(\#\Sib^{-}(x))^{3}}\sum_{y, z \in \Sib^{-}(x)}D_{y}^{\pp -}\b 1_{z}^{*} = \left(  \frac{2}{Z_{x}(\#\Sib^{-}(x))^{2}}\sum_{y \in \Sib^{-}(x)}D_{y}^{\pp -}\right) \b u(x)^{*}\,.
\end{equation*}
Thus, Lemmas \ref{lem:descending-ball} and \ref{lem:bound-siblings}, and the Schur test
imply that
\begin{equation*}
  \|B_{4}\| = \order\p{1}\,.
\end{equation*}

\section{Estimation of the size of $\W_{\lambda, \eta}$}
\label{sec:estim-size-w_lambda}

Let $\lambda, \eta > 0$ such that $2\sqrt{\frac{\log N}{\log\log N}} \leq \eta \leq \lambda/2$. We
consider the expectation $\E \left[ \# \W_{\lambda, \eta} \right]$,
which we rewrite as
\begin{equation*}
  \E \left[ \# \W_{\lambda, \eta} \right]
  = \sum_{\substack{x \in \vertices\\ w_{x} \leq \sqrt{\log N}}}\P \left( (\lambda - \eta)^{2} \leq D_{x} \leq (\lambda + \eta)^{2} \right) + \sum_{\substack{x \in \vertices\\ w_{x} > \sqrt{\log N}}}\P \left( (\lambda - \eta)^{2} \leq D_{x} \leq (\lambda + \eta)^{2} \right)\,.
\end{equation*}
The first part can be bounded using Bennett's inequality \cite[Theorem 2.9]{BLM13}:
\begin{equation*}
  \P \left( D_{x} \geq (\lambda - \eta)^{2} \right) \leq \exp \left( - (\lambda - \eta)^{2}\log(\frac{\left( (\lambda - \eta)^{2} \right)}{d_{x}}) - \left( \lambda - \eta \right)^{2} + d_{x}\right) = \order\p{N^{-2}}\,,
\end{equation*}
so that
\begin{equation*}
  \E \left[ \# \W_{\lambda, \eta} \right]
  = \sum_{\substack{x \in \vertices\\ w_{x} > \sqrt{\log N}}}\P \left( (\lambda - \eta)^{2} \leq D_{x} \leq (\lambda + \eta)^{2} \right) + \order\p{N^{-1}}\,.
\end{equation*}

We start be recalling the following Lemma of approximation of the
degrees by Poisson variables.
\begin{lemma}[Approximation of degrees by a Poisson variable {\cite[Theorem 6.7]{hofstad_random_2016}}]
  There exists a coupling $(\hat{D}_{x}, \hat{P}_{x})$ of the degree
  $D_{x}$ of vertex $x$ and a Poisson variable $x$ with parameter
  $w_{x}$, such that
  \begin{equation*}
    \P \left( \hat{D}_{x} \neq \hat{P}_{x} \right) \leq \frac{w_{x}^{2}}{m_{1}N} \left( 1 + 2\frac{m_{2}}{m_{1}} \right)\,.
  \end{equation*}
\end{lemma}

This Lemma will be key in estimating $\E \left[ \# \W_{\lambda, \eta} \right]$
for some $\lambda, \eta > 0$. Indeed, we have
\begin{equation*}
  \begin{split}
    \E \left[ \# \W_{\lambda, \eta} \right]
    &= \sum_{\substack{x \in \vertices\\ w_{x} > \sqrt{\log N}}}\P \left( (\lambda - \eta)^{2} \leq D_{x} \leq (\lambda + \eta)^{2} \right) + \order\p{N^{-1}}\\
    &\leq \sum_{\substack{x \in \vertices\\ w_{x} > \sqrt{\log N}}}\P \left( (\lambda - \eta)^{2} \leq \hat{P}_{x} \leq (\lambda + \eta)^{2} \right) + \frac{m_{2}}{m_{1}} \left( 1 + 2 \frac{m_{2}}{m_{1}} \right) + \order\p{N^{-1}}\,.
  \end{split}
\end{equation*}
By \cref{hyp:behavior-m21}, the last term is of order at most $(\log N)^{2/3}$.
The first term can be written in term of incomplete Gamma functions
\begin{equation*}
  \Gamma(s, x) = \int_{x}^{\infty}t^{s-1}\ee^{-t}\dd t\,.
\end{equation*}
Indeed, we have
\begin{equation*}
  \sum_{\substack{x \in \vertices\\ w_{x} > \sqrt{\log N}}}\P \left( (\lambda - \eta)^{2} \leq \hat{P}_{x} \leq (\lambda + \eta)^{2} \right)
  = \sum_{\substack{x \in \vertices\\ w_{x} > \sqrt{\log N}}} \left( \frac{\Gamma(U_{N}, w_{x})}{\Gamma(U_{N})} - \frac{\Gamma(L_{N} - 1, w_{x})}{\Gamma(L_{N}- 1)} \right)\,,
\end{equation*}
where we set for convenience $L_{N} = \lfloor(\lambda - \eta)^{2}\rfloor$ and $U_{N} = \lceil (\lambda + \eta)^{2} \rceil$.

We shall use the two following properties of incomplete Gamma
function.
\begin{lemma}\label{lem:prop-incomplete-gamma}
  Let $s \geq 1$ and $x > 0$. Then, $\Gamma(s) - x^{s-1} \leq \Gamma(s, x) \leq \Gamma(s)$. Furthermore,
  $\Gamma(s, x) \sim x^{s-1}\ee^{-x}$ as $x \to \infty$.
\end{lemma}
\begin{proof}
  It is immediate that $\Gamma(s, x) \leq \Gamma(s)$. For the other bound, we have
  \begin{equation*}
    \Gamma(s, x) = \Gamma(s) - \int_{0}^{x}t^{s-1}\ee^{-t}\dd t \geq \Gamma(s) - x^{s-1}\int_{0}^{x}\ee^{-t}\dd t \leq \Gamma(s) - x^{s-1}\,.
  \end{equation*}

  To prove the asymptotic estimate, we remark that
  \begin{equation*}
    \frac{\Gamma(s, x)}{x^{s-1}\ee^{-x}}
    = \int_{x}^{\infty} \left( \frac{t}{x} \right)^{s-1}\ee^{-(t-x)}\dd t
    = \int_{0}^{\infty} \left( \frac{t}{x} +1\right)^{s-1}\ee^{-t}\dd t\,.
  \end{equation*}
  The dominated convergence theorem then implies that the limit of the left-hand term is $1$ as $x \to \infty$.
\end{proof}

We kept in the sum only terms $x$ such that $w_{x} \to \infty$ as $N \to \infty$, hence by \cref{lem:prop-incomplete-gamma}, we have
\begin{equation*}
  \E \left[ \# \W_{\lambda, \eta} \right]
  \leq \sum_{\substack{x \in \vertices\\w_{x} > \sqrt{\log N}}} \ee^{-w_{x}}\left( \frac{w_{x}^{U_{N}-1}}{(U_{N}-1)!} - \frac{w_{x}^{L_{N} -2}}{(L_{N}-2)!} \right)(1+o(1)) + \order\p{(\log N)^{2\delta}}\,.
\end{equation*}

We now consider two cases:
\begin{itemize}
  \item The weights $(w_{x})$ are the $(N+1)$-quantiles of an
        exponential law as in \cref{ex:exponential}.
  \item The weights $(w_{x})$ are the $(N+1)$-quantiles of a power law as in \cref{ex:heavy-tail}.
\end{itemize}

In the exponential case, we then have
\begin{equation*}
  \E \left[ \# \W_{\lambda, \eta} \right] \leq N \int_{0}^{\infty}\alpha\ee^{-\alpha t-t}  \left( \frac{t^{U_{N}-1}}{(U_{N}-1)!} - \frac{t^{L_{N}-2}}{(L_{N}-2)!} \right)\dd t(1 + o(1)) + \order\p{(\log N)^{2\delta}}\,.
\end{equation*}
Using that the $k$-th moment of an exponential law of parameter
$\alpha + 1$ is $k!/(\alpha + 1)^{k}$ we get
\begin{equation*}
  \E \left[ \# \W_{\lambda, \eta} \right] \leq N \frac{\alpha}{\alpha+1}\left( \frac{1}{(\alpha+1) ^{L_{N}-2}} - \frac{1}{(\alpha+1)^{U_{N}-1}} \right)(1 + o(1)) + \order\p{(\log N)^{2\delta}}\,.
\end{equation*}
That is,
\begin{equation}\label{eq:exponential-W}
  \E \left[ \# \W_{\lambda, \eta} \right] \leq N \frac{\alpha}{(\alpha+1)^{\lfloor (\lambda - \eta)^{2}\rfloor -1}}(1 + o(1)) + \order\p{(\log N)^{2\delta}}\,.
\end{equation}

In the power law case, we denote the power law measure by $\mu$ and recall that
\begin{equation*}
  \mu([t, +\infty)) = L(t)t^{-\alpha},
\end{equation*}
where $L$ is a positive, slowly varying function. We notice that
\begin{equation*}
  \int_{0}^{\infty} \frac{t^{U_{N}-1}}{(U_{N}-1)!} \ee^{-t} \dd \mu(t)
  = \int_{0}^{\infty}L(t) t^{-\alpha}\left( \frac{t^{U_{N}-2}}{(U_{N}-2)!} - \frac{t^{U_{N}-1}}{(U_{N}-1)!} \right) \ee^{-t}\dd t\,.
\end{equation*}
Then, for all $\iota > 0$, there exists $C > 0$ such that
$0 \leq L(t) \leq C t^{\iota}$. Note that if $L$ is bounded we can choose
$\iota = 0$. We observe that on $(0, \frac{U_{N}-2}{\ee})$, the
integrand is positive. Hence, we have
\begin{equation*}
  \int_{0}^{\infty}L(t) t^{-\alpha}\left( \frac{t^{U_{N}-2}}{(U_{N}-2)!} - \frac{t^{U_{N}-1}}{(U_{N}-1)!} \right) \ee^{-t}\dd t \leq  C\int_{0}^{+\infty} \left( \frac{t^{U_{N}-\alpha-2+\iota}}{(U_{N}-2)!} - \frac{t^{U_{N}-\alpha-1+\iota}}{(U_{N}-1)!} \right) + \order\pB{\ee^{-\frac{U_{N} - 2}{2}}},
\end{equation*}
where the error $\order\p{\ee^{-\frac{U_{N} - 2}{2}}}$ is a bound for
the integral on $(\frac{U_{N}-2}{\ee}, +\infty)$, with and without
$L$. Finally, we have
\begin{equation*}
  \int_{0}^{\infty} \frac{t^{U_{N}-1}}{(U_{N}-1)!} \ee^{-t} \dd \mu(t) \leq C \left( \frac{\Gamma(U_{N}-\alpha-1+\iota)}{\Gamma(U_{N}-1)} - \frac{\Gamma(U_{N}-\alpha+\iota)}{\Gamma(U_{N})} \right)\,.
\end{equation*}
\begin{lemma}[Stirling's approximation]\label{lem:stirling-first}
  Let $t > 0$. We have as $x \to \infty$
  \begin{equation*}
    \frac{\Gamma(x - t)}{\Gamma(x)} = \frac{1}{x^{t}} \ee^{\frac{-t}{2x} + o(1/x)}\,.
  \end{equation*}
\end{lemma}
\begin{proof}
  Stirling's approximation to the first order,
\begin{equation*}
  \log \Gamma(x) = x\log x - x + \frac{1}{2}\log \frac{2\pi}{x} + \frac{12}{x} + \order\pB{\frac{1}{x^{2}}}\,,
\end{equation*}
gives
\begin{equation*}
  \begin{split}
    \log \frac{\Gamma(x-t)}{\Gamma(x)}
    &= - t\log x - (x-t) \left( \frac{t}{x}  + o\left(\frac{1}{x}\right) \right)  + t +\frac{1}{2}\left(\frac{t}{x} + o\left(\frac{1}{x}\right)\right)  + o\left(\frac{1}{x}\right)\\
    &= - t\log x  +\frac{t}{2x} + o\left(\frac{1}{x}\right)\,. \qedhere
  \end{split}
\end{equation*}
\end{proof}
\cref{lem:stirling-first} immediately gives
\begin{equation*}
  \frac{\Gamma(U_{N}-1-\alpha+\iota)}{\Gamma(U_{N}-1)} = \frac{1}{ \left( U_{N} - 1 \right)^{\alpha - \iota}}\ee^{\frac{t}{2U_{N}} + o(1/U_{N})} = \frac{1}{U_{N}^{\alpha - \iota}} \left( 1 + \frac{\alpha-\iota}{U_{N}} + o(\frac{1}{U_{N}}) \right)\,.
\end{equation*}
Similarly, we have
\begin{equation*}
  \frac{\Gamma(U_{N}-\alpha+\iota)}{\Gamma(U_{N})} = \frac{1}{ U_{N}^{\alpha - \iota}} \left( 1 + o(\frac{1}{U_{N}}) \right)\,.
\end{equation*}
This means that
\begin{equation*}
  \int_{0}^{\infty} \frac{t^{U_{N}-1}}{(U_{N}-1)!} \ee^{-t} \dd \mu(t) = \frac{\alpha - \iota}{U_{N}^{\alpha - \iota + 1}} + o\pB{\frac{1}{U_{N}^{\alpha+1-\iota}}}\,.
\end{equation*}
Similarly,
\begin{equation*}
  \int_{0}^{\infty} \frac{t^{L_{N}-2}}{(L_{N}-2)!} \ee^{-t} \dd \mu(t) =\frac{\alpha - \iota}{L_{N}^{\alpha - \iota + 1}} + o\pB{\frac{1}{L_{N}^{\alpha+1-\iota}}}\,.
\end{equation*}
Putting everything together, we get that there is a constant $C > 0$ such that
\begin{equation*}
  \E \left[ \# \W_{\lambda, \eta} \right] \leq C\left( \frac{\alpha - \iota}{L_{N}^{\alpha - \iota + 1}} -  \frac{\alpha - \iota}{U_{N}^{\alpha - \iota + 1}}\right) \left( 1 + o(\frac{1}{L_{N}}) \right) + \order\pB{(\log N)^{2\delta}}\,.
\end{equation*}
This means
\begin{equation}\label{eq:heavy-W}
 \E \left[ \# \W_{\lambda, \eta} \right] = \order\pB{\frac{N}{\lambda^{2\alpha+3-2\iota}}} + \order\pB{(\log N)^{2\delta}}\,.
\end{equation}

We now turn to the proof of \cref{prop:size-W}. The proof
will use the following variant of \cref{lem:bound-D+}.
\begin{lemma}\label{lem:bound-Dhat+}
  For each vertex $x \in [N]$, define
  \begin{equation*}
    \hat{D}^{+}_{x} = \#\{y \in S_{1}(x) \col w_{y} \geq w_{x}\} = \sum_{y \neq x}\ind_{\{x \sim y, w_{x} \leq w_{y}\}}\,.
  \end{equation*}
  Let $\nu > 0$. With $\nu$-high probability, we have
  \begin{equation*}
    \hat{D}_{x}^{+} \leq \frac{2\nu}{1 - \delta}\frac{\log N}{\log\log N}\,.
  \end{equation*}
\end{lemma}
\begin{proof}
  Let $k \geq 1$ be an integer. The union bound implies
  \begin{equation*}
    \P \left( \hat{D}^{+}_{x} \geq k \right) \leq \frac{1}{k!}\sum_{\substack{x_{1}, \ldots, x_{k}\\\text{distinct}}}\P \left( \forall i \in [k], x \sim x_{i}, w_{x} \leq w_{x_{i}}\right)\,.
  \end{equation*}
  By independence and \eqref{eq:proba-edge}, we have
  \begin{equation*}
    \P \left( \hat{D}^{+}_{x} \geq k \right) \leq \frac{1}{k!}\sum_{\substack{x_{1}, \ldots, x_{k}\\\text{distinct}}}\prod_{i=1}^{k} \left( \frac{w_{x}w_{x_{i}}}{m_{1}N}\ind_{\{w_{x} \leq w_{x_{i}}\}} \right) \leq \frac{1}{k!}\sum_{\substack{x_{1}, \ldots, x_{k}\\\text{distinct}}}\prod_{i=1}^{k} \left( \frac{w_{x_{i}}^{2}}{m_{1}N}\right)\,,
  \end{equation*}
  where in the last line, we used that
  $\ind_{\{w_{x} \leq w_{x_{i}}\}} \leq w_{x_{i}} / w_{x}$. By definition
  of the second empirical moment, we have
  \begin{equation*}
    \P \left( \hat{D}^{+}_{x} \geq k \right) \leq  \frac{1}{k!} \left( \frac{m_{2}}{m_{1}} \right)^{k}(1 + o(1))\,.
  \end{equation*}
  Taking $k = \lfloor \frac{2\nu}{1 - \delta}\frac{\log N}{\log\log N}\rfloor$ allows
  us to conclude.
\end{proof}

\begin{proof}[Proof of \cref{prop:size-W}]
  Let $k \geq 1$ be an integer. We use the union bound to write
  \begin{equation*}
    \P \left( \# \W_{\lambda, \eta} \geq k\right) \leq \frac{1}{k!}\sum_{\substack{x_{1}, \ldots, x_{k}\\\text{distinct}}}\P \left( \forall i \in [k], (\lambda - \eta)^{2} \leq D_{x_{i}} \leq (\lambda+\eta)^{2} \right)\,.
  \end{equation*}
  Set $\hat{D}^{-}_{x} = D_{x} - \hat{D}_{x}^{+}$ for all $x \in [N]$.
  By \cref{lem:bound-Dhat+}, we have for all $x \in [N]$ that
  \begin{equation*}
    D_{x} = \hat{D}_{x}^{+} + \hat{D}_{x}^{-} \leq \hat{D}_{x}^{-} + \frac{2\nu+2}{1-\delta}\frac{\log N}{\log \log N}\,,
  \end{equation*}
  with $\nu$-high probability. For convenience, write
  $c_{\nu, N} = \frac{2\nu+2}{1-\delta}\frac{\log N}{\log \log N}$.

  Now, notice that the random variables $\hat{D}_{x}^{-}$ are
  independent. It implies
  \begin{equation*}
    \begin{split}
      \P \left( \# \W_{\lambda, \eta} \geq k\right)
      &\leq \frac{1}{k!}\sum_{\substack{x_{1}, \ldots, x_{k}\\\text{distinct}}}\prod_{i=1}^{k}\P \left((\lambda - \eta)^{2}-c_{\nu, N} \leq \hat{D}^{-}_{x_{i}} \leq (\lambda+\eta)^{2} \right) + \order\p{N^{-\nu}}\\
      &\leq \frac{1}{k!} \left( \sum_{x}\P \left((\lambda - \eta)^{2}-c_{\nu, N} \leq \hat{D}^{-}_{x} \leq (\lambda+\eta)^{2} \right) \right)^{k} + \order\p{N^{-\nu}}\,.
    \end{split}
  \end{equation*}

Notice that we have 
  \begin{equation*}
    \begin{split}
      \P \left( \# \W_{\lambda, \eta} \geq k\right)
      &\leq \frac{1}{k!} \left( \sum_{x}\P \left((\lambda - \eta)^{2}-c_{\nu, N} \leq D_{x} \leq (\lambda+\eta)^{2}  + c_{\nu, N}\right) \right)^{k} + \order\p{N^{-\nu}}\\
      &\leq \frac{\left( \Ec\#\W_{\lambda, \eta+c_{\nu, N}/2\lambda} \right)^{k}}{k!} + \order\p{N^{-\nu}}\,.
    \end{split}
  \end{equation*}

  Markov's inequality gives the crude bound
  \begin{equation*}
    \P \left( \# \W_{\lambda, \eta} \geq k\right)
    \leq \frac{1}{k!} \left( \sum_{x}\frac{w_{x}^{2}}{\left( (\lambda - \eta)^{2} - c_{\nu, N} \right)^{2}} \right)^{k} = \frac{1}{k!} \left(N\frac{m_{2}}{\left( (\lambda - \eta)^{2} - c_{\nu, N} \right)^{2}} \right)^{k}\,.
  \end{equation*}
Choosing
  $k = \lfloor \frac{2m_{2}}{(\lambda - \eta)^{4}}N \vee \frac{2\nu \log N}{\log\log N}\rfloor$
  gives the result. The result can be improved in our two examples
  using expressions \eqref{eq:exponential-W} and \eqref{eq:heavy-W}, derived above.
\end{proof}

\subsection*{Acknowledgements}
TBdA was supported by ERC Project LDRAM : ERC-2019-ADG Project 884584
in the initial stages of this projet, and by ERC Project InSpeGMos
101096550 in the late stages of this project. AK acknowledges funding from the European Research Council (ERC) and the Swiss State Secretariat for Education, Research and Innovation (SERI) through the consolidator grant ProbQuant, as well as funding from the Swiss National Science Foundation through the NCCR SwissMAP grant.

\bibliographystyle{alpha}
\bibliography{bibliography}

\end{document}